\documentclass[aop,preprint]{imsart}

\RequirePackage[OT1]{fontenc}
\RequirePackage{amsthm,amsmath}
\RequirePackage{amsfonts,extarrows,mathrsfs,xcolor,graphicx,epstopdf,changepage,amssymb,caption,longtable,booktabs}
\RequirePackage[numbers]{natbib}
\RequirePackage[colorlinks,citecolor=blue,urlcolor=blue]{hyperref}

\startlocaldefs
\numberwithin{equation}{section}
\theoremstyle{plain}
\newtheorem{theorem}{Theorem}[section]
\newtheorem{proposition}{Proposition}[section]
\newtheorem{corollary}{Corollary}[section]
\newtheorem{lemma}{Lemma}[section]

\newtheorem{remark}{Remark}
\newtheorem{example}{Example}
\endlocaldefs

\begin{document}

\begin{frontmatter}
\title{On Limiting Behavior of Stationary
Measures for Stochastic Evolution Systems with
Small Noise Intensity\thanksref{T1}}
\runtitle{On Limiting Behavior of Stationary
Measures}
\thankstext{T1}{This work was supported
by the National Natural Science Foundation of China (NSFC)(Nos. 11371252, 11271356, 11371041, 11431014, 11401557),
Research and Innovation Project of Shanghai
Education Committee (No. 14zz120),
Key Laboratory of Random Complex Structures and Data Science,
Academy of Mathematics and Systems Science, CAS, and
the Fundamental Research Funds for the Central Universities (No. WK0010000048).}

\begin{aug}
\author{\fnms{Lifeng} \snm{Chen}\ead[label=e1]{1000360929@smail.shnu.edu.cn}},
\author{\fnms{Zhao} \snm{Dong}\ead[label=e2]{dzhao@amt.ac.cn}},
\author{\fnms{Jifa} \snm{Jiang}\thanksref{t1}\ead[label=e3]{jiangjf@shnu.edu.cn}}
\and
\author{\fnms{Jianliang} \snm{Zhai}\ead[label=e4]{zhaijl@ustc.edu.cn}}

\thankstext{t1}{Corresponding author}

\runauthor{L. Chen, Z. Dong, J. Jiang and J. Zhai}

\address{L. Chen\\
Mathematics and Science College\\
Shanghai Normal University\\
Shanghai 200234\\
People's Republic of China\\
\printead{e1}\\
}
\address{Z. Dong\\
Academy of Mathematics and Systems Science\\
Chinese Academy of Sciences\\
Beijing 100190\\
People's Republic of China\\
\printead{e2}\\
}
\address{J. Jiang\\
Mathematics and Science College\\
Shanghai Normal University\\
Shanghai 200234\\
People's Republic of China\\
\printead{e3}\\
}
\address{J. Zhai\\
Wu Wen-Tsun Key Laboratory of Mathematics\\
University of Science and Technology of China\\
Hefei Anhui 230026\\
People's Republic of China\\
\printead{e4}\\
}
\end{aug}

\begin{abstract}
The limiting behavior of stochastic evolution processes with small noise intensity $\epsilon$ is investigated in distribution-based approach. Let $\mu^{\epsilon}$ be stationary measure for stochastic process $X^{\epsilon}$ with small $\epsilon$ and $X^{0}$ be a semiflow on a Polish space.  Assume that $\{\mu^{\epsilon}: 0<\epsilon\leq\epsilon_0\}$ is tight. Then all their limits in weak sense are $X^0-$invariant and their supports are contained in Birkhoff center of $X^0$.  Applications are made to various stochastic evolution systems, including stochastic ordinary differential equations, stochastic partial differential equations, stochastic functional differential equations driven by Brownian motion or L\'{e}vy process.
\end{abstract}



\end{frontmatter}

\section{Introduction}
Mumford \cite{Mum} addressed that

{\it ``Stochastic differential equations are more fundamental and relevant to modeling the
world than deterministic equations $\cdots$.
A major step in making the equation more relevant is to add a small stochastic term.
Even if the size of the stochastic term goes to 0, its asymptotic effects need
not. It seems fair to say that all differential equations are better models of
the world when a stochastic term is added and that their classical analysis
is useful only if it is stable in an appropriate sense to such perturbations".}
This shows that it is important to check the asymptotic stability of stochastic systems with small noise.
For this purpose, a basic method is to study the stationary measures and their limit measures.
The latters are called  zero-noise limits by Young \cite{Young} and Cowieson and Young \cite{CLai}, where they proved SRB measures can be realized as zero-noise limits. Huang, Ji, Liu and Yi \cite{HJLY,HJLY0} have investigated stochastic ordinary differential equations with small white noise where the drift vector field is dissipative. They have shown that limiting measures are invariant for the flow generated by the drift vector field and their supports are in its global attractor. For non-degenerate noise, Freidlin and Wentzell \cite{FW} estimated the concentration of limiting measure for stationary measures via the large-deviation technique and proved that the stationary measure value $\mu^\epsilon(P)$ tends to zero for any subset $P$ not intersecting with any attractor for the drift vector field, which implies that any limiting measure will support on the global attractor of  the drift vector field;   Li and Yi \cite{LY,YL} have presented more precise estimation for stationary measures near the global attractor or outside of the global attractor via the Fokker-Planck equation and the level set method  developed in \cite{HJLY0}, which are applied by them to study systematic measures of biological network including degeneracy, complexity, and robustness. Hwang \cite{Hwang} proved limiting probability measure of Gibbs measures for gradient system with additive noise concentrates on the minimal energy states. Huang, Ji, Liu and Yi \cite{HJLYd}  have explored the stochastic  stability of invariant sets and measures for gradient systems with noises.

This paper is intended to establish the close connection between deterministic dynamical systems and their stochastic perturbations by considering the limiting behavior of stationary measures for
stochastic evolution systems with small random perturbations.
These stochastic evolution systems $X^{\epsilon}(t,x)$ may be solutions of various stochastic differential equations driven by white or L\'{e}vy noise with the intensity $\epsilon$.
The corresponding solution of deterministic equations is denoted by $X^{0}(t,x)$. Let $\mu^{\epsilon}$ be the stationary probability measure of $X^{\epsilon}(t,x)$. We prove that all their limits of stationary measures $\mu^{\epsilon}$ of $X^{\epsilon}$ are $X^0-$invariant and their supports are contained in the Birkhoff center of $X^0$ as $\epsilon$ tends to zero (see Theorem \ref{mthm}). For various stochastic differential equations with small noise intensity, we prove the probability convergence property and provide the existence of stationary measures and their tightness and applications to all corresponding stochastic systems (see sections 3-5).
Usually, a global attractor for finite dimensional system has positive Lebesgue measure if it is not a globally stable equilibrium, however, the Birkhoff center always has zero Lebesgue measure for dissipative system. Compared to the existing results, which mostly focus on SODEs with non-degenerate noise, ours gives much more precise positions for limiting measures to support. We note that our result is the best if we don't put any restriction to types of noise because we can construct a diffusion term such that a sequence of stationary measures weakly converges to a given invariant measure of the drift vector field (see Proposition \ref{procons} and Remark \ref{Re8}). As far as we know, among all existing examples (see, for example, \cite {FW,HJLY,Hwang}), the limiting measures support at stable orbits, such as, stable equilibrium or closed orbits. A natural question arises : when a dissipative drift vector field has no stable motion in its global attractor, where does any limiting measure support?  Utilizing our result, we construct Bernoulli's lemniscate with non-degenerate noise such that stationary measures weakly converge to a delta measure at a saddle of the drift vector field, however, the global attractor in this case is the closed domain surrounded by the lemniscate of   Bernoulli (see Example 3). In May-Leonard system perturbed by a one dimensional white noise (see Example 4), we have proved that the limiting measures will support at the three saddles when the deterministic May-Leonard system admits a heteroclinic cycle.
Also, from this example, the limiting measures can be distinguished by different initial values because of the
various kinds of asymptotic behavior for deterministic equations. In a word, limiting measures always support at ``most relatively stable positions".

This article is organized as follows. In section 2, we present the framework to study the limiting measures of stationary measures for stochastic evolution processes and their supports.
From sections 3--5, we prove the probability convergence, the existence of stationary measures and their tightness for various stochastic differential equations. Specially,
in section 3, we deal with all these problems of stochastic ordinary differential equations (SODEs).
In section 4, we investigate stochastic reaction-diffusion equations, stochastic $2D$ Navier-Stokes equations and stochastic Burgers type equations driven by Brownian motions or  L\'{e}vy process.
In sections 5, we consider a class of stochastic functional differential equations (SFDEs). Section 6 collects the basic properties on invariant measures of deterministic flow.

Here and throughout of this article, we will use the same symbol $|\cdot|$ to
denote Euclidean norm of a vector or the operator norm of a matrix. Sometimes we will write $X^{\epsilon}(t,x)$, $X^{0}(t,x)$ as
$X^{\epsilon}_{t}(x)$, $X^{0}_{t}(x)$, respectively, unless noted otherwise.

\section{General framework to study limiting measures}

In this section, we will give general criterion on studying limiting measures of stationary measures for stochastic
evolution processes and describe their concentration.

Let $(\Omega,\mathcal{F},\mathbb{P})$ be a probability space,
$(M,\rho)$ be a Polish space and $\mathcal{B}(M)$ be the Borel $\sigma-$algebra on $M$.
Assume that $\Phi_{t}(x):=X^{0}_{t}(x)$ is a deterministic semi-dynamical system (semiflow) on $(M,\rho)$
and for $0<\epsilon\ll 1$, $X^{\epsilon}_t(x)$ is a noise driven process on $(M,\rho)$ with noise intensity $\epsilon$.

Throughout this article we assume that $\Phi:\mathbb{R}_{+}\times M\longrightarrow M$ is a mapping
with the following properties

(i) $\Phi_{\cdot}(x)$ is continuous, for all $x\in M$,

(ii) $\Phi_{t}(\cdot)$ is Borel measurable, for all $t\in\mathbb{R}_{+}$,

(iii) $\Phi_{0}={\rm id}$, $\Phi_{t}\circ\Phi_{s}(x)=\Phi_{t+s}(x)$,
for all $t,s\in \mathbb{R}_{+}$, $x\in M$. Here $\circ$ denotes composition of mappings.

Let $\{X^\epsilon_t(x),t\geq 0\}$ be a family of processes with initial value $x$ on state space $M$, $\epsilon\in(0,1]$.
The {\it probability transition function} is defined as
$$P^{\epsilon}_t(x,A):=\mathbb{P}(X^{\epsilon}_t(x)\in A), t\geq 0, x\in M, A\in \mathcal{B}(M).$$
A probability measure $\mu^{\epsilon}$ on $\mathcal{B}(M)$ is called {\it stationary} (or {\it invariant}) with respect
to $\{P^\epsilon_t\}_{ t\geq 0}$  if
$$P^{\epsilon}_t\mu^{\epsilon}=\mu^{\epsilon}\ \textrm{for\ any} \:\: t\geq 0.$$

Let $\mathscr{I}^{\epsilon}$ denote the set of all stationary measures
of the process $\{X^{\epsilon}_{t}\}_{t\geq0}.$

For our purpose, a necessary condition is $X^{\epsilon}_T(x)\xlongrightarrow{\mathbb{P}}X^{0}_T(x)$ as
$\epsilon\rightarrow 0$. For a technical reason, we impose the following Hypothesis.

{\bf Hypothesis (Probability Convergence)}: For any given compact set $K\subset M$, $T>0$ and $\delta>0$,
\begin{equation}\label{pc}
\displaystyle\lim_{\epsilon\to 0}\sup_{x\in K}
\mathbb{P}\{\displaystyle
\rho\big(X^{\epsilon}(T,x),\Phi(T,x)\big)\geq \delta\}=0.
\end{equation}

\begin{theorem}
\label{mthm}Assume hypothesis {\rm(\ref{pc})} holds. If $\mu^{\epsilon_i}\in \mathscr{I}^{\epsilon_i}$, and $\mu^{\epsilon_i} \overset{w}{\rightarrow}
\mu $ as $\epsilon_i \rightarrow 0$, then $\mu$ is an invariant
measure of $\Phi$, i.e. $\mu\circ\Phi_{t}^{-1}=\mu$ for every $t\geq 0$.
Moreover, this invariant probability measure $\mu$ is concentrated on $B(\Phi)$, where $B(\Phi):
=\overline{\{x\in M:x\in\omega(x)\}}$ denotes the Birkhoff center of $\Phi$
(see the definition in Appendix).
\end{theorem} \begin{proof}
Let $\mu^{\epsilon_i} \overset{w}{\rightarrow}
\mu $ as $\epsilon_i \rightarrow 0$.
It suffices to prove that for any nonzero $g\in C_b(M)$ and $T>0$,
  $$\int g(x)\mu \circ\Phi_T^{-1}(dx)=\int g(x)\mu(dx),$$
or equivalently,
  $$\int g\big(\Phi(T,x)\big)\mu(dx)=\int g(x)\mu(dx).$$
Since $\{\mu^{\epsilon_i}\}$ is relatively compact, it is tight. For every $\eta>0$,
there exists a compact set
$K\subset M$ such that $\displaystyle\inf_{\epsilon_i}\mu^{\epsilon_i}(K)\geq
1-\frac{\eta}{\|g\|}$.
\begin{align*}
&|\int g(x)\mu^{\epsilon_i}\circ\Phi(T,\cdot)^{-1}(dx)-\int g(x)\mu^{\epsilon_i}(dx)|\\
=&|\int \mathbb{E}g\big(\Phi(T,x)\big)\mu^{\epsilon_i}(dx)-\int \mathbb{E}g\big(X^{\epsilon_i}(T,x)\big)\mu^{\epsilon_i}(dx)|\\
\leq&\int \mathbb{E}|g\big(\Phi(T,x)\big)-g\big(X^{\epsilon_i}(T,x)\big)|\mu^{\epsilon_i}(dx)\\
\leq&\int\mathbb{E}|I_{K}(x)[g(\Phi(T,x))-g(X^{\epsilon_i}(T,x))]|\mu^{\epsilon_i}(dx)+
2\eta.
\end{align*}
$\widetilde{K}:=\Phi(T\times K)\subset M$ is a compact set since $\Phi(T, x)$ is continuous on $x$. We claim that
there exists $\delta>0$ such that $\forall y, z \in M$ with $z\in \widetilde{K}$
and $\rho(y,z)<\delta$, one has
$$ |g(y)-g(z)|<\eta.  $$
If not, then there exist $\eta_{0}>0$ and $y_{n}\in M$ and $z_{n}\in \widetilde{K}$ with $\rho(y_{n},z_{n})<\frac{1}{n}$
such that $|g(y_{n})-g(z_{n})|\geq \eta_{0}$, $n=1,2,\cdots$. The
compactness of $\widetilde{K}$ and $\{z_{n}\}\subset \widetilde{K}$
imply that, without loss of generality,  $z_{n}\rightarrow z_{0}\in\widetilde{K}$
as $n\rightarrow\infty$. Therefore, it follows from $\rho(y_{n},z_{n})<\frac{1}{n}$
that $y_{n}\rightarrow z_{0}$. By the continuity of $g$, letting $n\rightarrow\infty$, we have
$$  0=|g(z_{0})-g(z_{0})|\geq \eta_{0}, $$
a contradiction.

Hence one can derive that
\begin{align*}
  &\int\mathbb{E}|I_{K}(x)[g\big(\Phi(T,x)\big)-g\big(X^{\epsilon_i}(T,x)\big)]|\mu^{\epsilon_i}(dx)\\
     =&\int_K\mathbb{E}|I_{\{\rho(\Phi(T,x),X^{\epsilon_i}(T,x))\geq\delta \}}(\omega)[g\big(\Phi(T,x)\big)-g\big(X^{\epsilon_i}(T,x)\big)]|\mu^{\epsilon_i}(dx)\\
    &+\int_K\mathbb{E}|I_{\{\rho(\Phi(T,x),X^{\epsilon_i}(T,x))<\delta \}}(\omega)[g\big(\Phi(T,x)\big)-g\big(X^{\epsilon_i}(T,x)\big)]|\mu^{\epsilon_i}(dx)\\
    \leq&2\|g\|\displaystyle\sup_{x\in K}
  \mathbb{P}\Big(\displaystyle\ \rho\big(X^{\epsilon_i}(T,x),\Phi(T,x)\big)\geq\delta\Big)+\eta.
\end{align*}
Therefore, by the hypothesis (\ref{pc}), one can show that
\begin{align*}
 &\displaystyle\limsup_{\epsilon_i \rightarrow 0}|\int g(x)\mu^{\epsilon_i}\circ\Phi(T,\cdot)^{-1}(dx)-\int g(x)\mu^{\epsilon_i}(dx)|\\
    \leq&2\|g\|\displaystyle\lim_{\epsilon_i \rightarrow 0}\displaystyle\sup_{x\in K}
  \mathbb{P}\big\{\displaystyle\ \rho\big(X^{\epsilon_i}(T,x),\Phi(T,x)\big)\geq\delta\big\}+\eta+
    2\eta=3\eta .
\end{align*}
Since $\eta>0$ is arbitrary and $\mu^{\epsilon_i}\xlongrightarrow{w}\mu$, hence
$\int g(x)\mu\circ\Phi(T,\cdot)^{-1}(dx)=\int g(x)\mu(dx)$. This shows
that $\mu$ is an invariant probability measure of the semiflow $\Phi$.

It remains to prove that $\mu(B(\Phi))=1$. Indeed, the result of this fact relies on
the following well-known lemma, the Poincar\'{e} recurrence theorem.
\qquad\end{proof}

\begin{lemma}\label{PRe}
The support of semiflow $\Phi$-invariant probability measure $\mu$ is
contained in $B(\Phi)$. Consequently this implies that $\mu(B(\Phi))=1$.
\end{lemma}

The above result is a slightly variant version of the
Poincar\'{e} recurrence theorem (see e.g., Ma\~{n}\'{e} \cite[Theorem 2.3, p.29]{Ma}) to obtain
the concentration of invariant measures. For readers' convenience, we also give a self-contained
proof of Lemma \ref{PRe} which is postponed to Appendix.

\begin{remark} Observing the proof of Theorem {\rm\ref{mthm}}, we only need to prove the
probability convergence property for a compact set $K$ satisfying the definition of tightness.
This remark will be used in {\rm SPDEs} of section {\rm 4}.
\end{remark}

In order to apply  Theorem \ref{mthm} to various stochastic differential equations, the probability convergence (\ref{pc}) and  the existence
of stationary measures for $X^{\epsilon}_t(x)$ and their tightness are needed to be proved. In the rest of this paper, we will check them for various stochastic evolution systems.

\section{ODEs driven by L\'{e}vy noise}
Let $(\Omega,\mathcal{F},\mathbb{P})$ be a probability space equipped with a filtration $\{\mathcal{F}_t, t\geq 0\}$ satisfying the usual conditions,  $W=\{W_{t}, t\geq 0\}$ a $k$-dimensional Wiener process and $N$ a Poisson random measure on $\mathbb{R}_{+}\times (\mathbb{R}^{l}\backslash \{O\})$ with the $\sigma$-finite intensity measure $\nu$ on $\mathbb{R}^{l}\backslash \{O\}$, and denote its associated compensator as $\tilde{N}(dt,dy)=N(dt,dy)-\nu(dy)dt$. Denote by $(L_2(\mathbb{R}^k,\mathbb{R}^m),\|\cdot\|_2)$ the Hilbert space of all Hilbert-Schmidt operators from $\mathbb{R}^k$ to $\mathbb{R}^m$. Actually, $L_2(\mathbb{R}^k,\mathbb{R}^m)$ is $m\times k$ matrices set.

Consider the following SODEs driven by a L\'evy process
\begin{equation}\label{lsode}
\begin{split}
dX^{\epsilon,x}(t)=&b(X^{\epsilon,x}(t))dt+\epsilon\sigma(X^{\epsilon,x}(t))dW_{t}\\
&+\epsilon\int_{|y|_{\mathbb{R}^l}<c}F(X^{\epsilon,x}(t-),y)\tilde{N}(dt,dy)
\end{split}
\end{equation}
with initial condition $X^{\epsilon,x}(0)=x\in \mathbb{R}^m$ and $\epsilon, c>0$. The mappings $b:\mathbb{R}^m\rightarrow \mathbb{R}^m$ and $\sigma: \mathbb{R}^m\rightarrow L_2(\mathbb{R}^{k},\mathbb{R}^m)$ are $\mathcal{B}(\mathbb{R}^m)$ measurable functions,
$F:\mathbb{R}^m\times \mathbb{R}^l\rightarrow\mathbb{R}^m$ is $\mathcal{B}(\mathbb{R}^m)\otimes\mathcal{B}(\mathbb{R}^l)$ measurable function.

$b$, $\sigma$ and $F$ are called to satisfy \emph{local Lipschitz condition}, respectively, if for every integer $n\geq 1$, there is a positive constant $L_1(n)$ such that for all $x,y\in \mathbb{R}^{m}$ with $|x|\leq n$ and $|y|\leq n$,
\begin{equation}\label{bLy}
  |b(x)-b(y)|^{2} \leq L_1(n)|x-y|^{2},
\end{equation}
\begin{equation}\label{dLy}
  \|\sigma(x)-\sigma(y)\|_2^{2} \leq L_1(n)|x-y|^{2},
\end{equation}
\begin{equation}\label{FLy}
  \int_{_{\|z\|_{\mathbb{R}^l}<c}}|F(x,z)-F(y,z)|^{2}\nu(dz) \leq L_1(n)|x-y|^{2},
\end{equation}
respectively. In addition, we say that $F$ satisfies \emph{local growth condition}, if for every integer $n\geq 1$, there is a positive constant $L_2(n)$
such that for all $|x|\leq n$,
\begin{equation}\label{locGth}
  \int_{_{\|z\|_{\mathbb{R}^l}<c}}|F(x,z)|^{2}\nu(dz) \leq L_2(n)(1+|x|^{2}).
\end{equation}

If $L_i(n)$, $i=1,2$ are independent of $n$, we say that the coefficient functions admit \emph{global Lipschitz}
and \emph{linear growth conditions}.

For a $C^2$ scalar function $V$, and $\epsilon\geq 0$, we define
  \begin{align*}
    \mathcal{L}^{\epsilon}V(x):=&
  \langle \nabla V(x),b(x)\rangle+
  \frac{\epsilon^{2}}{2}\displaystyle\sum_{i,j=1}^{m}a_{ij}(x)\frac{\partial^{2}V(x)}{\partial x_{i}\partial x_{j}}\\
  &+
  \int_{|y|_{\mathbb{R}^l}<c}\big(V(x+\epsilon F(x,y))-V(x)-\langle \nabla V(x),\epsilon F(x,y)\rangle\big)\nu(dy),
   \end{align*}
where $A(x)=(a_{ij}(x)):= \sigma(x)\sigma^T(x)$ is the diffusion matrix. Let $\mathcal{S}^{\epsilon}$ denote the set of all stationary measures of {\rm (\ref{lsode})} for a given $\epsilon$. The following is
the main result of this section.

\begin{theorem}[Support on Limiting Measures]\label{supportode}
Let $b(x)$, $\sigma(x)$ and $F(x,y)$ in {\rm (\ref{lsode})} be locally Lipschitz continuous and locally linear growth,
and $F(x,y)$ locally bounded with respect to $(x,y)$.
Suppose that there exists a nonnegative function $V(x)\in C^{2}(\mathbb{R}^{m})$ such that
\begin{equation}\label{Lyunbd}
  \displaystyle\inf_{|x|>R}V(x)\rightarrow +\infty, \quad \textrm{as} \ R\rightarrow\infty,\ {\rm and}
\end{equation}
\begin{equation}\label{Lproper}
  \displaystyle\sup_{|x|>R}\mathcal{L}^{\epsilon}V(x)\leq -A_{R}\rightarrow-\infty \:
  \textrm{as}\: R\rightarrow\infty.
\end{equation}
If $\mu^{\epsilon_i}_{x_i}\in \mathcal{S}^{\epsilon_i},$ and $\mu^{\epsilon_i}_{x_i} \overset{w}{\rightarrow}
\mu $ as $\epsilon_i \rightarrow 0$, then $\mu$ is an invariant
measure of $X^{0}(t)$, which supports on the Birkhoff center $B(X^{0})$.
\end{theorem}

The proof of the Theorem \ref{supportode} follows from subsections 3.1 and 3.2.
\subsection{The criterion for probability convergence}
By standard arguments, we have

\begin{lemma}\label{HT}
Suppose that the coefficient functions $b,\sigma$ and $F$ admit global Lipschitz
and linear growth conditions with positive constant $L$.
Then the system {\rm(\ref{lsode})} admits a unique strong solution $X^{\epsilon,x}=\{X^{\epsilon,x}(t):t\geq 0\}$, which is adapted and has c\`{a}dl\`{a}g sample paths. Moreover, for every fixed $T>0$, there is a constant $D_{L,T}$ such that
for each $x\in \mathbb{R}^m$,
\begin{equation}\label{HTE}
\sup_{\epsilon\in(0,1]}\ \sup_{t\in[0,T]}\mathbb{E}|X^{\epsilon,x}(t)|^2 \leq D_{L,T}(1+|x|^2).
\end{equation}
\end{lemma}
%


Denote by $X^{0,x}(t)$ the solution for (\ref{lsode}) as $\epsilon=0$. Then we have

\begin{proposition}\label{procong}
Suppose that the coefficient functions $b,\sigma$ and $F$ admit global Lipschitz
and linear growth conditions with positive constant $L$. Then there exists a
constant $D^*_{L,T}$ such that for every $\epsilon\in(0,1]$
\[ \displaystyle \mathbb{E}[\displaystyle\sup_{0\leq t\leq T}|X^{\epsilon,x}(t)-X^{0,x}(t)|^{2}]\leq
D^*_{L,T}\epsilon^{2}(1+|x|^2) \]
for all $x\in \mathbb{R}^m$.
\end{proposition}

\

\begin{theorem}\label{reofLy}
Let $b$, $\sigma$ and $F$ be locally Lipschitz continuous and locally linear growth.
If there exist a function $V\in C^{2}(\mathbb{R}^{m},\mathbb{R}_{+})$, $\epsilon_{0}>0$ and a constant $c^*<+\infty$, such that {\rm (\ref{Lyunbd})} and
\begin{equation}\label{Lyfc}
  \mathcal{L}^{\epsilon}V(x)\leq c^* V(x),\quad \forall\epsilon\in[0,\epsilon_0]
\end{equation}
hold.
Then there exists a global unique solution $X^{\epsilon,x}(t)$ to {\rm(\ref{lsode})} for all $x\in \mathbb{R}^{m}$
and all $\epsilon\in[0,\epsilon_0]$.
Moreover the hypothesis {\rm(\ref{pc})} holds,
that is, for any given compact set $K\subset \mathbb{R}^{m}$, $T>0$ and $\delta>0$,
$$\displaystyle\lim_{\epsilon\to 0}\sup_{x\in K}
\mathbb{P}\{|X^{\epsilon,x}(T)-X^{0,x}(T)|\geq \delta\}=0.$$
\end{theorem}

\begin{proof} For the global existence and uniqueness of solution to (\ref{lsode}) we refer to a similar proof in Khasminskii
\cite[Theorem 1.1.3 and Theorem 3.3.5]{Khas}, for instance. Without loss of generality, we assume $c^*>0$. Let $\tau_{n}^{\epsilon,x}=\inf\{t:|X^{\epsilon,x}(t)|>n\}$ and $\tau_{n}^{0,x}=\inf\{t:|X^{0,x}(t)|>n\}$. It is easy to see that  $\tau_{n}^{\epsilon,x}$ and $\tau_{n}^{0,x}$
nondecreasingly tend to infinity as $n\rightarrow \infty$.

For each $n\in\mathbb{N}^{*}$, let $S^{n}(r)$ be a nonincreasing $C^{\infty}$ function with values in $[0,1]$
such that
\[  S^{n}(r)=
\begin{cases}
1 & \textrm{if} \ r\in[0, n] ,\\
\frac{n+\frac{1}{2}}{r} & \textrm{if} \ r\in[n+1,+\infty).
\end{cases} \]
Construct functions
\begin{eqnarray}
b_{n}(x)=b\big(xS^{n}(|x|)\big),
\end{eqnarray}
$\sigma_{n}(x)$ and $F_{n}(x,y)$ similarly. Then $b_{n}(x)$, $\sigma_{n}(x)$ and $F_{n}(x,y)$ clearly satisfy global Lipschitz
and linear growth conditions. Let $X_{n}^{\epsilon,x}(t)$ be the solution associated with
functions $b_{n}(x)$, $\sigma_{n}(x)$ and $F_{n}(x,y)$. It is easy to see that $X^{\epsilon,x}(t)=X_{n}^{\epsilon,x}(t)$ for $t\leq\tau_{n}^{\epsilon,x}$. Repeating the proof in  Khasminskii
\cite[Theorem 3.5, p.76]{Khas},  we know that
$$\mathbb{P}(\tau_{n}^{\epsilon,x}\leq T)\leq \frac{\exp(c^*T)V(x)}{\inf_{|y|>n}V(y)},$$
which implies that $\mathbb{P}(\tau_{n}^{\epsilon,x}\leq T)\rightarrow 0$ as $n\rightarrow\infty$ uniformly for $x\in K$.
This shows that $\forall\eta>0$, $\exists N_{0}\in\mathbb{N}^{*}$, such that $\forall n\geq N_{0}$, we have
$\displaystyle\sup_{x\in K}\mathbb{P}(\tau_{n}^{\epsilon,x}\leq T)<\eta$.
The compactness of $K$ and continuity for solution $X^{0,x}(t)$ with respect to initial point ensure that there exists $N_{1}\in \mathbb{N}^{*}$, such that for all $n\geq N_{1}$,
$$  \displaystyle\inf_{x\in K}\tau_{n}^{0,x}>T.$$
Now choosing $n\geq N_{0}\vee N_{1}$, we have
\begin{align*}
&\displaystyle\sup_{x\in K}
\mathbb{P}\{|X^{\epsilon,x}(T)-X^{0,x}(T)|\geq \delta\}\\
\leq &\displaystyle\sup_{x\in K}
\mathbb{P}\{|X^{\epsilon,x}(T)-X^{0,x}(T)|\geq \delta,T<\tau_{n}^{\epsilon,x}\wedge \tau_{n}^{0,x}\}\\
&+\displaystyle\sup_{x\in K}\mathbb{P}(\tau_{n}^{\epsilon,x}\wedge\tau_{n}^{0,x}\leq T)\\
= &\displaystyle\sup_{x\in K}
\mathbb{P}\{|X_{n}^{\epsilon,x}(T)-X_{n}^{0,x}(T)|\geq \delta,T<\tau_{n}^{\epsilon,x}\wedge \tau_{n}^{0,x}\}\\
&+\displaystyle\sup_{x\in K}\mathbb{P}(\tau_{n}^{\epsilon,x}\leq T)\\
\leq &\frac{1}{\delta^{2}}\displaystyle\sup_{x\in K}\mathbb{E}|X_{n}^{\epsilon,x}(T)-X_{n}^{0,x}(T)|^{2}
+\eta\\
\leq &\displaystyle\sup_{x\in K}\frac{D^{*}_{L_{n},T}(1+|x|^2)}{\delta^{2}}\epsilon^{2}
+\eta\\
\leq &2\eta ,\quad \forall\epsilon\in(0,\epsilon_{0}]
\end{align*}
for some constant $\epsilon_{0}>0$. The proof is complete.
\end{proof}

\subsection {The criteria on the existence of stationary measures and their tightness}

Following the arguments as Khasminskii in \cite{KHAS1,Khas}, we obtain the criterion on the tightness of a family of stationary measures for (\ref{lsode}).

\begin{theorem}[Tightness Criterion]\label{tightcri}
Suppose that
$b(x)$, $\sigma(x)$ and $F(x,y)$ in {\rm(\ref{lsode})} are locally Lipschitz continuous and locally linear growth, and $F(x,y)$
is locally bounded with respect to $(x,y)$, and that there exists a scalar function $V(x)\in C^{2}(\mathbb{R}^{m},\mathbb{R}_+)$ such that {\rm(\ref{Lyunbd})}  and {\rm(\ref{Lproper})} hold.
Then for any $x\in \mathbb{R}^{m}$, there exists at least
a stationary measure $\mu^{\epsilon}_x$ for every $\epsilon$, and
the set $\mathcal{S}:=\bigcup\{\mathcal{S}^{\epsilon}:0<\epsilon\leq\epsilon_0\}$ of stationary measures
is tight.
\end{theorem}

\begin{proof}  For any fixed $x\in\mathbb{R}^{m}$, it follows from  Theorem \ref{reofLy} that the solution $X^{\epsilon}(t,x)$ is  globally defined on $[0, +\infty)$.  For any $n\in\mathbb{N}^{*}$,
we define stopping time $\tau_{n}^{\epsilon}=\inf\{t:|X^{\epsilon}(t,x)|>n\}$. Then It\^{o}'s formula and Doob's optional sampling theorem (see \cite{Apple,KaraShre}) imply that
$$ \mathbb{E}V(X^{\epsilon}(t\wedge\tau_{n}^{\epsilon},x))-V(x)
=\mathbb{E}\int_{0}^{t\wedge\tau_{n}^{\epsilon}}\mathcal{L}^\epsilon V(X^{\epsilon}(s,x))ds.   $$
Since $V(x)\in C^2(\mathbb{R}^m)$ and $F(x,y)$
is locally bounded, applying Taylor expansion and (\ref{locGth}), we obtain
$$
\sup_{\epsilon\in[0,\epsilon_0]}\int_{|y|_{\mathbb{R}^l}<c}\big(V(x+\epsilon F(x,y))-V(x)-\langle \nabla V(x),\epsilon F(x,y)\rangle\big)\nu(dy)
<
\infty.
$$

By $V(x)\in C^2(\mathbb{R}^m)$ again and (\ref{Lproper}), $C:=\displaystyle\sup_{\epsilon\in[0,\epsilon_0]}\sup_{x\in\mathbb{R}^{m}}\mathcal{L}^\epsilon V(x)<\infty$. Hence we have
$$ \mathcal{L}^\epsilon V(X^{\epsilon}(s,x))\leq -I_{\{|X^{\epsilon}(s,x)|>R\}}A_{R}+C, $$
it is easy to get
$$ A_{R}\mathbb{E}\int_{0}^{t\wedge \tau_{n}^{\epsilon}}I_{\{|X^{\epsilon}(s,x)|>R\}}ds
\leq V(x)+Ct,  $$
where we have used the condition (\ref{Lproper}).
Since $t\wedge\tau_{n}^{\epsilon}\rightarrow t$ a.s. as $n\rightarrow\infty$,
letting $n\rightarrow\infty$ and then changing the order of integration in the last inequality, we have for $t>0$,
\begin{equation}\label{consta}
  \frac{1}{t}\int_{0}^{t}P^{\epsilon}(s,x,U_{R}^{c})ds\leq \frac{1}{A_{R}}\big(\frac{V(x)}{t}+C\big),
\end{equation}
where $U_{R}^{c}=\{x\in\mathbb{R}^{m}:|x|>R\}$. This implies that
$$\lim_{R\rightarrow \infty}\liminf_{t\rightarrow \infty} \frac{1}{t}\int_{0}^{t}P^{\epsilon}(s,x,U_{R}^{c})ds = 0.$$
Applying Khasminskii \cite[Theorem 3.1, p.66]{Khas}, there exists at least a stationary measure $\mu_x^{\epsilon}$, which is produced by Krylov-Bogoliubov procedure, that is, $\mu_x^{\epsilon}$ is a weak limit of a subsequence of probability measures on $\mathbb{R}^{m}$ defined by
$$  P^{\epsilon,t}(x,B)=\frac{1}{t}\int_{0}^{t}P^{\epsilon}(s,x,B)ds. $$

Denote by $\mathcal{S}^{\epsilon}_x$ the set of all their weak limits of probability measures $\{P^{\epsilon,t}(x,\cdot):t>0,\}$ and $\mathcal{S}^{\epsilon}:=\bigcup\{\mathcal{S}^{\epsilon}_x:x\in \mathbb{R}^{m}\}$ for $\epsilon \in (0,\epsilon_0]$. Pick any $\mu_x^{\epsilon}\in \mathcal{S}$. Then there are $\epsilon\in(0,\epsilon_{0}]$, $x\in \mathbb{R}^{m}$ and $t_n\rightarrow \infty$ such that $P^{\epsilon,t_{n}}(x,\cdot) \overset{w}{\rightarrow}
\mu_x^{\epsilon}(\cdot) $ as $n \rightarrow \infty$, by the Portmanteau Theorem and (\ref{consta}), we
obtain
$$ \mu^{\epsilon}_x(U_{R}^{c})\leq \displaystyle\liminf_{n\rightarrow\infty}P^{\epsilon,t_{n}}(x,U_{R}^{c})\leq \frac{C}{A_R}.$$
Since $\lim_{R\rightarrow \infty}\frac{C}{A_R}=0$ uniformly in $\epsilon\in(0,\epsilon_0]$ by assumptions,
the set $\mathcal{S}$ of stationary measures is tight. This completes the proof.
\end{proof}

\begin{remark}From our proof,  the conclusions still hold if $C\leq 0$ and  there is a constant $\gamma > 0$ such that
$A_R \geq \gamma >0$ for $R$ sufficiently large.
\end{remark}

\begin{remark} Huang, Ji, Li and Yi {\rm\cite{HJLY0}} gave the estimate of stationary measures in the essential domain of a Lyapunov-like function in case $F\equiv 0$, which provides the criterion for the tightness of stationary measures.
\end{remark}

Our results allow the nonlinear terms in (\ref{lsode}) to be polynomial growth, which is stated as the following corollary.
\begin{corollary}\label{lnegc}
Suppose that $b(x)$, $\sigma(x)$ and $F(x,y)$ in {\rm (\ref{lsode})} are locally Lipschitz continuous and locally linear growth, and that $F(x,y)$
is locally bounded with respect to $(x,y)$.
If there are positive
constants $c_1,c_2$ and $q\geq 2$ such that for $|x|$ sufficiently large, one has
\begin{eqnarray*}
 \langle b(x),x\rangle\leq -c_1|x|^q,
\end{eqnarray*}
\begin{eqnarray*}
 \frac{1}{2}\|\sigma(x)\|^2_2+\int_{|y|_{\mathbb{R}^l}<c}|F(x,y))|^2\nu(dy)\leq c_2|x|^{q},
\end{eqnarray*}
then the conclusions of Theorems {\rm\ref{reofLy}} and {\rm\ref{tightcri}} hold.
\end{corollary}
\begin{proof}
Define $V:\mathbb{R}^{m}\rightarrow \mathbb{R}_+$ by
$$V(x):=\frac{1}{2}\sum_{i=1}^{m}{(x^i)}^2.$$
Then
\begin{align*}
\mathcal{L}^{\epsilon}V(x)=&
  \langle b(x),x\rangle+
  \frac{\epsilon^{2}}{2}\|\sigma(x)\|^2_2+
  \epsilon^{2}\int_{|y|_{\mathbb{R}^l}<c}|F(x,y))|^2\nu(dy)\\
   &\leq
  -(c_1-\epsilon^{2}c_2)|x|^{q}\\
   &\leq
  -\frac{c_1}{2}|x|^{q}
\end{align*}
for $|x|$ sufficiently large and $\epsilon$ sufficiently small. It is easy to see that (\ref{Lyunbd}) and (\ref {Lproper}) and all other conditions
in Theorems \ref{reofLy} and \ref{tightcri} hold. The proof is complete.
\end{proof}

It is easy to see that Theorem \ref{supportode} follows immediately from Theorems \ref{reofLy},  \ref{tightcri} and \ref{mthm}.

\subsection{On the uniqueness and ergodicity  of stationary measure}

In this subsection, we will provide a result for the uniqueness and ergodicity  of stationary measure for $\{P_t^\epsilon\}_{t\geq 0}$. To achieve this goal, we will give the sufficient conditions for $\{P_t^\epsilon\}_{t\geq 0}$ to be irreducible and strong Feller.

\begin{lemma}\label{SFir}Suppose the assumptions of Lemma {\rm\ref{lsode}} are satisfied.
If the non-degeneracy
\begin{equation}\label{nondjd}
\sup_{x\in \mathbb{R}^m}|\sigma^T(x)\big(\sigma(x)\sigma^T(x)\big)^{-1}|:=\widetilde{K}<\infty \ \ \ \text{holds},
\end{equation}
then the semigroup $P_{t}$ of solution for equations {\rm(\ref{lsode})} is  irreducible.

Furthermore, if the following conditions

\ \ \ ${\rm a)}$ $b,\sigma\in C^1_b(\mathbb{R}^m)$,

\ \ \ ${\rm b)}$ there exists a nonnegative function $\tilde{c}\in L^2(\mathbb{R}^l,\mathcal{B}(\mathbb{R}^l),\nu)$ such that
 $$|F(x,y)|\le \tilde{c}(y),\ \ \ (x,y)\in \mathbb{R}^m\times\mathbb{R}^l,$$

\ \ \ ${\rm c)}$ there exists a constant $C>0$ such that
$$\int_{|y|_{\mathbb{R}^l}<c}\|D_{x}F(0,y)\|_{2}^2\nu(dy)\leq C,\  and$$
\[\int_{|y|_{\mathbb{R}^l}<c}\|D_x F
(x_1,y)-D_xF(x_2,y)\|_{2}^2\nu(dy)\leq C|x_1-x_2|^2, \ x_1,x_2\in \mathbb{R}^m,\]
then the semigroup $P_{t}$ of solution for equations {\rm(\ref{lsode})} is  strong Feller.
\end{lemma}

\begin{proof} For the diffusion case, i.e. $F\equiv 0$, it is well known that the semigroup $P_{t}$ of solution for equation (\ref{lsode})
is strong Feller and irreducible, see, for example, \cite{ZhXC}. We will prove the jump diffusion case. In the below, we denote by $X(t)$ the solution for equation (\ref{lsode}) and $X^d(t)$ for the the diffusion case.

{\bf (1) Irreducible:}

{\bf Step  1.} Suppose $\nu(\{|y|_{\mathbb{R}^l}<c\})<\infty.$

Let $\{\tau_i\}_{i\ge 1}$ be the interarrive times of the Poisson random measure $N$. Then $\{y_{\tau_i}\}_{i\ge 1}$ is a point process associated with the Poisson random measure $N$ which satisfies

(i) $\{y_{\tau_i}\}\subset U:= \{y:|y|_{\mathbb{R}^l}<c\},$

(ii) $\{(\tau_i,y_{\tau_i}),i\ge 1\}$ is independent and for measurable set $\widetilde{O}\subset U(\subset {\mathbb{R}^l}),t>0,$
\[
\mathbb{P}(\tau_i>t,y_{\tau_i}\in \widetilde{O})=e^{-t\nu (U)}\nu (\widetilde{O}).
\]
On $[0,\tau_1), X(t)=X^d(t),$ and $X_{\tau_1}=X_{\tau_1-}+F(X_{\tau_1},y_{\tau_1}).$
Since $\{(\tau_i,y_{\tau_i}),i\ge 1\}$ is independent with the solution $X^d(t)$, as proved in \cite{Dong} and \cite{Dong-Xie2}, we have the relationship for $x\in \mathbb{R}^m, t>0, \Gamma\in \mathcal{B}(\mathbb{R}^m)$,
\begin{align}\label{irjump}
P_t(x,\Gamma)=&e^{-t\nu(U)}P^0_t(x,\Gamma) \nonumber\\
              &+\int_{\mathbb{R}^m}\int^t_0\int_Ue^{-s\nu (U)}P_{t-s}(z+F(z,y),\Gamma)\nu(dy)ds P_t^0(x,dz),
\end{align}
where $P^0_t$ is the semigroup of solution for equation (\ref{lsode}) with $F\equiv 0$, which is irreducible. Therefore,  we have that $P_t$ is irreducible.

{\bf Step 2.} Suppose $\nu(\{|y|_{\mathbb{R}^l}<c\})=\infty.$

The  irreducibility of $X(t)$ can be proved by using the arguments in \cite{Dong-Xie}.

{\bf(2) Strong Feller property:}

Denote by $X(t,x)$ the solution of (\ref{lsode}) ($\epsilon=1$) with initial value  $x$.

For any $\phi\in C^1(\mathbb{R}^m)$, $t\geq0$, and $h\in\mathbb{R}^m$,
\[
D_x{\mathbb E}\phi(X(t,x))h={\mathbb E}[D_x\phi(X(t,x))\eta^{h}_t],
\]
where $\eta^{h}_t=D_x\big(X(t,x)\big)h$ is the solution of the equation
\begin{align}\label{sfjd}
 d \eta^{h}_t=&D_xb(X(t,x))\eta^{h}_tdt+D_x\sigma(X(t,x))\eta^{h}_tdW_t\nonumber \\
&+\int_{|y|_{\mathbb{R}^l}<c}D_xF\big(X(t-,x),y\big)\eta^{h}_t\tilde N(dt,dy), \:\:\eta^{h}_0=h.
\end{align}

From a)--c) and (\ref{nondjd}), by the standard method, we know that there exists some constant $C_T>0$,
independent of $h$ such that
\begin{equation}\label{bddJd}
    {\mathbb E}|\eta^{h}_t|^{2}\le C_T|h|^{2},\quad \textrm{for\ all} \:\:\: t\in[0,T].
\end{equation}

For $\phi\in C^1(\mathbb{R}^m)$, we can prove that $V(t,x)=\mathbb E\phi(X(t,x))$ is a solution of the equation
$$\left\{
\begin{aligned}
\frac{dV(t,x)}{dt}=&{\cal L}V(t,x)\\
 V(0,x)=&\phi(x).
\end{aligned}
\right.$$
(see \cite[Theorem 3.1, p.89]{GaMena}). Using the It\^{o} formula on $V(t-s,x)$ with respect to $s\in[0,t]$ and $x\in\mathbb{R}^m$ of $X(t,x)$
\begin{equation}\label{ItoJdp}
    \begin{split}
    &\phi(X(t,x))\\
    =&V(t,x)+\int^t_0\left[\frac{\partial}{\partial s}V(t-s,X(s,x))+{\cal L}V\big(t-s,X(s,x)\big)\right]ds \\
    &+\int^t_0 D_xV(t-s,X(s,x))\sigma(X(s,x))dW_s  \\
    &+\int^t_0\int_U\big[V\big(t-s,X(s-,x)+F(X(s-,x),y)\big)\\
    &-V(t-s,X(s-,x))\big]\tilde N(ds,dy) \\
    =&V(t,x)+ \int^t_0D_xV(t-s,X(s,x))\sigma(X(s,x))dW_s \\
    &+\int^t_0\int_U\big[V(t-s,X(s-,x)+F(X(s-,x),y))\\
    &-V(t-s,X(s-,x))\big]\tilde N(ds,dy).
    \end{split}
\end{equation}
Multiplying both sides of (\ref{ItoJdp}) by $$\int^t_0\langle\sigma^T(X(s,x))\big(\sigma(X(s,x))\sigma^T(X(s,x))\big)^{-1}\eta^{h}_s,dW_s\rangle_{\mathbb{R}^k}$$ and taking expectation, we get the following Bismut-Elworthy-Li formula for (\ref{lsode}) (see Lemma 7.13 in Da Prato and Zabczyk \cite{DaZaErg}).
Indeed,
\begin{align*}
&\mathbb{E}\big[\phi(X(t,x))
\int^t_0\langle\sigma^T(X(s,x))\big(\sigma(X(s,x))\sigma^T(X(s,x))\big)^{-1}\eta^{h}_s,dW_s\rangle_{\mathbb{R}^k}\big]\\
=&\mathbb{E}\int^t_0\langle\{D_xV(t-s,X(s,x))\sigma(X(s,x))\}^T,
\sigma^T(X(s,x))\big(\sigma(X(s,x))\sigma^T(X(s,x))\big)^{-1}\eta^{h}_s\rangle_{\mathbb{R}^{k}} ds\\
=&\mathbb{E}\int^t_0\langle \big(D_xV(t-s,X(s,x))\big)^T,
\sigma(X(s,x))\sigma^T(X(s,x))\big(\sigma(X(s,x))\sigma^T(X(s,x))\big)^{-1}\eta^{h}_s\rangle_{\mathbb{R}^{m}} ds\\
=&\mathbb{E}\int^t_0 D_xV(t-s,X(s,x))
\eta^{h}_sds\\
=&\mathbb{E}\int^t_0 D_xP_{t-s}\big(\phi(X(s,x))\big)h ds\\
=&\int^t_0 D_x\mathbb{E}P_{t-s}\big(\phi(X(s,x))\big)h ds\\
=&t D_x\mathbb{E}\phi(X(t,x)) h.
\end{align*}
This implies,
\begin{equation}\label{DeJp}
{\small \begin{split}
&D_x\mathbb{E}\phi(X(t,x))h\\
=&\frac{1}{t}\mathbb{E}\Big[\phi(X(t,x))
\int^t_0\langle\sigma^T(X(s,x))\big(\sigma(X(s,x))\sigma^T(X(s,x))\big)^{-1}\eta^{h}_s,dW_s\rangle_{\mathbb{R}^k}\Big].
\end{split}}
\end{equation}
For any $\phi\in C_b^1(\mathbb{R}^m)$, from (\ref{nondjd}), (\ref{bddJd}) and (\ref{DeJp}),
it follows
\begin{align*}
|D_x\mathbb{E}\phi(X(t,x))h|^{2}
\leq&\frac{\|\phi\|^{2}}{t^{2}}\mathbb{E}\Big|
\int^t_0\langle\sigma^T(X(s,x))\big(\sigma(X(s,x))\sigma^T(X(s,x))\big)^{-1}\eta^{h}_s,dW_s\rangle_{\mathbb{R}^k}\Big|^{2}\\
\leq&\frac{\|\phi\|^{2}}{t^{2}}\mathbb{E}
\int^t_0\Big|\sigma^T(X(s,x))\big(\sigma(X(s,x))\sigma^T(X(s,x))\big)^{-1}\eta^{h}_s\Big|^{2}ds\\
\leq&\frac{\|\phi\|^{2}}{t}\widetilde{K}C_T |h|^2.
\end{align*}
Then there exists a
positive constant $\widetilde{C}_T$, we get
\[  |D_x\mathbb{E}\phi(X(t,x))|\leq \frac{\|\phi\|}{\sqrt{t}}\widetilde{C}_T.  \]
Therefore, we have
\[
|P_t\phi(x)-P_t\phi(y)|\le \frac{\widetilde{C}_T}{\sqrt{t}}\|\phi\||x-y|,\quad \textrm{for\ all} \:\:\: t\in(0,T].
\]
By \cite[Lemma 7.15]{DaZaErg}, $P_t$  has strong Feller property.
\end{proof}

\begin{remark}
The \emph{uniformly elliptic} property of diffusion matrix $\sigma\sigma^{T}$,
i.e., there is a constant $\lambda>0$ such that $\xi^{T}\sigma(x)(\sigma(x))^{T}\xi\geq \lambda|\xi|^{2}$
for all $x \in\mathbb{R}^m$ and $\xi\in\mathbb{R}^{m}$, which implies {\rm(\ref{nondjd})}. Indeed,
the boundedness of $\widetilde{\sigma}=\sigma^{T}(\sigma\sigma^{T})^{-1}$ follows from the fact that
\[|\widetilde{\sigma}\xi|^2=\xi^T(\widetilde{\sigma})^T\widetilde{\sigma}\xi=\xi^T(\sigma\sigma^{T})^{-1}\xi\leq\lambda^{-1}|\xi|^2, \quad \forall\xi\in\mathbb{R}^m.    \]
\end{remark}

\begin{theorem}\label{if}
Suppose the assumptions of Theorem {\rm\ref{reofLy}} are satisfied. If the non-degeneracy
\begin{equation}\label{nondjd2}
|\sigma^T(x)\big(\sigma(x)\sigma^T(x)\big)^{-1}|<\infty \ \ \ \text{holds},
\end{equation}
then the semigroup $P_{t}$ is irreducible.

Furthermore, if

\ \ \ ${\rm a)}$ $b,\sigma\in C^1(\mathbb{R}^m)$,

\ \ \ ${\rm b)}$ for any $n>1$, there exists a nonnegative function $c_n\in L^2(\mathbb{R}^l,\mathcal{B}(\mathbb{R}^l),\nu)$ such that
 $$\sup_{|x|\le n}|F(x,y)|\le c_n(y),\ \ \ y\in \mathbb{R}^l,$$

\ \ \ ${\rm c)}$ there exist positive constants $C$ and $C_r$ for any $r>0$ such that
$$\int_{|y|_{\mathbb{R}^l}<c}\|D_{x}F(0,y)\|_{2}^2\nu(dy)\leq C,$$
\[\int_{|y|_{\mathbb{R}^l}<c}\|D_x F
(x_1,y)-D_xF(x_2,y)\|_{2}^2\nu(dy)\leq C_r|x_1-x_2|^2, \ |x_1|\vee|x_2|\leq r,\]
then the semigroup $P_{t}$  has the strong Feller property.
\end{theorem}

\begin{proof} It can be readily checked that for each $n\geq 1$ the coefficients $b_{n},\sigma_n$ and $F_{n}$ as
in the proof of Theorem \ref{reofLy} satisfy the assumptions of Lemma \ref{SFir}. Therefore the transition semigroup $P_{t}^{n}$ corresponding to $X_{n}(t)$
enjoys strong Feller property and irreducibility.

Thus, for any $t>0$ and $f\in \mathcal{B}_{b}(\mathbb{R}^{m})$
\begin{align*}
  &|\mathbb{E}f(X^{x}(t))-\mathbb{E}f(X^{x_{0}}(t))|\\
  =&|\mathbb{E}\big(f(X^{x}(t));t<\tau_{n}^{x}\big)+\mathbb{E}\big(f(X^{x}(t));t\geq\tau_{n}^{x}\big)\\
  &-\mathbb{E}\big(f(X^{x_0}(t));t<\tau_{n}^{x_0}\big)-\mathbb{E}\big(f(X^{x_0}(t));t\geq\tau_{n}^{x_0}\big)|\\
  =&|\mathbb{E}f(X_n^{x}(t))-\mathbb{E}\big(f(X_n^{x}(t));t\geq\tau_{n}^{x}\big)+\mathbb{E}\big(f(X^{x}(t));t\geq\tau_{n}^{x}\big)\\
  &-\mathbb{E}f(X_n^{x_0}(t))-\mathbb{E}\big(f(X^{x_0}(t));t\geq\tau_{n}^{x_0}\big)-
  \mathbb{E}\big(f(X^{x_0}(t));t\geq\tau_{n}^{x_0}\big)|\\
  \leq&|\mathbb{E}f(X_n^{x}(t))-\mathbb{E}f(X_n^{x_0}(t))|\\
  &+|\mathbb{E}\big(f(X_n^{x}(t));t\geq\tau_{n}^{x}\big)|+|\mathbb{E}\big(f(X^{x}(t));t\geq\tau_{n}^{x}\big)|\\
  &+|\mathbb{E}\big(f(X^{x_0}(t));t\geq\tau_{n}^{x_0}\big)|+
  |\mathbb{E}\big(f(X^{x_0}(t));t\geq\tau_{n}^{x_0}\big)|\\
  \leq&|\mathbb{E}f(X_n^{x}(t))-\mathbb{E}f(X_n^{x_0}(t))|+2\|f\|\mathbb{P}(\tau_{n}^{x}\leq t)+2\|f\|\mathbb{P}(\tau_{n}^{x_{0}}\leq t).
\end{align*}
Since $\mathbb{P}(\tau_{n}^{x}\leq T)\rightarrow 0$ as $n\rightarrow\infty$ uniformly for $x$ in compact of $\mathbb{R}^m$.
Thus $\forall \eta>0$, there is a sufficient large $n\in\mathbb{N}^{*}$ such that
$$  \mathbb{P}(\tau_{n}^{x}\leq T)\leq \eta $$
for all $x\in B_{\frac{1}{2}}(x_{0})$. Noting that $P^{n}_{t}$ is strong Feller, this implies
$$  \displaystyle \lim_{x\rightarrow x_0}|\mathbb{E}f(X_n^{x}(t))-\mathbb{E}f(X_n^{x_0}(t))|=0.$$
Consequently,
$$\displaystyle \lim_{x\rightarrow x_0}|\mathbb{E}f(X^{x}(t))-\mathbb{E}f(X^{x_{0}}(t))|\leq 4\eta.$$
Since $\eta$ is arbitrary, the strong Feller property of $P_t$ holds.

Now, we prove that the semigroup $P_{t}$ is irreducible. In fact, for any open ball $B_\delta(z)\subset\mathbb{R}^m$,
choosing $n\in \mathbb{N}^{*}$ sufficiently large such that $\overline{B_\delta(z)}\subset B_{n}(0)$.
Then for each $t>0$ and $x\in\mathbb{R}^{m}$, we have
\begin{align*}
  &\mathbb{P}(X^{x}(t)\in B_\delta(z))\\
  =&\mathbb{P}(X^{x}(t)\in B_\delta(z),t<\tau_{n}^x)+\mathbb{P}(X^{x}(t)\in B_\delta(z),t\geq\tau_{n}^x)\\
  \geq&\mathbb{P}(X^{x}(t)\in B_\delta(z),t<\tau_{n}^x)\\
  =&\mathbb{P}(X^{x}(t\wedge \tau_{n}^x)\in B_\delta(z),t<\tau_{n}^x)\\
  =&\mathbb{P}(X^{x}(t\wedge \tau_{n}^x)\in B_\delta(z))-\mathbb{P}(X^{x}(t\wedge \tau_{n}^x)\in B_\delta(z),t\geq\tau_{n}^x)\\
  =&\mathbb{P}(X_n^{x}(t)\in B_\delta(z))-\mathbb{P}(X^{x}(\tau_{n}^x)\in B_\delta(z)).
\end{align*}
Since $|X^{x}(\tau_{n}^x)|\geq n$, $\mathbb{P}(X^{x}(\tau_{n}^x)\in B_\delta(z))=0$.
Therefore the irreducibility of $P_{t}$ follows from the fact that $P^n_{t}$ is irreducible.
\end{proof}

\begin{corollary}Suppose all the assumptions of Theorems {\rm\ref{tightcri}} and {\rm\ref{if}} are satisfied.
Then there exists an  $\epsilon_0$  such that for $\epsilon \in (0, \epsilon_0]$,
{\rm (\ref{lsode})} has a unique stationary measure $\mu^{\epsilon}$ and $\{\mu^{\epsilon}: 0 < \epsilon \leq \epsilon_0\}$ is tight. If  $\mu^{\epsilon_i} \overset{w}{\rightarrow}
\mu $ as $\epsilon_i \rightarrow 0$, then $\mu$ is an invariant
measure of $X^{0}(t)$, which supports on the Birkhoff center $B(X^{0})$.
\end{corollary}

\subsection{Examples}

\begin{example}[Monotone Cyclic Feedback Systems with Noise]
\end{example}
A typical monotone cyclic feedback system is
given by the $N+1$ equations
\begin{equation}\label{mcfs}
  \dot{x}^i(t)= -b_i{x}^i(t)+f^i({x}^{i+1}(t)),\ 0\leq i \leq N
\end{equation}
where each $b_i$ is positive, $N\geq 0$, the indices are taken mod $N+1$ and each $f^i$ enjoys the monotonicity property
\begin{equation}\label{mon}
  \frac{df^i(s)}{ds} \neq 0,\ \textrm{ for\ all}\ s\in\mathbb{R}, \ 0\leq i \leq N.
\end{equation}
After a sequence of normalizing transformations fully described in Mallet-Paret and Sell \cite{MPSE1}, we may assume that
\begin{equation}\label{mon1}
  \frac{df^i(s)}{ds} > 0,\  \delta\frac{df^N(s)}{ds} > 0,\ \textrm{ for\ all}\ s\in\mathbb{R}, \ 0\leq i \leq N-1
\end{equation}
where $\delta \in \{-1, 1\}$. Monotone cyclic feedback systems (\ref{mcfs}), (\ref{mon1}) arise in versions of the classical Goodwin model of enzyme synthesis and in the theory of neural networks. In application, the functions $f^i, i=0,1,\cdots,N$, are often assumed to have sigmoidal shapes. Hence, we always assume that $f^i, i=0,1,\cdots,N$, are bounded and continuously differentiable with bounded derivatives. Then Mallet-Paret and Smith \cite{MPSM} proved the following Poincar\'{e}-Bendixson Theorem.

\begin{theorem}[The Poincar\'{e}-Bendixson Theorem]\label{PB}
Consider the system {\rm(\ref{mcfs})} with each $f^i$ satisfying the above assumptions.
Let $x(t)$ be a solution of {\rm(\ref{mcfs})} on
$[0, \infty)$. Let $\omega(x)$ denote the $\omega$-limit set of this solution in the phase
space $\mathbb{R}^{N+1}$. Then either

${\rm (a)}$ $\omega(x)$ is a single non-constant periodic orbit; or else

${\rm (b)}$ for solutions with
$u(t)\in \omega(x)$  for all $t\in \mathbb{R}$, we have that
$$\alpha(u)\cup \omega(u)\subset \mathcal{E},$$
where $\alpha(u)$ and $\omega(u)$ denote the $\alpha$- and $\omega$-limit sets, respectively, of
this solution, and where $\mathcal{E}\subset \mathbb{R}^{N+1}$ denotes the set of equilibrium points of
{\rm(\ref{mcfs})}.
\end{theorem}

Now we consider the system driven by a  L\'{e}vy process
\begin{equation}\label{smcfs}
\begin{split}
d{x}^i(t)=&\big[-b_i{x}^i(t)+f^i({x}^{i+1}(t))\big]dt\\
&+\epsilon\sum_{j=1}^{k}\sigma_{ij}(x(t))dW_j(t)+\epsilon\int_{|y|_{\mathbb{R}^l}<c}F^i(x(t-),y)\tilde{N}(dt,dy),\
\end{split}
\end{equation}
for $0\leq i \leq N$, where $(N+1)\times k-$dispersion matrix $\sigma(x):=(\sigma_{ij}(x))$ and $F$ have global Lipschitz continuous and linear growth properties.

Define $V:\mathbb{R}^{N+1}\rightarrow \mathbb{R}_+$ by
$$V(x):=\frac{1}{2}\sum_{i=0}^{N}{(x^i)}^2.$$
Then
\begin{align*}
       \mathcal{L}^\epsilon V(x)=&
  -\displaystyle\sum_{i=0}^{N}b_{i}{(x^i)}^2+\displaystyle\sum_{i=0}^{N}x^if^i({x}^{i+1})+ \frac{\epsilon^{2}}{2}\displaystyle\sum_{i=0}^{N}a_{ii}(x)\\
  &+
  \int_{|y|_{\mathbb{R}^l}<c}\big(V(x+\epsilon F(x,y))-V(x)-\langle \nabla V(x),\epsilon F(x,y)\rangle_{\mathbb{R}^{N+1},\mathbb{R}^{N+1}}\big)\nu(dy).
\end{align*}
It follows from the assumptions that all $f^i, i=0,1,\cdots,N$, are bounded and the dispersion matrix $\sigma(x)$ and $F$ have
linear growth that there is a positive constant $\widetilde{L}$ such that
$$\mathcal{L}^{\epsilon}V(x)\leq-b\displaystyle\sum_{i=0}^{N}{(x^i)}^2+\widetilde{L}\big(|x|+\epsilon^2(|x|^{2}+1)\big)$$
where $b=\displaystyle\min_{0\leq i\leq N}b_{i}$. This shows that there are $\epsilon_0>0$ and $R>0$  such that as $\epsilon\in(0,\epsilon_{0}]$  one enjoys
$$\mathcal{L}^{\epsilon}V(x)\leq-\frac{b}{2}\displaystyle\sum_{i=0}^{N}{(x^i)}^2,\  \textrm{ for}\ |x|>R.$$
By Theorem \ref{tightcri}, the set of all stationary measures for (\ref{mcfs}) ($0<\epsilon \leq \epsilon_{0}$)
is tight.

From the Poincar\'{e}-Bendixson Theorem, we know that the Birkhoff center $B(\Phi)=\mathcal{E}\cup \mathcal{P}$, where $\Phi$ is the flow generated by (\ref{mcfs}) and $\mathcal{P}$ denotes the set of nontrivial periodic orbits.

Applying  Theorem \ref{supportode}, we conclude that

\begin{theorem}\label{supp}
Let $\mu=\displaystyle\lim_{\epsilon_i \rightarrow 0}\mu^{\epsilon_i}$ be a weak limit point of
$\{\mu^{\epsilon_i}\}$. Then $\mu$ is an invariant measure of the flow $\Phi$ and the ${\rm supp}(\mu)$ is contained
in $\mathcal{E}\cup \mathcal{P}$.
\end{theorem}

\begin{remark} Theorem {\rm\ref{supp}} is still valid for those systems if the Poincar\'{e}-Bendixson Theorem holds for the
unperturbed systems, for example,  planar systems and Morse-Smale higher dimensional systems, perturbed by white noise or L\'{e}vy process.
\end{remark}

Example 1 is complete.
$\qquad\qquad\qquad\qquad\qquad\qquad\qquad\qquad\qquad\qquad\quad\Box$

\

We give an example to show our result can be used to system for drift term and diffusion term to have polynomial growth.

\begin{example} Consider the system
\begin{equation}\label{lcycle}
     \left\{\begin{aligned}
dx&=[x-y-x(x^2+y^2)]dt+\epsilon (x^2+y^2)dW_{t}^{1},\\
dy&=[x+y-y(x^2+y^2)]dt+\epsilon (x^2+y^2)dW_{t}^{2}.
\end{aligned}
\right.
\end{equation}
\end{example}
Let $V(x,y)= x^2+y^2$. Then for $0<\epsilon<\frac{1}{\sqrt{2}}$,
$$\mathcal{L}^{\epsilon}V(x,y)=2(x^2+y^2)\big(1-(x^2+y^2)\big)+2\epsilon^2(x^2+y^2)^2\leq -\frac{1}{2}(x^2+y^2)^2$$
for $x^2+y^2$ sufficiently large. This shows that all conditions of Theorem \ref{supportode} hold. It is easy to see that the Birkhoff center for corresponding deterministic system in (\ref{lcycle}) with $\epsilon = 0$ is $\{O, \mathbf{S}^1\}$ where $ \mathbf{S}^1$ denotes the unit cycle. Employing Theorem \ref{supportode}, we have that ${\rm supp}(\mu)\subset \{O, \mathbf{S}^1\} $ for any stationary measures $\{\mu^{\epsilon_i}\}$ of (\ref{lcycle}) such that $\mu=\displaystyle\lim_{\epsilon_i \rightarrow 0}\mu^{\epsilon_i}$ in the sense of a weak limit.

In particular, $O(0,0)$ is a solution of (\ref{lcycle}), which implies that $\mu^{\epsilon}=\delta_O$ is a stationary measure of (\ref{lcycle}) and concentrates at the origin. If we replace $x^2+y^2$ in the diffusion terms by $x^2+y^2-1$, then $\mathbf{S}^1$ is invariant for (\ref{lcycle}) in this case. Therefore, the Haar measure on $\mathbf{S}^1$ is a stationary measure for any $\epsilon$.

Which invariant measure for deterministic system $\dot{x}=b(x)$ can be limiting measure for a sequence of stationary measures for (\ref{lsode})? Such a problem strongly depends on the type of noise, which is shown as follows.

Suppose that $X^{0,x_0}(t)$ is a bounded solution of $\dot{x}=b(x)$. Denote by $\mathcal{I}_{x_0}$ the set of invariant measures generated by  the family of probability measures
$$  P^{0,t}(x_0,B)=\frac{1}{t}\int_{0}^{t}\delta_{X^{0,x_0}(s)}(B)ds $$
via Krylov-Bogoliubov procedure. Let $r>0$ such that $X^{0,x_0}(t)\in B_r(O)$ for all $t \geq 0$. We can construct a $C^{\infty}$ diffusion term $\sigma$ satisfying $\sigma = 0$ on $B_r(O)$ and $\sigma = {\rm constant\ matrix} M$ on $(B_{r+1}(O))^c$. Consider SODEs
\begin{equation}\label{cons}
dX^{\epsilon,x}(t)=b(X^{\epsilon,x}(t))dt+\epsilon\sigma(X^{\epsilon,x}(t))dW_{t}.
\end{equation}
Then we have

\begin{proposition}\label{procons}
Suppose that $b$ is globally Lipschitz continuous. Then $X^{\epsilon,x_0}(t)=X^{0,x_0}(t)$ for all $t\geq 0$ and $\mathcal{S}^{\epsilon}_{x_0}\equiv \mathcal{I}_{x_0}$ for all $\epsilon$. In particular, for any $\mu \in \mathcal{I}_{x_0}$, $\mu_{x_0}^{\epsilon}\equiv \mu \overset{w}{\rightarrow}
\mu$ as $\epsilon \rightarrow 0$.
\end{proposition}

This proposition illustrates that under mild regular condition on drift term, for any invariant measure $\mu$ of $\dot{x}=b(x)$, there exists a diffusion term $\sigma$ with small noise intensity $\epsilon$ such that there is a sequence of stationary measures for (\ref{cons}) converging to $\mu$ weakly as $\epsilon \rightarrow \infty$.

Example 2 is complete.
$\qquad\qquad\qquad\qquad\qquad\qquad\qquad\qquad\qquad\qquad\quad\Box$

\

We have observed from examples that the limiting measures of stationary measures will support in stable orbits of the deterministic system decided by drift term. However, the following two examples show that the limit measure can  support at saddles for deterministic system. In summary, limiting measures always support at ``most relatively stable positions".

\begin{example}[The Lemniscate of Bernoulli with Noise]
\end{example}
Let $I(x,y)=(x^{2}+y^{2})^{2}-4(x^{2}-y^{2})$. Define
$$V(I):=\frac{I^2}{2(1+I^2)^{\frac{3}{4}}}\ \textrm{and}\ H(I):= \frac{I}{(1+I^2)^{\frac{3}{8}}}.$$
Consider the vector field
\begin{equation}\label{orth}
  b(x,y):= -[\nabla V(I)+(\frac{\partial H(I)}{\partial y}, -\frac{\partial H(I)}{\partial x})^T]:= -[\nabla V(I)+\Theta(x,y)],
\end{equation}
where $\nabla V(I)=\frac{dV(I)}{dI}(\frac{\partial I}{\partial x},\frac{\partial I}{\partial y})^T$,
$\Theta(x,y)=\frac{dH(I)}{dI}(\frac{\partial I}{\partial y},-\frac{\partial I}{\partial x})^T$.

By calculation,
$$\frac{\partial I}{\partial x}=4x(x^{2}+y^{2})-8x \ , \ \frac{\partial I}{\partial y}=4y(x^{2}+y^{2})+8y.$$

Consider the unperturbed system of ordinary differential equations
\begin{equation}\label{unds}
  \left\{\begin{aligned}
\frac{dx}{dt}&=-f(I)\big(4x(x^{2}+y^{2})-8x\big)-g(I)\big(4y(x^{2}+y^{2})+8y\big),\\
\frac{dy}{dt}&=-f(I)\big(4y(x^{2}+y^{2})+8y\big)-g(I)\big(-4x(x^{2}+y^{2})+8x\big).
\end{aligned}
\right.
\end{equation}
Here $f(I)=\frac{dV(I)}{dI}=\frac{I(I^{2}+4)}{4(1+I^{2})^{\frac{7}{4}}}$ and $g(I)=\frac{dH(I)}{dI}=\frac{I^{2}+4}{4(1+I^{2})^{\frac{11}{8}}}$.

We will summarize the global behavior of (\ref{unds}) in the following proposition.

\begin{proposition}\label{Bernoulli}
The system {\rm(\ref{unds})} has a global Lipschitz constant and the equilibria $O(0,0)$, $P^+(\sqrt{2},0)$ and $P^-(-\sqrt{2},0)$. $V(I)$ is its Lyapunov function. When the initial point $p$ locates outside of the {\it Bernoulli Lemniscate}:
\begin{equation}\label{BL}
L:\ \ (x^{2}+y^{2})^{2}=4(x^{2}-y^{2}),
\end{equation}
its $\omega$-limit set $\omega(p)=L$, which is a red curve in Figure ${\rm\ref {diagram}}$; when the initial point $p\neq P^-$ (resp. $p\neq P^+$ ) locates left (resp. right) inside of the Bernoulli Lemniscate, its $\omega$-limit set the left (resp. right) branch of $L$. However, the Birkhoff center $B(\Phi)$ for this solution flow $\Phi$ is $\{O, P^+, P^-\}$.
\end{proposition}

\begin{figure}[htp]
\begin{center}
  \includegraphics[width=7cm,scale=1.2]{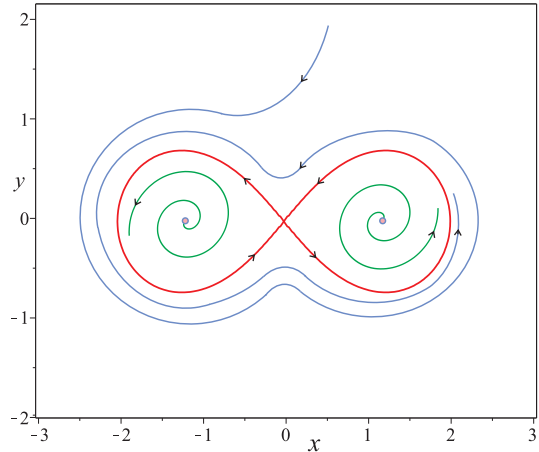}\\
  \caption{The phase portrait of (\ref{unds}) with $b(x,y)=-\nabla V(x,y)-\Theta(x,y)$.  }\label{diagram}
 \end{center}
\end{figure}
\begin{proof}
It is easy to see that
\begin{equation}\label{Vinfty}
\lim_{|(x,y)|\rightarrow \infty}V(I(x,y))=\infty.
\end{equation}
Since $\nabla V(I)$ and $\Theta(x,y)$ are orthogonal, the derivative of the function $V(I(x,y))$ along a solution is
\begin{equation}\label{LiaV}
\dot{V}=-|\nabla V(I)|^2.
\end{equation}
(\ref{Vinfty}) and (\ref{LiaV}) imply that all positive trajectories for (\ref{unds}) are bounded. The LaSalle invariance
principle deduces that for any $p\in \mathbb{R}^2$,
$$\omega(p)\subset \{(x,y): \nabla V(I)=0\}=L\cup \{P^+,P^-\}.$$
In particular, the equilibria for (\ref{unds}) is contained in $L\cup \{P^+,P^-\}.$  It is easy to calculate that there uniquely exists an equilibrium on $L$, which is the origin $O$, and that the other equilibria are $P^+$ and $P^-$. It is not hard to get that
\begin{equation*}
  \begin{array}{cc}
    Db(0,0)=\left(
     \begin{array}{cc}
       0 & -8 \\
       -8 & 0 \\
     \end{array}
   \right), &  Db(\sqrt{2},0)=Db(-\sqrt{2},0)=16\left(
     \begin{array}{cc}
        \frac{20}{17^{\frac{7}{4}}}& -\frac{5}{17^{\frac{11}{8}}} \\
        [5pt]
       \frac{5}{17^{\frac{11}{8}}} & \frac{20}{17^{\frac{7}{4}}} \\
     \end{array}
   \right). \\
  \end{array}
\end{equation*}
This implies that $(0,0)$ is a saddle point and $(\pm \sqrt{2},0)$ are unstable focus. Combining the LaSalle invariance principle and the Poincar\'{e}-Bendixson Theorem, we can obtain the $\omega$-limit set of each trajectory for (\ref{unds}), as shown in Figure \ref{diagram}.

By estimation, we can obtain the following inequalities:
$$ |\frac{\partial b_{i}(x,y)}{\partial x}|,\  |\frac{\partial b_{i}(x,y)}{\partial y}|\leq 130\sqrt[4]{8},\  r\geq 4,\ i=1,2,$$
where $r=\sqrt{x^2+y^2}$. Therefore, $b(x,y)$ is a globally Lipschitz function. This completes the proof.
\end{proof}

Now we consider perturbed system of (\ref{unds}) driven by Brownian motion:
\begin{equation}\label{persode}
\left\{\begin{aligned}
dx&=b_{1}(x,y)dt+\epsilon [\sigma_{11}(x,y)dW_{t}^{1}+\sigma_{12}(x,y)dW_{t}^{2}],\\
dy&=b_{2}(x,y)dt+\epsilon [\sigma_{21}(x,y)dW_{t}^{1}+\sigma_{22}(x,y)dW_{t}^{2}],
\end{aligned}
\right.
\end{equation}
where $\sigma_{ij}(i,j\in\{1,2\})$ satisfies global Lipschitz condition, which implies that there exist nonnegative constants $C_{1}$, $C_{2}$ such that
$$|a_{ij}(x,y)|\leq C_{1}|(x,y)|^{2}+C_{2}\leq C_{1}r^{2}+C_{2}, \:\:\textrm{for}\:\: i,j=1,2,$$
where $a_{ij}(x,y)= \displaystyle\sum_{k=1}^{2}\sigma_{ik}(x,y)\sigma_{jk}(x,y)$.

\

\begin{theorem}\label{SML}
Suppose that $\sigma_{ij}(i,j\in\{1,2\})$ satisfies global Lipschtiz condition,

{\rm(i)} if $C_{1}=0$, then for any $\epsilon$, the system {\rm(\ref{persode})} admits at least one stationary measure $\mu^{\epsilon}$;

{\rm(ii)} if $C_{1}>0$, then the system {\rm(\ref{persode})} possesses at least one stationary measure $\mu^{\epsilon}$ for $0<\epsilon<\frac{1}{8\sqrt{26C_{1}}}$.

If, in addition, the diffusion matrix $a(x,y)$ is positively definite everywhere, then for a given $\epsilon$ as above,
the stationary measure $\mu^{\epsilon}$ is unique, and $\mu^{\epsilon} \overset{w}{\rightarrow}
\delta_O(\cdot)$ as $\epsilon\rightarrow 0$, where $\delta_O(\cdot)$ denotes the Dirac measure at the saddle $O$.
\end{theorem}

\begin{proof} In fact, from the above inequalities one can see that for $r>4$,
$$|\frac{\partial^{2} V}{\partial x^{2}}|, |\frac{\partial^{2} V}{\partial y^{2}}|,|\frac{\partial^{2} V}{\partial y\partial x}| \leq 104\sqrt{2},\ |\nabla V(x,y)|^{2}\geq\frac{r^{2}}{4\sqrt{2}}, $$
and
$$\mathcal{L}^{\epsilon}V(x,y)\leq-[\frac{1}{4\sqrt{2}}-208\sqrt{2}C_{1}\epsilon^{2}]r^{2}+208\sqrt{2}C_{2}\epsilon^{2}\rightarrow -\infty \ {\rm as}\  r\rightarrow\infty.$$
Applying the Khasminskii Theorem \ref{tightcri}, we conclude that there is at least one stationary measure for (\ref{persode}) if $0<\epsilon<\frac{1}{8\sqrt{26C_{1}}}$ with $C_{1}>0$ or all $\sigma_{ij}(x,y)$ are bounded on the plane, and that this stationary measure is unique if $a(x,y)$ is positively definite everywhere.

In (i) and (ii),  from the Tightness Criterion it follows that the set of stationary measures
$\{\mu^{\epsilon}:\epsilon\in(0,\epsilon_{0}]\}$ is tight. Thus the Prohorov theorem implies that any sequence $\{\mu^{\epsilon_i}\}$ of stationary measures with $\epsilon_i\rightarrow 0$ contains a weak convergent subsequence.
Let $\mu$ be any weak limit measure. Then Theorem \ref{mthm} deduces that ${\rm supp}(\mu)\subset B(\Phi)$. However, in view of Proposition \ref{Bernoulli}, $B(\Phi)= \{O, P^+, P^-\}$, which implies that ${\rm supp}(\mu)\subset\{O, P^+, P^-\}$.

Finally, we show that $\mu(\{P^+,P^-\})=0$.
Since matrix $Db(P^+)$ has all eigenvalues with positive real parts. Thus there exists a positive
definite matrix $B$ satisfying
\[  \big(Db(P^{+})\big)^{T}B+B\big(Db(P^{+})\big)=I  \]
Let $\tilde{V}^{+}(z)=(z-P^+)^{T}B (z-P^+)$, where $z=(x,y)$.
It is easy to see that there exists a neighborhood $\mathcal{U}:=B_\delta(P^+)$ of $P^+$
such that $\langle\nabla \tilde{V}^{+}, b\rangle>0$ on
$\mathcal{U}\setminus \{P^+\}$ (e.g. see \cite{HSB}).
We denote $\rho_{M}:=\displaystyle\sup_{(x,y)\in \mathcal{U}}\tilde{V}^{+}(x,y)$ (called essential upper bound of $\tilde{V}^{+}$).
Then
\begin{equation}\label{antiLy}
  \begin{split}
\mathcal{L}^{\epsilon}\tilde{V}^{+}(x,y)
    =&\langle b(x,y),\nabla \tilde{V}^{+}(x,y)\rangle\\
    &+\frac{\epsilon^{2}}{2}[a_{11}(x,y)\frac{\partial^{2} V}{\partial x^{2}}+
    2a_{12}(x,y)\frac{\partial^{2} V}{\partial y\partial x}+a_{22}(x,y)\frac{\partial^{2} V}{\partial y^{2}}]\\
    \geq&\frac{\epsilon^{2}}{2}[a_{11}(x,y)b_{11}+2
    a_{12}(x,y)b_{12}+a_{22}(x,y)b_{22}]\\
    \geq&\widetilde{m}\epsilon^2=:\gamma>0, \quad \forall (x,y)\in\mathcal{U}.
  \end{split}
\end{equation}
We used here the fact that $B$ is positively
definite and $A(x,y)$ is positively definite on $\bar{B}_\delta(P^+)$.
It follows from (\ref{antiLy}) that $\tilde{V}^{+}$ is an anti-Lyapunov function with respect
to (\ref{persode}) in $B_\delta(P^+)$ with anti-Lyapunov constant $\widetilde{m}\epsilon^2$ and
essential lower bound $\rho_m=0$
(e.g. see \cite[Definition 2.2]{HJLY2}).
It is obviously
\begin{align}\label{bddgra}
  \nabla \tilde{V}^{+}(x,y)
=&\big(2b_{11}(x-\sqrt{2})+2b_{12}y,2b_{12}(x-\sqrt{2})+2b_{22}y\big)\nonumber\\
\neq &0, \quad
\forall (x,y)\in \big(\tilde{V}^{+}\big)^{-1}(\rho) \:\:\:\textrm{for}\: a.e.\:\rho\in[0,\rho_{M}),
\end{align}
where $\big(\tilde{V}^{+}\big)^{-1}(\rho)=\{(x,y)\in\mathcal{U}:\tilde{V}^{+}(x,y)=\rho\}$. Note that
\begin{equation}\label{Hrho}
  \begin{split}
&\frac{\epsilon^2}{2}[a_{11}(x,y)(\frac{\partial \tilde{V}^{+}}{\partial x})^{2}+
    2a_{12}(x,y)\frac{\partial \tilde{V}^{+}}{\partial x}\frac{\partial \tilde{V}^{+}}{\partial y}+a_{22}(x,y)(\frac{\partial \tilde{V}^{+}}{\partial y})^{2}]\\
    \leq& \frac{\epsilon^2}{2}\|A(x,y)\| |\big(\frac{\partial \tilde{V}^{+}}{\partial x}(x,y),\frac{\partial \tilde{V}^{+}}{\partial y}(x,y)\big)|^2\\
    \leq& \displaystyle\sup_{(x,y)\in B_\delta(P^+)}\|A(x,y)\|\widetilde{M}_0|\tilde{V}^{+}(x,y)|\epsilon^2\\
    =:&\widetilde{M}\rho\epsilon^2=:H(\rho), \quad (x,y)\in \big(\tilde{V}^{+}\big)^{-1}(\rho) \:\:\:\textrm{for}\:\:\rho\in[0,\rho_{M}).
  \end{split}
\end{equation}
Without loss of generality, we may assume $\mu^{\epsilon}(\mathcal{U})>0$ for each $\epsilon>0$. It is easy to verify that $\tilde{\mu}^{\epsilon}(\cdot)=\frac{\mu^{\epsilon}|_{\mathcal{U}}(\cdot)}{\mu^{\epsilon}(\mathcal{U})}$ is a
stationary measure in $\mathcal{U}$, by a regularity result on stationary measure in \cite{BogaKR}, we have known that
$\tilde{\mu}^{\epsilon}$ admits a positive density function $\tilde{u}\in W^{1,p}_{loc}(\mathcal{U})$. Let $\Omega_{\rho}=\{(x,y)\in\mathcal{U}:\tilde{V}^{+}(x,y)<\rho\}$,
$\Omega^{*}_{\rho}=\Omega_{\rho}\cup\big(\tilde{V}^{+}\big)^{-1}(\rho)$ for each $\rho\in[0,\rho_M)$.
The regularity implies that $\tilde{\mu}^{\epsilon}(\Omega^{*}_{\rho_m})=\tilde{\mu}^{\epsilon}(\{P^+\})=0$. Measure
estimate theorem in Huang-Ji-Liu-Yi \cite[Theorem B a)]{HJLY} asserts that for any $\rho_0\in(0,\rho_M)$,
\begin{equation*}
  \begin{split}
  \tilde{\mu}^\epsilon(\Omega_\rho)=\tilde{\mu}^\epsilon(\Omega_\rho \setminus\Omega^{*}_{\rho_m})
\geq&\tilde{\mu}^\epsilon(\Omega_{\rho_0} \setminus\Omega^{*}_{\rho_m})e^{\gamma\int_{\rho_0}^{\rho}\frac{1}{H(t)}dt}\\
=&\tilde{\mu}^\epsilon(\Omega_{\rho_0})e^{\widetilde{m}\epsilon^2\int_{\rho_0}^{\rho}\frac{1}{\widetilde{M}\epsilon^2 t}dt}\\
=&\tilde{\mu}^\epsilon(\Omega_{\rho_0})e^{\frac{\widetilde{m}}{\widetilde{M}}\int_{\rho_0}^{\rho}\frac{1}{t}dt}
  ,\quad \rho\in(\rho_0,\rho_M),
   \end{split}
\end{equation*}
This is,
\[  \mu^\epsilon(\Omega_{\rho_0})\leq \mu^\epsilon(\Omega_\rho)e^{-\frac{\widetilde{m}}{\widetilde{M}}\int_{\rho_0}^{\rho}\frac{1}{t}dt}
  \leq e^{-\frac{\widetilde{m}}{\widetilde{M}}\int_{\rho_0}^{\rho}\frac{1}{t}dt},\quad \rho\in(\rho_0,\rho_M). \]
Since $\mu^\epsilon\xlongrightarrow{w}\mu$ as $\epsilon\rightarrow 0$, and $\Omega_{\rho_0}$ is an open set, we have
\[  \mu(\Omega_{\rho_0})\leq e^{-\frac{\widetilde{m}}{\widetilde{M}}\int_{\rho_0}^{\rho}\frac{1}{t}dt}.\]
Finally, letting $\rho_0\rightarrow 0$, we obtain $\mu(\{P^+\})=0$.
Analogously, we can verify that $\mu(\{P^-\})=0$. We conclude that $\mu=\delta_O(\cdot)$, that is,
$\mu^{\epsilon} \overset{w}{\rightarrow}
\delta_O(\cdot)$ as $\epsilon\rightarrow 0$.
\end{proof}

\begin{remark} From the above arguments, we have obtained that if the system ${\rm(\ref{unds})}$ is driven by Brownian motion and
the diffusion matrix is positively definite everywhere, then any limiting measure is $\delta_{O}$. However, if we get rid of nondegenerate condition for the diffusion matrix , then it is possible for limiting measure to be either $\delta_{P^+}$ or $\delta_{P^-}$ from Proposition {\rm\ref{procons}}.
The problem is whether or not such result still holds if it is driven by L\'{e}vy process,
we can only get that the limiting measure is supported in $\{O, P^+, P^-\}$.
\end{remark}

Example 3 is complete.
$\qquad\qquad\qquad\qquad\qquad\qquad\qquad\qquad\qquad\qquad\quad\Box$

\

\begin{example}[May-Leonard System with a Noise Perturbation]
\end{example}
Consider the May-Leonard system with a white noise perturbation:
\begin{equation}
    \label{sys:02}
    \left\{
        \begin{array}{l}
        dy_1=y_1(1-y_1-\beta y_2-\gamma y_3)dt+\epsilon y_1\circ dW_t, \\
        dy_2=y_2(1-y_2-\beta y_3-\gamma y_1)dt+\epsilon y_2\circ dW_t, \\
        dy_3=y_3(1-y_3-\beta y_1-\gamma y_2)dt+\epsilon y_3\circ dW_t, \\
        \end{array}
    \right.
\end{equation}
where $\circ$ denotes the Stratonovich stochastic integral, $\beta,\ \gamma >0$ and
$\epsilon$ denotes noise intensity.

Recalling from \cite{CDJNZ}, we have the following
{\it Stochastic\ Decomposition\ Formula}:
\begin{equation}\label{sdfi}
\Phi^{\epsilon}(t,\omega,y)= g^{\epsilon}(t,\omega,g_0)\Phi^0(\int_0^tg^{\epsilon}(s,\omega,g_0)ds,\frac{y}{g_0}),
\end{equation}
where $\Phi^{\epsilon},\ \Phi^0$ are the solutions of (\ref{sys:02}) and the corresponding deterministic system without noise, respectively,
and $g^{\epsilon}$ is the solution of stochastic Logistic equation
\begin{equation}\label{slogistic}
dg =  g(1- g)dt + \epsilon g\circ dW_t.
\end{equation}

In order to obtain the stationary properties for (\ref{sys:02}) in detail, we need the asymptotic properties for deterministic flow $\Phi^0$. It is well-known from Hirsch \cite{H88} that the flow $\Phi^0$ admits an invariant surface $\Sigma$ (called {\it carrying simplex}), homeomorphic to the closed unit simplex $\Delta^2=\{y\in \mathbb{R}^3_+: \sum_i y_i=1\}$ by radial projection, such that every trajectory in $\mathbb{R}^3_+ \setminus \{O\}$ is asymptotic to one in $\Sigma$. So we will draw phase portraits on $\Sigma$ (see Table \ref{biao0}).

It is easy to see that $\Phi^0$ always possesses equilibria: the origin $O(0,0,0)$, three axial
equilibria $R_1(1,0,0),\ R_2(0,1,0),\ R_3(0,0,1)$ and the unique positive equilibrium $P=\frac{1}{1+\beta+\gamma}(1,1,1)$.
$\Phi^0$ has planar equilibria:
$R_{12}=\frac{1}{1-\beta\gamma}(1-\beta,1-\gamma,0)$, $R_{23}=\frac{1}{1-\beta\gamma}(0,1-\beta,1-\gamma)$,
$R_{31}=\frac{1}{1-\beta\gamma}(1-\gamma,0,1-\beta)$
if and only if $(1-\beta)(1-\beta\gamma)>0$, $(1-\gamma)(1-\beta\gamma)>0$.
The classification for the flow $\Phi^0$ on the carrying simplex $\Sigma$ is summarized in Table \ref{biao0}.

\begin{center}
\begin{longtable}{l@{\extracolsep{\fill}}l@{\extracolsep{\fill}}c}
\caption{The classification for the flow $\Phi^0$ on $\Sigma$}\\[-2pt]
        \hline
         Parameter conditions & Equilibria   & Phase Portrait\\
        \hline
        \endfirsthead
        \caption[]{(continued)}\\
        \hline
         Parameter conditions & Equilibria   & Phase Portrait\\
        \hline
&&\\
        \endhead
        \hline
        \endfoot
        \endlastfoot
&&\\
a: \begin{tabular}{l}
 $0<\beta,\gamma<1$
\end{tabular}
 & $O,R_{1},R_{2},R_{3},R_{12},R_{13},R_{23},P$&
    \parbox{2cm}{\vspace{2pt}\includegraphics[width=2.2cm,height=1.6cm]{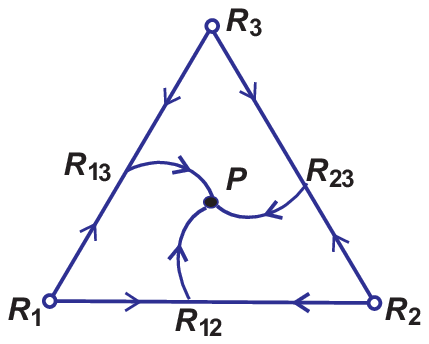}\vspace{2pt}}\\
&&\\
b:\begin{tabular}{l}
 (i)~$\beta+\gamma<2$\\
 (ii)~$\beta\geq 1,\gamma<1$\\
  or ~$\gamma\geq 1,\beta<1$
\end{tabular} &$O,R_{1},R_{2},R_{3},P$
 &
    \parbox{2cm}{\vspace{2pt}\includegraphics[width=2.2cm,height=1.6cm]{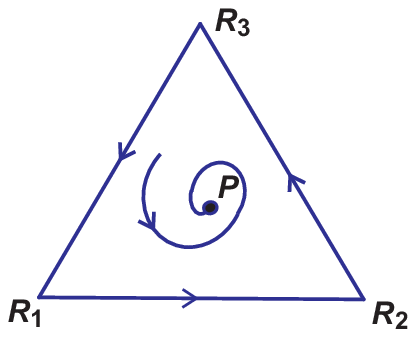}\vspace{2pt}} \\
&&\\
c:\begin{tabular}{l}
 (i) $\beta+\gamma=2$\\
(ii)~ $\beta,\gamma\neq 1$
\end{tabular} &$O,R_{1},R_{2},R_{3},P$
 &
    \parbox{2cm}{\vspace{2pt}\includegraphics[width=2.2cm,height=1.6cm]{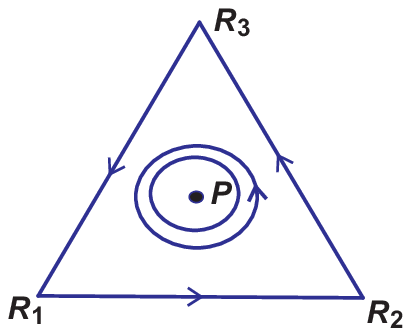}\vspace{2pt}}\\
&&\\
d:\begin{tabular}{l}
(i)~ $\beta+\gamma>2$\\
(ii)~ $\gamma> 1,\beta \leq 1$\\
or ~$\gamma\leq 1,\beta > 1$
\end{tabular} &
$O,R_{1},R_{2},R_{3},P$
 &
    \parbox{2cm}{\vspace{2pt}\includegraphics[width=2.2cm,height=1.6cm]{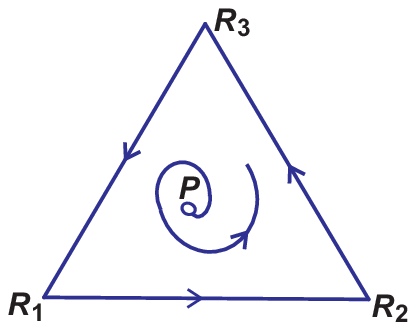}\vspace{2pt}} \\
&&\\
e: \begin{tabular}{l}
 $\beta,\gamma>1$
\end{tabular}
&
$O,R_{1},R_{2},R_{3},R_{12},R_{13},R_{23},P$
 &
    \parbox{2cm}{\vspace{2pt}\includegraphics[width=2.2cm,height=1.6cm]{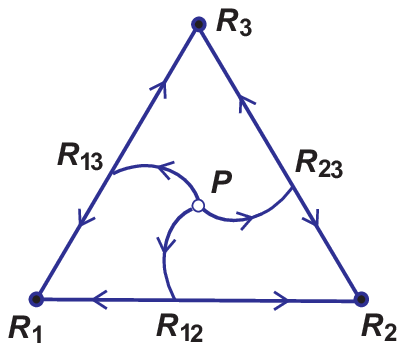}\vspace{2pt}} \\
&&\\
f: \begin{tabular}{l}
 $\beta=\gamma=1$
\end{tabular} &
$\forall x\in \Sigma\cup\{O\}$
 &\parbox{2cm}{\vspace{2pt}\includegraphics[width=2.2cm,height=1.6cm]{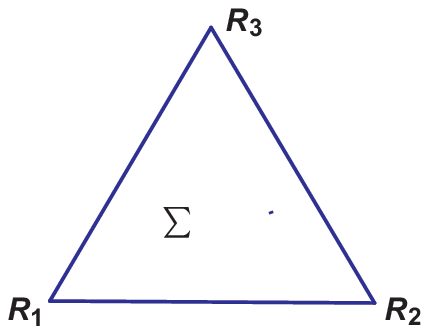}\vspace{2pt}}     \\
[-8pt]
\label{biao0}
\end{longtable}
\end{center}

Set by $\omega(z)$ the $\omega-$limit set for the trajectory $\Phi^0(t,z)$.
Then it follows from Table \ref{biao0} that any $\omega(z)$ is either an equilibrium, or a closed orbit, or heterclinic cycle. Define
$$\mathcal{A}(\omega(z)):=\{y\in \mathbb{R}^3_+: \lim_{t\to \infty}\mathrm{dist}(\Phi^0(t,y),\omega(z))=0\}$$
to be the attracting domain for $\omega(z)$,  and let $\mathcal{A}_{\Sigma}(\omega(z))$ denote the attracting domain for $\omega(z)$ on $\Sigma$, which can be derived from Table \ref{biao0} for each case. It follows from \cite[Proposition 4.13]{JN} that any pair of points on $L(y):=\{\lambda y:\lambda > 0\}$ have the same $\omega-$limit set. We can obtain the attracting domain for $\omega(z)$ as follows
\begin{equation}\label{attra}
 \mathcal{A}(\omega(z))=\bigcup\{L(y): y\in \mathcal{A}_{\Sigma}(\omega(z))\}.
\end{equation}
Together with Table \ref{biao0} and (\ref{attra}), we can obtain the attracting domains for an equilibrium, a closed orbit, or the heterclinic cycle, respectively.

Using the Stochastic Decomposition Formula (\ref{sdfi}),
we have shown that Probability Convergence (\ref{pc}) holds, that is, $\Phi^{\epsilon}$ converges to $\Phi^0$ as $\epsilon \rightarrow 0$ in probability (see \cite[Proposition 1]{CDJNZ}) and that for any $y\neq O$, the probability measures
\begin{equation}\label{0pms}
P_T(\cdot,y):=\frac{1}{T}\int_0^{T}P(t,y,\cdot)dt,\ \ T>0
\end{equation}
is tight, and has at least a limiting measure  $\nu_y^{\epsilon}$ in weak sense, which is a stationary measure for (\ref{sys:02}) (see \cite[Theorem 8]{CDJNZ}). Now denote by

$$\mathcal{M}_S(\epsilon_0):= \{\nu_y^{\epsilon}: {\rm for\ all}\ 0<\epsilon \leq \epsilon_0,\ y\in \mathbb{R}^3_+  \} $$
all the stationary measures obtained in a manner just stated, which is tight (see \cite[Proposition 2]{CDJNZ}) and produced by all solutions from $\Sigma \cup \{O\}$ (see \cite[Theorem 12]{CDJNZ}).

It is not difficult to see that stochastic Logistic equation (\ref{slogistic}) has a unique nontrivial stationary solution $u^{\epsilon}(\theta_t\omega)$ where $\theta_t$ is the metric
dynamical system generated by Brownian motion. It follows from the Stochastic Decomposition Formula that $u^{\epsilon}(\theta_t\omega)Q$ is a stationary solution of (\ref{sys:02}) for any equilibrium $Q\in \mathcal{E}$ (see \cite[Theorem 5]{CDJNZ}), whose distribution measure, denoted by $\mu^\epsilon_Q$,  defines an ergodic stationary measure for (\ref{sys:02}) (see \cite[Theorem 7]{CDJNZ}).
In addition, for each $y\in \mathcal{A}(Q)$, $P(t,y,\cdot)  {\rightarrow}
\mu^\epsilon_Q$ weakly as $t\rightarrow \infty$, and
\begin{equation}\label{bwc}
\lim_{t\rightarrow \infty}P(t,y,A)=\mu^\epsilon_Q(A), \ {\rm for\ any}\
A\in
 \mathcal{B}(\mathbb{R}_+^3).
\end{equation}
This shows
$$\mathcal{M}_S(\epsilon_0)= \{\mu^\epsilon_Q: Q\in \mathcal{E},0<\epsilon \leq \epsilon_0  \}$$
in the cases a, b, e, f of Table \ref{biao0}.

In the case c of Table \ref{biao0}, the carrying simplex is full of periodic orbits $\Gamma(h)\
(0<h\leq\frac{1}{27})$:
\begin{displaymath}
    \left \{
    \begin{array}{l}
        y_1+y_2+y_3=1, \\
        \noalign{\medskip}
        y_1y_2y_3=h,\\
 \end{array}
    \right.
\end{displaymath}
whose attracting domain is the  invariant cone surface $\Lambda(h):$
$$\ \frac{y_1y_2y_3}{(y_1+y_2+y_3)^3}=h,\ 0<h\leq\frac{1}{27}.$$
Then there exists a unique ergodic nontrivial stationary measure $\nu_h^{\epsilon}$ supporting on the cone surface $\Lambda(h)$ (see \cite[Theorem 20]{CDJNZ}). Hence,
$$\mathcal{M}_S(\epsilon_0)=\{\nu_h^{\epsilon}:0<h\leq\frac{1}{27}, 0<\epsilon \leq \epsilon_0\}\bigcup \{\mu_Q^{\epsilon}:Q\in \mathcal{E}, 0<\epsilon \leq \epsilon_0\}.$$

The above discusses illustrate a uniform characteristic: the time average measure (\ref{0pms}) of transition probability function for each solution weakly converges to an ergodic stationary measure for (\ref{sys:02}) on the attracting domain of the $\omega$-limit set of the orbit for $\Phi^0$ through the same initial point. But case d is quite different. If $y\in {\rm Int}\mathbb{R}_+^3\backslash \overline{L(P)}$, the corresponding time average  measure (\ref{0pms}) of transition probability function has infinite weak limit points, which are not ergodic (see \cite[Theorem 23 and Appendix B]{CDJNZ}).

Applying Theorem \ref{mthm}, we conclude that

\begin{theorem}
\begin{enumerate}
\item[{\rm (i)}] For any equilibrium $Q\in \mathcal{E}$, $\mu^\epsilon_Q(\cdot) \stackrel{w}{\rightarrow}
\delta_Q(\cdot)$ as $\epsilon \rightarrow 0$, which is valid to the cases {\rm a, b, e, f}.
\item[{\rm (ii)}]  For the case {\rm c},  $\nu_h^\epsilon$ converges weakly to the Haar measure on the closed orbit $\Gamma(h)$ as $\epsilon \rightarrow 0,  0<h\leq\frac{1}{27}.$
\item[{\rm (iii)}] For the case {\rm d}, if $\mu^i:= \nu_{y}^{\epsilon^i}\in \mathcal{M}_S(\epsilon_0) ,\ i= 1,2,\cdots,$ satisfying $\epsilon^i\rightarrow 0$ and $\mu^i \stackrel{w}{\rightarrow}\mu$ as $i\rightarrow \infty$, then
    $\mu(\{R_1, R_2, R_3\})=1.$

\end{enumerate}
\end{theorem}

\begin{remark} The Birkhoff center in Example ${\rm(\ref{unds})}$ consists of the origin (saddle) and strongly unstable foci $\{ P^+, P^-\}$. Relatively, the orgin $O$ is more stable than $\{ P^+, P^-\}$. So all limiting measures support at the origin in the case that the diffusion matrix is nondegenerate.  The Birkhoff center for the flow $\Phi^0$ on $\Sigma$ is composed of $P$ (which is strongly repelling on $\Sigma$) and three saddles $\{R_1, R_2, R_3\}$, which are relatively more stable than $P$ and supported by all limiting measures determined by those solutions with initial points on $\Sigma$. Nevertheless, if the drift system is Morse-Smale and the diffusion matrix is positively definite everywhere, then we conjecture that all limiting measures will support at either stable equilibrium or stable periodic orbit.
\end{remark}

Example 4 is complete.
$\qquad\qquad\qquad\qquad\qquad\qquad\qquad\qquad\qquad\qquad\quad\Box$

\section{Applications to SPDEs}

Although our main result can be applied to many SPDEs, we prefer  to apply it to stochastic reaction-diffusion equations, stochastic $2D$ Navier-Stokes equations and stochastic Burgers type equations driven by Brownian motions or  L\'{e}vy process.
\subsection{Stochastic reaction diffusion equation with a polynomial nonlinearity}
Let $\Lambda\subset \mathbb{R}^{n}$ be a bounded domain with smooth boundary $\partial\Lambda$ and
let $g$ be a polynomial of odd degree
with negative leading coefficient
\begin{equation}\label{leacoet}
  g(u)=\displaystyle\sum_{i=0}^{2k-1}a_{i}u^{i}, \quad a_{2k-1}<0.
\end{equation}
Consider the following initial boundary value problem:
\begin{equation}
    \label{nonpoly}
    \left\{
        \begin{array}{l}
        \frac{\partial u}{\partial t}= \Delta u + g(u),\ x\in\Lambda,\ t>0, \\
        u(x,t)=0, \ x\in\partial\Lambda,\ t\geq0, \\
        u(x,0)=\varphi\in L^2(\Lambda).
        \end{array}
    \right.
\end{equation}
Then by Temam \cite[p.84 Theorem III.1.1]{Temam}, equation (\ref{nonpoly}) has a unique solution $u(x,t;\varphi)\in L^2(\Lambda)$.
Therefore we can define a semigroup $T(t)$ on $L^2(\Lambda)$: $T(t):\varphi\in L^2(\Lambda)\mapsto u(t;\varphi)\in H^{1}_{0}(\Lambda)$. Thanks to \cite[p.88 Theorem III.1.2]{Temam}, equation (\ref{nonpoly}) has a global attractor $\mathcal{A}_{L^{2}(\Lambda)}$ which is bounded in $H^{1}_{0}(\Lambda)$, compact and connected in $L^{2}(\Lambda)$. We have that for any $\varphi \in H^{1}_{0}(\Lambda)$, $\{u(x,t;\varphi)\}_{t\geq 0}$ is bounded in $H_{0}^{1}(\Lambda)$.
Furthermore, since $-\Delta$ has compact resolvent,
trajectory $\{u(x,t;\varphi)\}_{t\geq 0}$ has a compact closure.
And let
$$ V(\varphi)=\int_{\Lambda}\big(\frac{1}{2}|\nabla \varphi|^{2}-G(\varphi)\big)dx, \ G(u)=\int_{0}^{u}g(\xi)d\xi,  \quad \varphi\in H_{0}^{1}(\Lambda),  $$
then
\begin{align*}
  \frac{d}{dt}V(u(x,t;\varphi))
  &=\frac{d}{dt}\int_{\Lambda}\big( \frac{1}{2}|\nabla u|^{2}-G(u) \big)dx\\
  &=\int_{\Lambda}\big( \frac{1}{2}2\nabla u \cdot \nabla\frac{\partial u}{\partial t} -g(u)\frac{\partial u}{\partial t} \big)dx\\
  &=-\int_{\Lambda}\big((\Delta u+g(u)) \frac{\partial u}{\partial t}\big)dx\\
  &=-\int_{\Lambda}\big(\frac{\partial u}{\partial t}\big)^{2}dx\leq 0
\end{align*}
where we have used the Green formula and boundary condition in the third equality.
Finally, by LaSalle's invariance principle, the $\omega$-limit set $\omega(\varphi)$ is contained in
the equilibrium set of $T(t)$ for any $\varphi \in H_{0}^{1}(\Lambda)$, this is, $u(x,t;\varphi)=\varphi$
satisfies the following
boundary value problem
\begin{equation*}
\left\{
        \begin{array}{l}
        \Delta \varphi + g(\varphi)=0,\ x\in\Lambda, \\
       \varphi(x)=0, \ x\in\partial\Lambda.
        \end{array}
    \right.
\end{equation*}
This implies that all solutions for (\ref{nonpoly}) are convergent to the equilibrium points.

\begin{proposition}\label{rdb}
The Birkhoff center for ${\rm(\ref{nonpoly})}$ is the equilibrium set $\mathcal{E}$.
\end{proposition}

Now let us consider the noise perturbation system of reaction diffusion equation, such as (\ref{nonpoly}).

Consider the following stochastic reaction diffusion  equation in $\Lambda$ with Dirichlet boundary conditions:
\begin{equation}\label{eq heat}
\left\{
\begin{aligned}
&dX(t,x)=\nu \Delta X(t,x)dt+g(x,X(t,x))dt+\epsilon \sigma(x, X(t,x))d W(t), \\
&X(t,x)=0,\ x\in \partial \Lambda,\ t>0, \\
&X(0)=h\in L^2(\Lambda).
\end{aligned}
\right.
\end{equation}
Here $\nu>0$, $g: \Lambda\times\mathbb{R}\rightarrow\mathbb{R}$ and $\sigma:\Lambda\times \mathbb{R}\rightarrow l^2$
are two measurable functions. $W(t)=(W_k(t))_{k\in\mathbb{N}}$ is a sequence of independent one dimensional standard Brownian motions
on some filtered probability space $(\Omega,\mathcal{F},\{\mathcal{F}_t\}_{t\geq0},\mathbb{P})$.

For $p\geq 1$, let $L^p(\Lambda)$ be the usual $L^p$-space over $\Lambda$ with the standard norm $\|\cdot\|_p$. For $m\in\mathbb{N}$,
let $\mathbb{H}^m_0(\Lambda)$ be the usual $m$-order Sobolev space over $\Lambda$ with Dirichlet boundary conditions, and its norm is
denoted by $\|\cdot\|_{2,m}$. Denote $\mathbb{H}^{-m}(\Lambda)$ be the dual space of $\mathbb{H}^m_0(\Lambda)$. Notice that the following Poincar\'{e}
inequality holds: for some $\lambda_{\Lambda}>0$,
$$
\lambda_{\Lambda}\|u\|^2_2 \leq \|\nabla u\|^2_2,\ \ u\in\mathbb{H}^1_0(\Lambda).
$$
Let $V := \mathbb{H}^1_0(\Lambda)$ and denote by $\|u\|_V := \|\nabla u\|_2$.

Now we identify the Hilbert space $H:=L^2(\Lambda)$ with itself by the Riesz representation, and set for $q\geq2$,
$$
V_q:=\mathbb{H}^1_0(\Lambda)\cap L^q(\Lambda),\ \ \ \ \ V^*_q=\mathbb{H}^{-1}(\Lambda)+L^{q^*}(\Lambda),
$$
where $q^*:=q/(q-1)$. For any $u\in V_q$ and $w=w_1+w_2\in\mathbb{H}^{-1}(\Lambda)+L^{q^*}(\Lambda)$, we have
$$
\langle u,w\rangle_{V_q,  V^*_q}=\langle u,w_1\rangle_{\mathbb{H}^1_0,\mathbb{H}^{-1}}+\langle u,w_2\rangle_{L^q(\Lambda),L^{q^*}(\Lambda)}.
$$
In what follows, we consider the evolution triple
$$
V_q\subset H\subset V^*_q.
$$
Assume that
\begin{itemize}
\item[(C1)] There exist $q\geq 2$, $c_i>0, i=1,2,3,4$ and $h_1\in L^1(\Lambda)$, $h_2\in L^{q^*}(\Lambda)$ such that for all $u,u'\in\mathbb{R}$ and $x\in\Lambda$,
    \begin{eqnarray*}
       (u-u')(g(x,u)-g(x,u'))\leq c_1|u-u'|^2,
    \end{eqnarray*}
    \begin{eqnarray*}
       u g(x,u)\leq -c_2|u|^q+c_3|u|^2+h_1(x),
    \end{eqnarray*}
    \begin{eqnarray*}
       |g(x,u)|\leq c_4|u|^{q-1}+h_2(x),
    \end{eqnarray*}
and the mapping $u\mapsto g(x,u)$ is continuous.

\item[(C2)] There exist $c_5,c_6>0$ and $h_3\in L^1(\Lambda)$ such that for all $u,u'\in \mathbb{R}$ and $x\in \Lambda$,
\begin{eqnarray*}
 \|\sigma(x,u)-\sigma(x,u')\|^2_{l^2}\leq c_5|u-u'|^2,
\end{eqnarray*}
and
\begin{eqnarray*}
\|\sigma(x,u)\|^2_{l^2}\leq c_6|u|^2+h_3(x),
\end{eqnarray*}

\end{itemize}
where $l^2$ be the Hilbert space of all square summable sequences of real
numbers.
By Theorem 3.2 in \cite{LXZ} or Theorem 3.6 in \cite{ZhangXC}, under (C1)-(C2), for any $p\geq 1$ and $h\in L^{2p}(\Omega, \mathcal{F}_0,\mathbb{P};L^2(\Lambda))$,
there exists a unique $L^2(\Lambda)$-valued $\mathcal{F}_t$-adapted process $X^{\epsilon,h}$ with
$$
X^{\epsilon,h}\in C_{loc}([0,\infty),L^2(\Lambda))\cap L^2_{loc}([0,\infty),V)\cap L^q_{loc}([0,\infty),L^q(\Lambda)),\ \ \mathbb{P}\text{-a.s.}
$$
and equation (\ref{eq heat}) holds in $V^*_q$. Moreover, we can obtain

\begin{lemma}\label{thm heat 1}
Assume {\rm(C1)-(C2)} hold, and $q>2$ or $q=2, \nu+\frac{c_2-c_3}{\lambda_{\Lambda}}>0$. Then
there exists $\epsilon_0>0$ such that for any $\epsilon\in(0,\epsilon_0]$,
\begin{equation}\label{sheatesti}
  \mathbb{E}\big(\sup_{t\in[0,T]}\|X^{\epsilon,h}(t)\|^2_2\big)+\mathbb{E}\big(\int_0^T\|X^{\epsilon,h}(t)\|^2_{V}dt\big)
  \leq C\big(\mathbb{E}(\|h\|^2_2)+T\big).
\end{equation}
\end{lemma}

\begin{remark}\label{Rem 6}
By Fubini's Theorem and {\rm(\ref{sheatesti})}, we have
$$
\int_0^T\mathbb{E}\Big(\|X^{\epsilon,h}(t)\|^2_{V}\Big)dt<\infty.
$$
This implies that there exists $\mathcal{T}_0\in\mathcal{B}([0,T])$ with zero Lebesgue measure such that
$$
\mathbb{E}\Big(\|X^{\epsilon,h}(t)\|^2_{V}\Big)<\infty,\ \ \ t\in[0,T]\setminus \mathcal{T}_0.
$$
Hence for any $t\in[0,T]\setminus \mathcal{T}_0$, there exists $\Omega_t\in\mathcal{F}$ with $\mathbb{P}(\Omega_t)=1$ such that
$$
X^{\epsilon,h}(t,\omega)\in V,\ \ \omega\in \Omega_t.
$$
\end{remark}

\begin{proof} Denote by $(L_2(l^2,H),\|\cdot\|_{l^2,H})$ the Hilbert space of all Hilbert-Schmidt operators for $l^2$ to $H$.

By It\^{o}'s formula,
\begin{align*}
&\|X^{\epsilon,h}(t)\|^2_2+2\nu\int_0^t\|X^{\epsilon,h}(s)\|^2_{V}ds\ \\
=&\|h\|^2_2+2\int_0^t\langle X^{\epsilon,h}(s),g(x, X^{\epsilon,h}(s))\rangle_{H,H}ds\\
&+2\epsilon\int_0^t\langle X^{\epsilon,h}(s),\sigma(x, X^{\epsilon,h}(s))dW(s)\rangle_{H,H}\\
&+\epsilon^2\int_0^t\|\sigma(x, X^{\epsilon,h}(s))\|^2_{l^2,H}ds\\
:=&I^\epsilon_1+I^\epsilon_2(t)+I^\epsilon_3(t)+I^\epsilon_4(t).
\end{align*}

For $I^\epsilon_3$, by (C2),
\begin{equation}\label{eq heat I3}
  \begin{split}
&\mathbb{E}\Big(\sup_{t\in[0,T]}|I^\epsilon_3(t)|\Big)\\
\leq&
2\epsilon \mathbb{E}\Big(\int_0^T\|X^{\epsilon,h}(t)\|^2_2\|\sigma(x, X^{\epsilon,h}(t))\|^2_{l^2,H}dt\Big)^{1/2}\\
\leq&
2\epsilon\mathbb{E}\Big(\sup_{t\in[0,T]}\|X^{\epsilon,h}(t)\|^2_2\Big)
+2\epsilon\mathbb{E}\Big(\int_0^T\|\sigma(x, X^{\epsilon,h}(t))\|^2_{l^2,H}dt\Big)\\
\leq&
2\epsilon\mathbb{E}\Big(\sup_{t\in[0,T]}\|X^{\epsilon,h}(t)\|^2_2\Big)+
2\epsilon c_6 \mathbb{E}\Big(\int_0^T\|X^{\epsilon,h}(t)\|^2_2dt\Big)
+2\epsilon T \|h_3\|_1 \\
\leq&
2\epsilon\mathbb{E}\Big(\sup_{t\in[0,T]}\|X^{\epsilon,h}(t)\|^2_2\Big)+
\frac{2\epsilon c_6}{\lambda_{\Lambda}} \mathbb{E}\Big(\int_0^T\|X^{\epsilon,h}(t)\|^2_{V}dt\Big)
+2\epsilon T \|h_3\|_1.
  \end{split}
\end{equation}

For $I^\epsilon_4$, applying (C2) again, we have
\begin{eqnarray}\label{eq heat I4}
\mathbb{E}\Big(I^\epsilon_4(T)\Big)
\leq
\frac{\epsilon^2 c_6}{\lambda_{{\Lambda}}} \mathbb{E}\Big(\int_0^T\|X^{\epsilon,h}(s)\|^2_{V}ds\Big)
+
\epsilon^2 T \|h_3\|_1.
\end{eqnarray}

 $I^\epsilon_2$ will be estimated according to two cases.
\begin{itemize}
\item[(1)] The case $q>2$ or $q=2, c_2-c_3\geq0$.

By (C1), it is easy to see that there exist $\kappa\geq 0$ and $\tilde{h}_1\in L^1(\Lambda)$ such that
$$
   u g(x, u)\leq -\kappa |u|^q+\tilde{h}_1(x).
$$
Hence
\begin{equation}\label{eq heat I21}
I^\epsilon_2(t)\leq 2\|\tilde{h}_1\|^1_1 t.
\end{equation}

\item[(2)] The case $q=2, c_2-c_3<0, \nu+\frac{c_2-c_3}{\lambda_{\Lambda}}>0$.

    By (C1),
    \begin{equation}\label{eq heat I22}
      \begin{split}
       I^\epsilon_2(t) \leq&
        2(c_3-c_2)\int_0^t\|X^{\epsilon,h}(s)\|^2_2ds
        +2t\|h_1\|_1  \\
        \leq&\frac{2(c_3-c_2)}{\lambda_{\Lambda}}\int_0^t\|X^{\epsilon,h}(s)\|^2_{V}ds
        +2t\|h_1\|_1.
      \end{split}
    \end{equation}
\end{itemize}

Combing (\ref{eq heat I21}), (\ref{eq heat I22}), (\ref{eq heat I3}) and (\ref{eq heat I4}), we conclude that there exist $\epsilon_0,\eta_1,\eta_2>0$ such that for $0<\epsilon \leq\epsilon_0$,
$$
\frac{1}{2}\mathbb{E}\Big(\sup_{t\in[0,T]}\|X^{\epsilon,h}(t)\|^2_2\Big)+\eta_1\mathbb{E}\Big(\int_0^T\|X^{\epsilon,h}(t)\|^2_{V}dt\Big)
\leq
\mathbb{E}(\|h\|^2_2)+\eta_2 T.
$$
This completes the proof.
\end{proof}

\

Denote by $X^{0,h}$ the solution for (\ref{eq heat}) when $\epsilon=0$. Applying the results
in Lemma \ref{thm heat 1}
and It\^{o}'s formula to $\|X^{\epsilon,h}(t)-X^{0,h}(t)\|^2_2$, we can obtain

\
\begin{theorem}\label{thm heat 2}
If the assumptions of Lemma {\rm\ref{thm heat 1}} hold, then

\begin{itemize}
\item[{\rm(1)}] For any $\widetilde{M}>0$, $\delta>0$ and $t\geq 0$,
\begin{eqnarray*}
\lim_{\epsilon\rightarrow0}\sup_{\|h\|^2_V\leq \widetilde{M}}\mathbb{P}\Big(\|X^{\epsilon,h}_t-X^{0,h}_t\|^2_2\geq \delta\Big)=0.
\end{eqnarray*}

\item[{\rm(2)}] There exists at least one stationary measure $\mu^{\epsilon,h}$ for $X^{\epsilon,h}$.

\item[{\rm(3)}] For $\epsilon\in(0,\epsilon_0]$, denote by $\{\mu^{\epsilon}_{i_\epsilon},\ i_\epsilon\in I_\epsilon\}$ all stationary measures for the semigroup $\{P^{\epsilon}_t\}_{t\geq 0}$. Then $\{\mu^{\epsilon}_{i_\epsilon},\ i_\epsilon\in I_\epsilon, \epsilon\in(0,\epsilon_0]\}$ is tight.
\end{itemize}

\end{theorem}

\begin{proof}
Repeating the proof of Lemma \ref{thm heat 1}, we can get that there is a positive constant $C$ such that
\begin{eqnarray}\label{eq heat star1}
\sup_{t\in[0,T]}\|X^{0,h}(t)\|^2_2+\int_0^T\|X^{0,h}(t)\|^2_{V}dt
\leq
C(\|h\|^2_2+ T),
\end{eqnarray}
where $C$ can be chosen to be the same as (\ref{sheatesti}).

Define the stopping time $\tau_N=\inf\{s\in[0,T],\,\|X^{\epsilon,h}_s\|^2_2\geq N\}\wedge T$.
Applying It\^{o}'s formula to $\|X^{\epsilon,h}(t)-X^{0,h}(t)\|^2_2$, we have
\begin{equation}\label{eq heat star 2}
  \begin{split}
  &\|X^{\epsilon,h}(t\wedge\tau_N)-X^{0,h}(t\wedge\tau_N)\|^2_2\\
  &+2\nu\int_0^{t\wedge\tau_N}\|X^{\epsilon,h}(s)-X^{0,h}(s)\|^2_{V}ds\\
=&2\int_0^{t\wedge\tau_N}\langle X^{\epsilon,h}(s)-X^{0,h}(s),g(x,X^{\epsilon,h}(s))-g(x,X^{0,h}(s))\rangle_{H,H}ds\\
      &+2\epsilon\int_0^{t\wedge\tau_N}\langle X^{\epsilon,h}(s)-X^{0,h}(s),\sigma(x,X^{\epsilon,h}(s))\rangle_{H,H}dW(s)\\
      &+\epsilon^2\int_0^{t\wedge\tau_N}\|\sigma(x,X^{\epsilon,h}(s))\|^2_{l^2,H}ds.
  \end{split}
\end{equation}
By the definition of $\tau_N$ and (\ref{eq heat star1}), we know that
$$
\Big\{2\epsilon\int_0^{t\wedge\tau_N}\langle X^{\epsilon,h}(s)-X^{0,h}(s),\sigma(x,X^{\epsilon,h}(s))\rangle_{H,H}dW(s)\Big\}_{t\in[0,T]}
$$
is a martingale.
Taking expectation of (\ref{eq heat star 2}) and by the Assumption (C1), (\ref{eq heat I4}) and (\ref{sheatesti}), we have
\begin{align*}\label{eq heat star3}
&\mathbb{E}\|X^{\epsilon,h}(t\wedge\tau_N)-X^{0,h}(t\wedge\tau_N)\|^2_2+2\nu\mathbb{E}\int_0^{t\wedge\tau_N}\|X^{\epsilon,h}(s)-X^{0,h}(s)\|^2_{V}ds\\
\leq&
2c_1\mathbb{E}\int_0^{t\wedge\tau_N}\|X^{\epsilon,h}(s)-X^{0,h}(s)\|^2_2ds+
\epsilon^2\mathbb{E}\int_0^{t\wedge\tau_N}\|\sigma(x,X^{\epsilon,h}(s))\|^2_{l^2,H}ds\\
\leq&2c_1\int_0^{t}\mathbb{E}\|X^{\epsilon,h}(s\wedge\tau_N)-X^{0,h}(s\wedge\tau_N)\|^2_2ds
+\epsilon^2\mathbb{E}\int_0^{T}\|\sigma(x,X^{\epsilon,h}(s))\|^2_{l^2,H}ds\\
\leq&
2c_1\int_0^{t}\mathbb{E}\|X^{\epsilon,h}(s\wedge\tau_N)-X^{0,h}(s\wedge\tau_N)\|^2_2ds
+\frac{\epsilon^2 c_6}{\lambda_{{\Lambda}}} \mathbb{E}\Big(\int_0^T\|X^{\epsilon,h}(s)\|^2_{V}ds\Big)
+\epsilon^2 T \|h_3\|_1\\
\leq&2c_1\int_0^{t}\mathbb{E}\|X^{\epsilon,h}(s\wedge\tau_N)-X^{0,h}(s\wedge\tau_N)\|^2_2ds
+\frac{\epsilon^2 c_6C}{\lambda_{{\Lambda}}}\Big(\mathbb{E}(\|h\|^2_2+T\Big)+\epsilon^2 T \|h_3\|_1.
\end{align*}
By Gronwall's lemma, for any $t\in [0,T]$ and integer $N$,
\begin{align*}
&\mathbb{E}\|X^{\epsilon,h}(t\wedge\tau_N)-X^{0,h}(t\wedge\tau_N)\|^2_2\\
\leq&\epsilon^2\Big(\frac{c_6C}{\lambda_{{\Lambda}}}\big(\mathbb{E}(\|h\|^2_2+T\big)+ T \|h_3\|_1\Big)\cdot e^{2c_1T}.
\end{align*}
Letting $N\rightarrow\infty$ and using the Fatou lemma, we get that for any $t\in [0,T]$,
\begin{align*}
&\mathbb{E}\|X^{\epsilon,h}(t)-X^{0,h}(t)\|^2_2\\
\leq&
\epsilon^2\Big(\frac{c_6C}{\lambda_{{\Lambda}}}\big(\mathbb{E}(\|h\|^2_2+T\big)+ T \|h_3\|_1\Big)\cdot e^{2c_1T}.
\end{align*}
Hence,  for any $t\in [0,T]$,
\begin{align*}
&\lim_{\epsilon\rightarrow0}\sup_{\|h\|^2_V\leq \widetilde{M}}\mathbb{P}\Big(\|X^{\epsilon,h}_t-X^{0,h}_t\|^2_2\geq \delta\Big)\\
\leq&\lim_{\epsilon\rightarrow0}\frac{\epsilon^2}{\delta}\Big(\frac{c_6C}{\lambda_{{\Lambda}}}\big(\widetilde{M}+T\big)+ T \|h_3\|_1\Big)\cdot e^{2c_1T}
  =0.
\end{align*}
We have obtained (1).

(2) Utilizing Lemma \ref{thm heat 1}, we deduce that for any $L>0$,
\begin{equation}\label{eq heat star 05}
\begin{split}
&\limsup_{t\rightarrow\infty}\frac{1}{t}\int_0^t\mathbb{P}\Big(\|X^{\epsilon,h}_s\|_{V}\geq L\Big)ds\\
\leq&
\frac{1}{L^2}\limsup_{t\rightarrow\infty}\frac{1}{t}\int_0^t\mathbb{E}\Big(\|X^{\epsilon,h}_s\|^2_{V}\Big)ds\\
\leq&
\frac{C}{L^2}.
\end{split}
\end{equation}
Notice that the embedding $V\subset H$ is compact. Then there exists at least one stationary measure $\mu^{\epsilon,h}$ for $X^{\epsilon,h}$ by the Prohorov theorem.

\vskip 0.3cm

(3) For $\epsilon\in(0,\epsilon_0]$, choose $\mu^\epsilon\in\{\mu^{\epsilon}_{i_\epsilon},\ i_\epsilon\in I_\epsilon\}$, by the definition of stationary measure, we have
$$
\mu^\epsilon(\|h\|_{V}\geq L)=\int_{H}\mathbb{P}\Big(\|X^{\epsilon,h}_s\|_{V}\geq L\Big)\mu^{\epsilon}(dh),\ \ \ \ \forall s\geq 0,
$$
hence, by Fubini's theorem, Fatou's lemma and (\ref{eq heat star 05})
\begin{align*}
\mu^\epsilon(\|h\|_V\geq L)
\leq&
\limsup_{t\rightarrow\infty}\frac{1}{t}\int_0^t\int_{H}\mathbb{P}\Big(\|X^{\epsilon,h}_s\|_{V}\geq L\Big)\mu^{\epsilon}(dh)ds\\
\leq&
\int_{H}\limsup_{t\rightarrow\infty}\frac{1}{t}\int_0^t\mathbb{P}\Big(\|X^{\epsilon,h}_s\|_{V}\geq L\Big)ds\mu^{\epsilon}(dh)\\
\leq&
\frac{C}{L^2}.
\end{align*}
It follows from the fact that $\{h\in V: \|h\|_{V}\leq L\}$ is compact in $H$ that $\{\mu^{\epsilon}_{i_\epsilon},\ i_\epsilon\in I_\epsilon, \epsilon\in(0,\epsilon_0]\}$ is tight.
\end{proof}

Using the Young inequality,  we can get that the  polynomial $g$ in (\ref{leacoet}) satisfies condition (C1).
This fact, together with Theorem \ref{thm heat 2} and Theorem \ref{mthm}, implies the main result in this subsection.

\
\begin{theorem}
Assume that $\nu=1$, $g$ is given in {\rm(\ref{leacoet})} and $\sigma$ satisfies condition {\rm(C2)}, then any limiting measures
of stationary measures for {\rm(\ref{eq heat})} are supported in the set of equilibrium points $\mathcal{E}$.
\end{theorem}

\vskip 0.3cm
Specially, we consider one dimensional cubic reaction-diffusion equation:
\begin{equation}
    \label{dcubic}
    \left\{
        \begin{array}{l}
        \frac{\partial u}{\partial t}= \frac{\partial^2 u}{\partial x^2} + \lambda^2u(1-u^2),\ 0<x<1,\ t>0, \\
        u(0,t)=u(1,t)=0, \\
        \end{array}
    \right.
\end{equation}
where $\lambda$ is a positive parameter.

\

\begin{proposition}{\rm \cite[Theorem 4.3.12]{Hale}}\label{cubicA}
If $\lambda\in(n\pi,(n+1)\pi]$ where $n$ is an integer, then there are $2n+1$ equilibrium points
$\phi_0=0, \phi_j^+, \phi_j^-, j=1,2,...,n$ of {\rm (\ref{dcubic})}, the points $\phi_j^{\pm}$ are hyperbolic with
{\rm dim}$W^u(\phi_j^{\pm})=j-1, j=1,2,...,n.$ If $\lambda\in(n\pi,(n+1)\pi)$, then  $\phi_0=0$ is hyperbolic,
{\rm dim}$W^u(0)=n,$ and the attractor $A_\lambda$ is given
$$A_\lambda=W^u(0)\bigcup\Big(\bigcup_{j=1}^nW^u(\phi_j^{\pm})\Big).$$
Here $W^u(\phi)$ denotes the unstable manifold for the equilibrium $\phi$. Hence the Birkhoff center for the semiflow
$\Phi$ is given by
$$B(\Phi)=\{0\}\bigcup \{\phi_j^+, \phi_j^-: j=1,2,...,n\}.$$
\end{proposition}

Now we consider the perturbed equation driven by Brownian motion:
\begin{equation}\label{scubic}
\left\{
\begin{array}{rl}
dX(t,x)=& \Delta X(t,x)dt+\lambda X(t,x)\big(1-(X(t,x))^2\big)dt\\
&+\epsilon \sigma(x, X(t,x))d W(t), \\
X(t,0)=&X(t,1)=0, \ t>0, \\
X(0)=&h\in L^2([0,1]).
\end{array}
\right.
\end{equation}

It is easy to see that the cubic polynomial $g(u) = u(1-u^2)$ satisfies that condition (C1). Together with Theorem \ref{thm heat 2}, Proposition \ref{cubicA} and Theorem \ref{mthm}, we have

\begin{theorem}\label{scubicthm}
Assume that $\sigma$ in {\rm(\ref{scubic})} satisfies condition {\rm (C2)} and  $\lambda\in(n\pi,(n+1)\pi]$.
If $\mu$ is any weak limit
point of $\{\mu^{\epsilon_i}\}$ as $\epsilon_i \rightarrow 0$, then $\mu$ is supported in the set of equilibrium points
$$\{0\}\bigcup \{\phi_j^+, \phi_j^-: j=1,2,...,n\}.  $$
\end{theorem}

\begin{remark}\label{Re8} Let $\lambda\in(n\pi,(n+1)\pi]$ and $\phi$, for example, $\phi_j^+$, be an equilibrium of {\rm (\ref{dcubic})}. Then  denote by
$$m_j={\rm min}\{\phi_j^+(x):x\in[0,1]\}\ {\rm and}\ M_j={\rm max}\{\phi_j^+(x):x\in[0,1]\}$$
the minimal and maximal values of $\phi_j^+$, respectively. As constructed in Proposition {\rm\ref{procons}},
one can construct diffusion term $\sigma$ such that  $\sigma(u)=0$ on $[m_j, M_j]$ and {\rm(C2)} is satisfied.
Thus, {\rm(\ref{scubic})} has a sequence of stationary measures whose weak limit supports at $\phi_j^+$. This shows that the result of Theorem {\rm\ref{scubicthm}} is the best as a general of result.
\end{remark}

\subsection {Stochastic 2D Navier-Stokes equation driven by L\'evy noise}
Let $D$ be an open bounded domain with smooth boundary $\partial D$ in $\mathbb{R}^2$. Denote by $u$ and $p$ the velocity and the
pressure fields, respectively. The Navier-Stokes equation is given as follows:
\begin{equation}\label{NS01}
\left\{
\begin{array}{l}
        \partial_tu-\nu\Delta u+(u\cdot\nabla) u+\nabla p=h\ \ \ \text{ in }\ \ \ D\times[0,T], \\
        \nabla\cdot u=0\ \ \ \text{in}\ \ \ D\times[0,T], \\
\end{array}
\right.
\end{equation}
with the conditions
\begin{equation}\label{condition01}
\left\{
 \begin{array}{l}
  u(\cdot,t)=0\ \ \ \ \ \text{in}\ \ \ \partial D\times[0,T]; \\
  u(0)=x\in L^2(D),
 \end{array}
\right.
\end{equation}
where $\nu>0$ is the viscosity, $h$ stands for the external force acting on the fluid.

Define
$$
V=\Big\{v\in W^{1,2}_0(D,\mathbb{R}^2):\ \nabla\cdot v=0\ a.e.\ in\ D\Big\},\ \ \ \ \ \ \|v\|_V:=\Big(\int_D\big(|\nabla v_1|^2+|\nabla v_2|^2\big)dx\Big)^{1/2},
$$

and let $H$ be the closure of $V$ in the following norm
$$
\|v\|_H:=\Big(\int_D|v|^2dx\Big)^{1/2}.
$$
By the Poincar\'{e}
inequality, we have the Gelfand triples: $V\subset H\cong H^*\subset V^*$.

\vskip 0.3cm
Define the Stokes operator $A$ in $H$ by
$$ Au=P_H\Delta u,\ \forall u\in D(A)=W^{2,2}(D,\mathbb{R}^2)\cap V,$$
 where the linear
operator $P_H$ (Helmhotz-Hodge projection) is the projection operator from $L^2(D,\mathbb{R}^2)$ into $H$. Since $V$ coincides with $D(A^{1/2})$,
we can endow $V$ with the norm $\|u\|_V=\|A^{1/2}u\|_H$. Because $A$ is a positive selfadjoint operator with compact resolvent, there is a complete orthonormal system $\{e_1,e_2,\cdots\}$  of eigenvectors of $A$ in $V$, with corresponding eigenvalues
$0<\lambda_1\leq\lambda_2\leq\cdots\rightarrow\infty$ ($Ae_i=\lambda_i e_i$), and
$$
\|u\|_V^2\geq \lambda_1\|u\|_H^2,\ \ \ \ u\in V.
$$
It is well known that the Navier-Stokes equation can be reformulated as follows:
\begin{equation}\label{dNSe}
u'=\nu Au+F(u)+h_0(u),\ \ \ u(0)=u_0\in H,
\end{equation}
and
$$
F:\mathcal{D}_F\subset H\times V\rightarrow H,\ \ \ \ F(u,v)=-P_H[(u\cdot\nabla) v],\ \ \ \ F(u)=F(u,u),\ \ h_0=P_Hh.
$$
One can show that the following mappings
$$
A:V\rightarrow V^*,\ \ \ \ \ F:V\times V\rightarrow V^*
$$
are well defined, and
$$
\langle F(u,v),w\rangle_{V^*,V}=-\langle F(u,w),v\rangle_{V^*,V},\ \ \ \ \ \langle F(u,v),v\rangle_{V^*,V}=0,\ \ \ u,v,w\in V.
$$

\textbf{Assumption 1}: For the mapping $h_0:V\rightarrow V^*$ there exists $\vartheta_0\in(0,\nu)$ such that
$$
\|h_0(v)-h_0(w)\|_{V^*}\leq \vartheta_0\|v-w\|_V,\ \ \ \ \ v,w\in V.
$$

\
Define
$$
\mathcal{A}(v)=\nu Av+F(v)+h_0(v).
$$
 Since $A$ is linear, $F$ is bilinear and the Assumption 1 holds, we can easily get that

 (H1) The map $s\mapsto \langle \mathcal{A}(v_1+sv_2),v\rangle_{V^*,V}$ is continuous on $\mathbb{R}$  for all $v,v_1,v_2\in V$.

\vskip 0.3cm
As estimated in \cite [Lemma 2.3]{Menaldi} or \cite [p.293]{BLZ}, we have that for any $\eta >0$,
\begin{align*}
  \langle F(u)-F(v),u-v\rangle_{V^*,V}
\leq&
   2\|u-v\|_V^{3/2}\|u-v\|^{1/2}_H\|v\|_{L^4(D,\mathbb{R}^2)}\\
\leq&
   \eta \|u-v\|^2_V+\frac{27}{16\eta^3} \|v\|^4_{L^4(D,\mathbb{R}^2)}\|u-v\|^{2}_H,\ \ \ \ u,v\in V.
\end{align*}
Take $\eta = \frac{\nu-\vartheta_0}{2}$. Then
\begin{align*}
  &2\langle \mathcal{A}(u)-\mathcal{A}(v),u-v\rangle_{V^*,V}\\
\leq&
  -2\nu \|u-v\|^2_V+2\eta\|u-v\|^2_V+\frac{27}{8\eta^3}\|v\|^4_{L^4(D,\mathbb{R}^2)}\|u-v\|^{2}_H\\
  &+2\|h_0(u)-h_0(v)\|_{V^*}\|u-v\|_V\\
\leq&-(\nu -\vartheta_0)\|u-v\|^2_V+\frac{27}{(\nu-\vartheta_0)^3} \|v\|^4_{L^4(D,\mathbb{R}^2)}\|u-v\|^{2}_H.
\end{align*}
Thus we get that
\vskip 0.3cm
(H2)\ \ $ 2\langle \mathcal{A}(u)-\mathcal{A}(v),u-v\rangle_{V^*,V} \leq \frac{27}{(\nu-\vartheta_0)^3} \|v\|^4_{L^4(D,\mathbb{R}^2)}\|u-v\|^{2}_H$
for all $ u,v\in V$.
\vskip 0.3cm
Similarly, we can prove that
 \vskip 0.3cm
(H3)\ \  $ 2\langle \mathcal{A}(v),v\rangle_{V^*,V} + (\nu-\vartheta_0)\|v\|^2_V \leq (\nu-\vartheta_0)^{-1}\|h_0(0)\|^2_{V^*}$ for all $ v\in V.$

\vskip 0.3cm
It follows from  (2.91) in \cite [p.291]{BLZ} that for all $ u,v\in V$,
$$|\langle F(v),u\rangle_{V^*,V}| \leq 2\|v\|_{L^4(D,\mathbb{R}^2)}\|u\|_V.$$
An easy calculation deduces that
$$ \|\mathcal{A}(v)\|_{V^*}\leq \nu \|v\|_V + 2\|v\|_{L^4(D,\mathbb{R}^2)}+\vartheta_0 \|v\|_V+\|h_0(0)\|_{V^*}$$
Applying \cite [Lemma 2.1]{Menaldi}, we have
\begin{equation}\label{l4}
\|v\|^4_{L^4(D,\mathbb{R}^2)}\leq 2\|v\|^2_H\|v\|^2_V,\ \ v\in V.
\end{equation}
Therefore, we obtain that there exists a positive constant $C$ such that
 \vskip 0.3cm
(H4)  $\|\mathcal{A}(v)\|^2_{V^*} \leq C(1+\|v\|^2_V)(1+\|v\|^2_H),\ \ v\in V.$
 \vskip 0.3cm

Now we present an attractor result for the deterministic system (\ref{dNSe}).
\begin{proposition}\label{dNSG}
If the Assumption {\rm 1} holds, then the deterministic system {\rm(\ref{dNSe})} generates a dynamical system
$\{\Phi(t)\}_{t\geq 0}$ which possesses a connected global attractor $\mathcal{A}_H$. Besides, $\Phi(t)|_{\mathcal{A}_H}$ is a group.
\end{proposition}

The proof is contained in  \cite[Theorem IV.8.2 and Theorem III.2.2]{Temam} or  \cite[Theorem 4.4.5]{Hale}.

Consider the stochastic 2D Navier-Stokes equation driven by L\'evy noise:
\begin{equation}\label{exa NS 01}
\begin{split}
   dX_t^{\epsilon,h}=&\big(\nu AX_t^{\epsilon,h}+F(X_t^{\epsilon,h})+h_0(X_t^{\epsilon,h})\big)dt\\
   &+\epsilon B(X_t^{\epsilon,h})dW_t+\epsilon\int_Z f(X_{t-}^{\epsilon,h},z)\tilde{N}(dt,dz),
\end{split}
\end{equation}
with a deterministic initial value $X_0^{\epsilon,h}=h\in H$.

Here $\{W_t\}_{t\geq0}$ is a $U$-valued cylindrical Wiener process on the probability space $(\Omega,\mathcal{F},\{\mathcal{F}_t\}_{t\geq 0}, \mathbb{P})$ with $U$ a separable Hilbert space, $(Z,\mathcal{Z})$ a locally compact Polish space, and $\nu$ a $\sigma$-finite measure on $Z$. Set $N$ being a Poisson random measure on $[0,\infty)\times Z$ with the
$\sigma$-finite intensity measure $\lambda\otimes \nu$, where $\lambda$ is the Lebesgue measure on $[0,\infty)$.
$\widetilde{N}([0,t]\times O)=N([0,t]\times O)-t\nu(O)$, $O\in \mathcal{Z}$ with $\nu(O)<\infty$, is the compensated Poisson random measure.
Denote by $(L_2(U,H),\|\cdot\|_2)$ the Hilbert space of all Hilbert-Schmidt
operators from $U$ to $H$.

\textbf{Assumption 2}: Suppose that $B:V\rightarrow L_2(U,H)$, $f:V\times Z\rightarrow H$ satisfy the following conditions:
there is a positive constant $L$ such that for any $v,w\in V$,
\begin{eqnarray*}
    \|B(v)-B(w)\|^2_2+\int_{Z}\|f(v,z)-f(w,z)\|^2_H\nu(dz) &\leq& L\|v-w\|^2_H, \\
     \int_Z\|f(v,z)\|^2_H\nu(dz) &\leq& L(1+\|v\|^2_H), \\
     \int_Z\|f(v,z)\|^4_H\nu(dz) &\leq& L(1+\|v\|^4_H).
\end{eqnarray*}
This implies that
\begin{equation}\label{eq SPDE con 1}
\|B(v)\|^2_2+\int_{Z}\|f(v,z)\|^2_H\nu(dz)
\leq
F+C\|v\|^2_H \leq K(1+\|v\|^2_H),
\end{equation}
where $F,C$ depends on $L$ and $\|B(0)\|_2$ and $K:=$ \rm{max}$\{(\nu-\vartheta_0)^{-1}\|h_0(0)\|^2_{V^*}, F, C\}$.

\vskip 0.3cm

Together with the Assumptions 1 and 2 with (H1)-(H4), all hypotheses in \cite {BLZ} are satisfied. So we have

\begin{proposition}{\rm(\cite[Theorem 1.2]{BLZ})}\label{th SPDE-solution}
Suppose that the Assumptions {\rm 1} and {\rm 2} hold. Then

{\rm(i)} for any $h\in L^4(\Omega,\mathcal{F}_0,\mathbb{P}; H)$, $T>0,$ Eq. {\rm(\ref{exa NS 01})} has a unique solution $\{X^{\epsilon,h}_t\}_{t\in[0,T]}$,

{\rm(ii)} there exists a constant $C_T$ independent of $\epsilon$ and $h$ such that, for $\epsilon\in (0,1]$,
\begin{equation}\label{eq SPDE result 1}
\sup_{t\in[0,T]}\mathbb{E}\|X^{\epsilon,h}_t\|^{4}_H+\mathbb{E}\int_0^T\|X^{\epsilon,h}_t\|^{2}_H\|X^{\epsilon,h}_t\|^{2}_Vdt
\leq
C_T(\mathbb{E}\|h\|^{4}_H+T),
\end{equation}
and
\begin{eqnarray}\label{eq SPDE result 2}
\mathbb{E}\int_0^T \|X^{\epsilon,h}_t\|^{2}_Vdt
<
\infty.
\end{eqnarray}
\end{proposition}

\begin{lemma}\label{SPDE prop} There exist $\eta>0$ and $\epsilon_\eta>0$ such that for any $h\in H$,
\begin{eqnarray*}
    \sup_{s\in[0,T]}\mathbb{E}\|X^{\epsilon,h}_s\|^2_H+\eta\mathbb{E}\int_0^T\|X^{\epsilon,h}_s\|^2_Vds
    \leq
    \|h\|^2_H+(1+\epsilon_\eta^2)KT,\ \ \ \forall \epsilon\in(0,\epsilon_\eta].
\end{eqnarray*}
In particular,
\begin{equation}\label{deterBd}
    \sup_{s\in[0,T]}\|X^{0,h}_s\|^2_H+\eta \int_0^T\|X^{0,h}_s\|^2_Vds
    \leq
    \|h\|^2_H+2KT.
\end{equation}

\end{lemma}

\begin{remark}\label{Rem NS} As in Remark {\rm\ref{Rem 6}}, there exists $\mathcal{T}_0\in\mathcal{B}([0,T])$ with zero Lebesgue measure
and for any $t\in[0,T]\setminus \mathcal{T}_0$, there exists $\Omega_t\in\mathcal{F}$ with $\mathbb{P}(\Omega_t)=1$ such that
$$
X^{\epsilon,h}(t,\omega)\in V,\ \ \omega\in \Omega_t.
$$
\end{remark}

\begin{proof} Using a similar argument as (4.18) in \cite{BLZ}, we have
\begin{equation}\label{eq SPDE es}
  \begin{split}
&\|X^{\epsilon,h}_t\|^2_H\\
=&\|h\|^2_H+2\int_0^t\langle \mathcal{A}(\bar{X}^{\epsilon,h}_s),\bar{X}^{\epsilon,h}_s\rangle_{V^*,V}ds\\
           &+2\epsilon\int_0^t\langle B(\bar{X}^{\epsilon,h}_s)dW_s,\bar{X}^{\epsilon,h}_s\rangle_{H,H}\\
    & +2\epsilon\int_0^t\int_{Z}\langle f(\bar{X}^{\epsilon,h}_{s},z),\bar{X}^{\epsilon,h}_{s}\rangle_{H,H}\widetilde{N}(ds,dz)\\
    & +\epsilon^2\int_0^t\|B(\bar{X}^{\epsilon,h}_s)\|^2_2ds + \epsilon^2 \int_0^t\int_{Z}\|f(\bar{X}^{\epsilon,h}_{s},z)\|^2_HN(ds,dz),
  \end{split}
\end{equation}
where $\bar{X}^{\epsilon,h}$ is any $V$-valued progressively measurable $dt\times\mathbb{P}$ version of $X^{\epsilon,h}$.

Set $\sigma_N=\inf\{s\geq0:\ \|X_s^{\epsilon,h}\|_H\geq N\}\bigwedge T$. By (\ref{eq SPDE result 1}) and (\ref{eq SPDE result 2}),
$$
L_N(t)=2\epsilon\int_0^{t\wedge \sigma_N}\langle B(\bar{X}^{\epsilon,h}_s)dW_s,\bar{X}^{\epsilon,h}_s\rangle_{H,H}
        +
       2\epsilon\int_0^{t\wedge \sigma_N}\int_{Z}\langle f(\bar{X}^{\epsilon,h}_{s},z),\bar{X}^{\epsilon,h}_{s}\rangle_{H,H}\widetilde{N}(ds,dz)
$$
is a square integrable martingale.
Combining (H3) and (\ref{eq SPDE con 1}), we obtain
\begin{equation*}
    \mathbb{E}\|X^{\epsilon,h}_{t\wedge \sigma_N}\|^2_H
   \leq
    \|h\|^2_H+\mathbb{E}\Big(\int_0^{t\wedge \sigma_N} \big(K-(\nu-\vartheta_0)\|\bar{X}^{\epsilon,h}_s\|^2_V\big)ds\Big)
    +\epsilon^2\mathbb{E}\Big(\int_0^{t\wedge \sigma_N}K(1+\|\bar{X}^{\epsilon,h}_s\|^2_H)ds\Big).
\end{equation*}

Recall that
\begin{eqnarray*}
\|v\|^2_V\geq \lambda_1\|v\|^2_H,\ \ \forall v\in V,
\end{eqnarray*}
we obtain
\begin{eqnarray*}
    \mathbb{E}\|X^{\epsilon,h}_{t\wedge \sigma_N}\|^2_H+(\nu-\vartheta_0-\frac{K\epsilon^2}{\lambda_1})\mathbb{E}\int_0^{t\wedge \sigma_N}\|\bar{X}^{\epsilon,h}_s\|^2_Vds
    \leq
    \|h\|^2_H+(1+\epsilon^2)Kt.
\end{eqnarray*}
Choose $\eta>0$ and $\epsilon_\eta>0$ such that
$$
(\nu-\vartheta_0-\frac{K\epsilon^2}{\lambda_1})\geq\eta,\ \ \ \forall \epsilon\in(0,\epsilon_\eta].
$$
Let $N\rightarrow\infty$. Then
\begin{eqnarray*}
    \mathbb{E}\|X^{\epsilon,h}_t\|^2_H+\eta\mathbb{E}\int_0^t\|\bar{X}^{\epsilon,h}_s\|^2_Vds
    \leq
    \|h\|^2_H+(1+\epsilon_\eta^2)Kt,\ \ \ \forall \epsilon\in(0,\epsilon_\eta],
\end{eqnarray*}
which implies that
\begin{eqnarray*}
    \sup_{s\in[0,t]}\mathbb{E}\|X^{\epsilon,h}_s\|^2_H+\eta\mathbb{E}\int_0^t\|\bar{X}^{\epsilon,h}_s\|^\alpha_Vds
    \leq
    \|h\|^2_H+(1+\epsilon_\eta^2)Kt,\ \ \ \forall \epsilon\in(0,\epsilon_\eta].
\end{eqnarray*}
This completes the proof.
\end{proof}

\begin{theorem}\label{SPDE main th} Suppose the Assumptions {\rm 1} and {\rm 2} hold.
Then
\begin{itemize}
\item[{\rm(1)}] For any $\widetilde{M}>0$, $\delta>0$ and $T\geq 0$,
\begin{eqnarray*}
\lim_{\epsilon\rightarrow0}\sup_{\|h\|^2_V\leq \widetilde{M}}\mathbb{P}\Big(\|X^{\epsilon,h}_T-X^{0,h}_T\|^2_H\geq \delta\Big)=0.
\end{eqnarray*}

\item[{\rm(2)}] There exists at least one stationary measure $\mu^{\epsilon,h}$ for $X^{\epsilon,h}$.

\item[{\rm(3)}] For $\epsilon\in(0,\epsilon_\eta]$, denote by $\{\mu^{\epsilon}_{i_\epsilon},\ i_\epsilon\in I_\epsilon\}$ all stationary measures for the semigroup $\{P^{\epsilon}_t\}_{t\geq 0}$. Then $\{\mu^{\epsilon}_{i_\epsilon},\ i_\epsilon\in I_\epsilon, \epsilon\in(0,\epsilon_\eta]\}$ is tight.
\end{itemize}
\end{theorem}

\begin{proof} (1) Set $\sigma_N:=\inf\{s\geq 0:\ \|X_s^{\epsilon,h}\|_H\geq N\}\bigwedge T$ and $\rho(v):=\frac{27}{(\nu-\vartheta_0)^3} \|v\|^4_{L^4(D,\mathbb{R}^2)}$.
By It\^{o}'s formula, we have
\begin{align*}
 & \exp\Big(-\int_0^{t\wedge \sigma_N}\rho(X_s^{0,h})ds\Big)\|X^{\epsilon,h}_{t\wedge \sigma_N}-X_{t\wedge \sigma_N}^{0,h}\|^2_H\\
=&\int_0^{t\wedge \sigma_N}  \exp\Big(-\int_0^{s}\rho(X_r^{0,h})dr\Big)
     \Big(
       2\langle \mathcal{A}(\bar{X}^{\epsilon,h}_s)-\mathcal{A}(X^{0,h}_s), \bar{X}^{\epsilon,h}_s-X^{0,h}_s\rangle_{V^*,V}\\
&- \rho(X_s^{0,h})\|\bar{X}^{\epsilon,h}_{s}-X_{s}^{0,h}\|^2_H
     \Big)ds\\
&+
  2\epsilon\int_0^{t\wedge \sigma_N}  \exp\Big(-\int_0^{s}\rho(X_r^{0,h})dr\Big)
       \langle B(\bar{X}^{\epsilon,h}_s)dW_s,\bar{X}^{\epsilon,h}_s-X^{0,h}_s\rangle_{H,H}\\
&+
  \epsilon^2\int_0^{t\wedge \sigma_N}  \exp\Big(-\int_0^{s}\rho(X_r^{0,h})dr\Big)
      \|B(\bar{X}^{\epsilon,h}_s)\|^2_2ds\\
& +
  2\epsilon\int_0^{t\wedge \sigma_N}\int_{Z}  \exp\Big(-\int_0^{s}\rho(X_r^{0,h})dr\Big)
      \langle f(\bar{X}^{\epsilon,h}_{s},z),\bar{X}^{\epsilon,h}_{s}-X^{0,h}_{s}\rangle_{H,H}\tilde{N}(ds,dz)\\
& +
  \epsilon^2\int_0^{t\wedge \sigma_N}\int_{Z}  \exp\Big(-\int_0^{s}\rho(X_r^{0,h})dr\Big)
  \|f(\bar{X}^{\epsilon,h}_{s},z)\|^2_HN(ds,dz),
\end{align*}
where $\bar{X}^{\epsilon,x}$ is any $V$-valued progressively measurable $dt\times\mathbb{P}$ version of $X^{\epsilon,x}$.
\vskip 0.3cm

From (\ref{eq SPDE result 1}), (\ref{eq SPDE result 2}) and Lemma \ref{SPDE prop}, we know that
\begin{align*}
L_N(t)=&2\epsilon\int_0^{t\wedge \sigma_N}  \exp\Big(-\int_0^{s}\rho(X_r^{0,h})dr\Big)
       \langle B(\bar{X}^{\epsilon,h}_s)dW_s,\bar{X}^{\epsilon,h}_s-X^{0,h}_s\rangle_{H,H}\\
        &+
       2\epsilon\int_0^{t\wedge \sigma_N}\int_{Z}  \exp\Big(-\int_0^{s}\rho(X_r^{0,h})dr\Big)
      \langle f(\bar{X}^{\epsilon,h}_{s},z),\bar{X}^{\epsilon,h}_{s}-X^{0,h}_{s}\rangle_{H,H}\tilde{N}(ds,dz)
\end{align*}
is a square integrable martingale. Therefore from (H2), (\ref{eq SPDE con 1}) and Lemma \ref{SPDE prop},
\begin{align*}
& \mathbb{E}\Big(\exp\Big(-\int_0^{t\wedge \sigma_N}\rho(X_s^{0,h})ds\Big)\|X^{\epsilon,h}_{t\wedge \sigma_N}-X_{t\wedge \sigma_N}^{0,h}\|^2_H\Big)\\
   \leq&
      \epsilon^2\mathbb{E}\int_0^{t\wedge \sigma_N}\|B(\bar{X}^{\epsilon,h}_s)\|^2_2ds
         +
      \epsilon^2\mathbb{E}\int_0^{t\wedge \sigma_N}\int_{Z} \|f(\bar{X}^{\epsilon,h}_{s},z)\|^2_H\nu(dz)ds\\
   \leq&
      \epsilon^2\mathbb{E}\int_0^tK(1+\|\bar{X}^{\epsilon,h}_s\|^2_H)ds\\
   \leq&
     K\epsilon^2t\Big(1+\sup_{s\in[0,t]}\mathbb{E}\|X_s^{\epsilon,h}\|^2_Hds\Big)\\
   \leq&
     \epsilon^2KT\Big(1+\|h\|^2_H+(1+\epsilon_\eta^2)KT\Big),
\end{align*}
which implies that
\begin{eqnarray*}
\mathbb{E}\Big(\exp\Big(-\int_0^{t}\rho(X_s^{0,h})ds\Big)\|X^{\epsilon,h}_{t\wedge \sigma_N}-X_{t\wedge \sigma_N}^{0,h}\|^2_H\Big)
   \leq
 \epsilon^2KT\Big(1+\|h\|^2_H+(1+\epsilon_\eta^2)KT\Big).
\end{eqnarray*}

\vskip 0.3cm
By (\ref{l4}) and Lemma \ref{SPDE prop},
\begin{align*}
\int_0^{t}\rho(X_s^{0,h})ds\leq&
\frac{54}{(\nu-\vartheta_0)^3}\int_0^t\|X^{0,h}_s\|^2_H\|X^{0,h}_s\|^2_Vds\\
\leq&
\frac{54}{(\nu-\vartheta_0)^3}C_\eta(\|h\|^{2}_H+2KT)^2,
\end{align*}
where $\eta$ comes from Lemma \ref{SPDE prop},
we obtain
\begin{align*}
&\mathbb{E}\Big(\|X^{\epsilon,h}_{t\wedge \sigma_N}-X_{t\wedge \sigma_N}^{0,h}\|^2_H\Big)\\
   \leq&
\epsilon^2KT \exp\Big(\frac{54}{(\nu-\vartheta_0)^3}C_\eta(\|h\|^{2}_H+2KT)^2\Big)\Big(1+\|h\|^2_H+(1+\epsilon_\eta^2)KT\Big).
\end{align*}
Leting $N\rightarrow\infty$, we have
\begin{align*}
&\mathbb{E}\Big(\|X^{\epsilon,h}_{t}-X_{t}^{0,h}\|^2_H\Big)\\
   \leq&
\epsilon^2KT \exp\Big(\frac{54}{(\nu-\vartheta_0)^3}C_\eta(\|h\|^{2}_H+2KT)^2\Big)\Big(1+\|h\|^2_H+(1+\epsilon_\eta^2)KT\Big),
\end{align*}
which implies the first part of this theorem by Chebyshev's inequality.

\vskip 0.3mm

Notice that the embedding $V \subset H$ is compact. Then the proofs for (2) and (3) are exactly the same as those in Theorem \ref{thm heat 2}, so we omit it.\qquad \end{proof}

\vskip 0.3cm
\begin{theorem}\label{NS} Assume that {\rm (\ref{exa NS 01})} satisfies the Assumptions {\rm 1} and {\rm 2}.
Let $\{\mu^{\epsilon_i}\}$ be a sequence of stationary measures for {\rm (\ref{exa NS 01})} such that $\mu^{\epsilon_i} \overset{w}{\rightarrow}\mu $ as $\epsilon_i \rightarrow 0$. Then $\mu$ is an invariant
measure of $\Phi$ and its support is contained in the Birkhoff
center for $\Phi(t)|_{\mathcal{A}_H}$.
\end{theorem}
\begin{proof}
It follows directly from  Theorems \ref{SPDE main th} and \ref{mthm}.
\qquad \end{proof}

\subsection{Stochastic  Burgers type equations}
Consider the classic Burgers equation  (see \cite[p.257-258]{DaZaErg})
\begin{equation}
    \label{Burgers}
    \left\{
        \begin{array}{l}
        \frac{\partial u}{\partial t}= \frac{\partial^2 u}{\partial x^2} + \lambda u\frac{\partial u}{\partial x},\ 0<x<1,\ t>0, \\
        u(0,t)=u(1,t)=0, \\
        u(x,0)=\varphi(x).
\end{array}
    \right.
\end{equation}
The solution generates a strongly monotone flow in a suitable function space $V=W^{1,2}_0(0,1)$
(see \cite{Hir88a}, or \cite{Pola, HNW}).
It is easy to compute that the trivial solution is the unique stationary solution for (\ref{Burgers}).
Set $H=L^2([0,1])$ and $A$ denote the Laplace operator. Since $V$ coincides with $D(A^{1/2})$, we endow $V$ with the norm $\|u\|_V=\|A^{1/2}u\|_H$, which is equivalent with the usual norm in the Sobolev space $V$.
Also we can prove that $\|u^{\varphi}(t)\|^2_V $ is bounded on $[0, \infty )$ for every initial value $\varphi\in V$.

In fact, if $u(0)=\varphi\in V$, then there exists a unique solution $u^\varphi\in C([0,\infty), V)\cap L^2([0,\infty),D(A))$ of (\ref{Burgers}). This fact can be obtained by a fixed-point theorem, and the proof is omitted.
We have
\begin{eqnarray}\label{eq Burgers H}
\|u^\varphi(t)\|^2_H+2\int_0^t\|u^\varphi(s)\|^2_{V}ds=\|\varphi\|^2_H,
\end{eqnarray}
where we have used fact that $\langle u\frac{\partial u}{\partial x},u\rangle_{H,H}=0$.

Since $|\langle u\frac{\partial u}{\partial x},Au\rangle_{H,H}|\leq c\|u\|^2_V\|u\|_{D(A)}$ (\cite{Temam95} or Lemma 2.2 in \cite{Dong-Xu}),
we have
\begin{align*}\label{eq Burgers V}
 &\|u^\varphi(t)\|^2_V+2\int_0^t\|u^\varphi(s)\|^2_{D(A)}ds\\
 =&
 \|\varphi\|^2_V-2\lambda\int_0^t\langle u^\varphi(s)\frac{\partial u^\varphi(s)}{\partial x},Au^\varphi(s)\rangle_{H,H}ds \\
 \leq&
 \|\varphi\|^2_V+2c|\lambda|\int_0^t\|u^\varphi(s)\|^2_V\|u^\varphi(s)\|_{D(A)}ds\\
 \leq&
 \|\varphi\|^2_V+\int_0^t\|u^\varphi(s)\|^2_{D(A)}ds+(\lambda c)^2\int_0^t\|u^\varphi(s)\|^4_Vds,
\end{align*}
hence
$$
\|u^\varphi(t)\|^2_V+\int_0^t\|u^\varphi(s)\|^2_{D(A)}ds
\leq
\|\varphi\|^2_V+(\lambda c)^2\int_0^t\|u^\varphi(s)\|^2_V\cdot \|u^\varphi(s)\|^2_V ds,
$$
by Gronwall's inequality and (\ref{eq Burgers H}),
\begin{eqnarray*}\label{eq Burgers V0}
 \|u^\varphi(t)\|^2_V+\int_0^t\|u^\varphi(s)\|^2_{D(A)}ds
 \leq
 \|\varphi\|^2_Ve^{(\lambda c)^2\int_0^t\|u^\varphi(s)\|^2_Vds}
 \leq
 \|\varphi\|^2_Ve^{\frac{1}{2}(\lambda c)^2\|\varphi\|^2_H}.
\end{eqnarray*}
This implies that $\{u^{\varphi}(t)\}_{t\geq 0} $ is bounded in $V$. Then applying the theory on monotone dynamical systems (see \cite{Hir88a}), we conclude that all solutions for (\ref{Burgers}) are convergent to the trivial solution. Therefore, the Birkhoff center for (\ref{Burgers}) is $\{0\}$.

Consider one dimensional stochastic  Burgers equation driven by L\'evy noise:
\begin{equation}\label{sburgers}
\begin{split}
  dX_t^{\epsilon,h}=&
 \big(\Delta X_t^{\epsilon,h}+X_t^{\epsilon,h}\cdot \nabla X_t^{\epsilon,h}\big)dt\\
  &+\epsilon B(X_t^{\epsilon,h})dW_t+\epsilon\int_Z f(X_{t-}^{\epsilon,h},z)\tilde{N}(dt,dz),
\end{split}
\end{equation}
with deterministic initial value $ X_0^{\epsilon,h}=h$.

\begin{theorem}\label{sBeq}
Suppose {\rm(\ref{sburgers})} satisfies the Assumption {\rm 2} in subsection {\rm 4.2} and $h\in H$.
Then the results of Theorem {\rm\ref{SPDE main th}} hold. Moreover any limiting measures
for its stationary measures are the Dirac measure $\delta_0$.
\end{theorem}

\begin{proof} The It\^{o} formula deduces that
\begin{equation}\label{eq star 3}
\begin{split}
 &\|X^{\epsilon,h}_t\|^2_H+2\int_0^t\|X^{\epsilon,h}_s\|^2_Vds\\
 =& \|h\|^2_H +2\epsilon\int_0^t\langle B(X^{\epsilon,h}_s)dW_s,X^{\epsilon,h}_s\rangle_{H,H}\\
     &+2\epsilon\int_0^t\int_{Z}\langle f(X^{\epsilon,h}_{s-},z),X^{\epsilon,h}_{s-}\rangle_{H,H}\widetilde{N}(ds,dz)\\
     &+\epsilon^2\int_0^t\|B(X^{\epsilon,h}_s)\|^2_2ds + \epsilon^2 \int_0^t\int_{Z}\|f(X^{\epsilon,h}_{s-},z)\|^2_HN(ds,dz).
\end{split}
\end{equation}
By the same argument of Lemma \ref{SPDE prop},
there  exist $\eta>0$, $\epsilon_\eta>0$ and $F>0$ such that

\begin{equation}\label{besti}
\begin{split}
&\sup_{s\in[0,T]}\mathbb{E}\|X^{\epsilon,h}_s\|^2_H+\eta\mathbb{E}\int_0^{T}\|X^{\epsilon,h}_s\|^2_Vds\\
    \leq&
    \|h\|^2_H+(1+\epsilon_\eta^2)FT,\  \forall \epsilon\in(0,\epsilon_\eta].
\end{split}
\end{equation}

Define $\mathcal{A}: V\rightarrow V^*$ by

$$\mathcal{A}(v) := v_{xx} +vv_x.$$
According to \cite[Lemma 2.1(2)]{BLZ}, for any $u, v\in V$, there is a positive constant $C$ such that
\begin{equation}\label{Aesti}
2\langle \mathcal{A}(u)-\mathcal{A}(v),u-v\rangle_{V^*,V}\leq -\|u-v\|^2_V+C(1+\|v\|^2_V)\|u-v\|^2_H.
\end{equation}
Combining(\ref{besti}) and (\ref{Aesti}), similar to the proof of Theorem \ref{SPDE main th}, we can obtain the probability convergence, existence of stationary measures and their tightness. Thus the last conclusion follows immediately from Theorem \ref{mthm}.
\end{proof}

\section{FDEs driven by white noise}

Consider the $m$-dimensional stochastic functional differential equations (SFDEs)
\begin{equation}\label{SFDE}
\begin{array}{l}
    dX^{\epsilon,\phi}(t)=b(X_{t}^{\epsilon,\phi})dt+\epsilon\sigma(X_{t}^{\epsilon,\phi})dW(t),\:\: \\ X_{0}^{\epsilon,\phi}=\phi\in\mathcal{C}:= C([-\tau,0],\mathbb{R}^{m}),
\end{array}
\end{equation}
where $W=\{W_{t}=(W_{t}^{1},\cdots,W_{t}^{k}),t\geq 0\}$ is a $k$-dimensional Wiener process,
$b(\cdot): \mathcal{C}\rightarrow\mathbb{R}^{m}$ and $\sigma(\cdot): \mathcal{C}\rightarrow\mathbb{R}^{m\times r}$
satisfy global Lipschitz condition and linear growth condition, that is, there exists a positive constant $L$ such that
$\forall \phi,\psi\in \mathcal{C}$,

{\rm(a)} $|b(\phi)-b(\psi)|+\|\sigma(\phi)-\sigma(\psi)\|_2\leq L\|\phi-\psi\|$,

{\rm(b)} $|b(\phi)|^{2}+\|\sigma(\phi)\|_{2}^{2}\leq L^{2}(1+\|\phi\|^{2})$.

\noindent Here $\mathcal{C}$ denotes the set of continuous functions
$\phi(s)$ from $[-\tau,0]$ into $\mathbb{R}^{m}$ with the uniform norm
$\|\phi\|=\displaystyle\sup_{-\tau\leq s\leq 0}|\phi(s)|$.

It is known that the Hypotheses (a) and (b) are sufficient
to ensure the global existence and uniqueness of a strong solution to (\ref{SFDE}).
Let $X^{\epsilon,\phi}(t)$
denote the solution to (\ref{SFDE}) with initial data $X^{\epsilon,\phi}_{0}=\phi$.
Then the segment process of $X^{\epsilon,\phi}(t)$ is given by
$$ X^{\epsilon,\phi}_{t}(\theta)=X^{\epsilon,\phi}(t+\theta),\:\theta\in[-\tau,0]. $$
That is, $\{X^{\epsilon,\phi}_{t}\}_{t\geq 0}$ is a process on $\mathcal{C}$.
Furthermore, the segment process $\{X^{\epsilon,\phi}_{t}\}_{t\geq 0}$ to (\ref{SFDE}) is immediately a Feller process
on the path space $\mathcal{C}$ (see, e.g. Mohammed \cite[ Theorem III 3.1, p.67-68]{Moham}), where
the associated Markov semigroup
\begin{equation}\label{SfdeMs}
   P_{t}^{\epsilon}g(\phi)=\mathbb{E}g(X^{\epsilon,\phi}_{t}),\quad t\geq 0,\:\phi \in\mathcal{C},\: g\in\mathcal{B}_{b}(\mathcal{C}).
\end{equation}

Consider the corresponding $m$-dimensional deterministic functional differential equations (FDEs)
\begin{equation}\label{FDE}
    dX^{\phi}(t)=b(X_{t}^{\phi})dt,\:\: X_{0}^{\phi}=\phi\in\mathcal{C}.
\end{equation}
Under Hypothesis (a), it is easy to see that (\ref{FDE}) generates a
semiflow $\Phi_{t}(\phi)=\Phi^{\phi}_{t}$, $t\geq0$ on $\mathcal{C}$.

The following result reveals a close connection between (\ref{SFDE}) and (\ref{FDE}).

\begin{lemma}
\label{sfdel2}
Suppose  {\rm(a)} and  {\rm(b)} hold.
Let $K\subset\mathcal{C}$ be a compact  ${\rm(}bounded {\rm)}$ set and $T>0$. Then for  sufficiently small $\epsilon>0$
\[ \displaystyle\sup_{\phi\in K}\mathbb{E}\big[\displaystyle\sup_{0\leq t\leq T}\|X^{\epsilon,\phi}_{t}-\Phi_{t}(\phi)\|^{2}\big]\leq
C\epsilon^{2}, \]
where $C=C(K,L,T)$ is a positive constant, depending only on $K$, $L$ and $T$.
\end{lemma}

The proof is easy, so we omit it.

\

The following probability convergence can be obtained by Chebyshev's inequality.

\begin{corollary}\label{corosfd}
Let $K\subset \mathcal{C}$ be a compact set. Then for any $T>0$ and $\delta>0$, we have
$$ \displaystyle\lim_{\epsilon\rightarrow 0}\sup_{\phi\in K}\mathbb{P}\big\{\displaystyle\sup_{0\leq t\leq T}\|X^{\epsilon,\phi}_{t}-\Phi_{t}(\phi)\|\geq\delta\big\}=0.  $$
\end{corollary}

In order to apply Theorem \ref{mthm}, we need to present a criterion on the existence of
stationary measures for (\ref{SFDE}) and  their tightness. For this purpose, from now on, we
focus on the following stochastic functional differential equations
\begin{equation}\label{SFDEsnp}
  dX^{\epsilon}(t)=[-BX^{\epsilon}(t)+Ag(X^{\epsilon}_{t})]dt+\epsilon\sigma(X^{\epsilon}_{t})dW(t)
\end{equation}
with initial data $X^{\epsilon}_{0}=\phi\in \mathcal{C}$, where
$B=(b_{ij})_{m\times m}$, $A=(a_{ij})_{m\times m}$ are two matrices,
$\sigma(\phi)=(\sigma_{ij}(\phi))_{m\times k}$ is an $m\times k$ matrix valued function defined on $\mathcal{C}$,
and $g:\mathcal{C}\rightarrow \mathcal{C}$ is a measurable function.
We will suppose the following assumptions on $g$ and $\sigma$:

${\rm(A_{1})}$ There exists a positive constant $\widetilde{L}$ such that for all
$\phi,\psi\in\mathcal{C}$
$$  |g(\phi)-g(\psi)|\leq \widetilde{L}\|\phi-\psi\|. $$

${\rm(A_{2})}$
There exists a positive constant $L$ such that for all $\phi,\psi\in\mathcal{C}$
$$  \|\sigma(\phi)-\sigma(\psi)\|_2 \leq L\|\phi-\psi\|. $$

To show the existence of a stationary measure,  it is sufficient  to  prove the tightness of the segments by using
the Arzel\`{a}-Ascoli tightness characterization and the Krylov-Bogoliubov theorem.

Using the idea presented in \cite[Proposition 2.1]{Scheu}, we have

\begin{theorem}\label{unbud}Assume ${\rm(A_{1})}$, ${\rm(A_{2})}$ and
\begin{equation}\label{dissip}
\langle x,Bx\rangle\geq b|x|^{2}\:\: \textrm{for\:any}\:\: x\in\mathbb{R}^{m}.
\end{equation}
If $b$\ satisfies the following:
\begin{equation}\label{conb}
  b>\frac{\gamma^{2} e^{6\tau}\big(16\widetilde{L}^{3}|A|^{3}\big)^{2}}{(1-\kappa e^{-3\tau})^{2}},
\end{equation}
where $\kappa\in(1,e^{3\tau})$ is arbitary, and $\gamma=\gamma(\kappa)=9[\frac{2\sqrt{\kappa}-1}{(\sqrt{\kappa}-1)^{2}}+1]$.
Then there exists $\epsilon_{0}>0$ such that
\begin{equation}\label{esti}
\displaystyle\sup_{0<\epsilon\leq\epsilon_{0}}\sup_{t\geq 0}\mathbb{E}\big[\|X^{\epsilon,\phi}_{t}\|^{6}\big]\leq 2e^{3\tau} \big(\|\phi\|^{6}+\widetilde{M}\big),
\end{equation}
where $\widetilde{M}$ is a constant independent of $\epsilon \in (0, \epsilon_{0}]$. Furthermore, for each $\phi\in \mathcal{C}$, there exists at least a stationary measure $\mu^{\epsilon,\phi}$ for {\rm(\ref{SFDEsnp})} corresponding to the segment process $\{X^{\epsilon,\phi}_{t}\}_{t\geq0}$.
\end{theorem}

\begin{proof} Fixing $\epsilon$, to simplify notation, we let $X(t)=X^{\epsilon,\phi}(t)$ and
set $Z(t)=|X(t)|^{2}$, $t\geq 0$. By the It\^{o} formula, (\ref{dissip}), ${\rm(A_{1})}$ and ${\rm(A_{2})}$,
we have
\begin{align*}
    dZ(t)\leq& -2bZ(t)dt+2|A||X(t)| |g(X_{t})|dt+2\epsilon^{2}(L^{2}\|X_{t}\|^{2}+\|\sigma(0)\|_{2}^{2})dt\\
    &+2\epsilon\langle X(t),\sigma(X_{t})dW(t)\rangle \\
    \leq& -2bZ(t)dt+2|A||X(t)| (\widetilde{L}\|X_{t}\|+|g(0)|)dt+2\epsilon^{2}(L^{2}\|X_{t}\|^{2}+\|\sigma(0)\|_{2}^{2})dt\\
    &+2\epsilon\langle X(t),\sigma(X_{t})dW(t)\rangle \\
    \leq& -2bZ(t)dt+2|A|(\widetilde{L}\|X_{t}\|^{2}+\widetilde{L}\|X_{t}\|^{2}+\frac{|g(0)|^{2}}{\widetilde{L}})dt
    +2\epsilon^{2}(L^{2}\|X_{t}\|^{2}+\|\sigma(0)\|_{2}^{2})dt\\
    &+2\epsilon\langle X(t),\sigma(X_{t})dW(t)\rangle \\
    =&:-2bZ(t)dt+C\|X_{t}\|^{2}dt
    +D dt
    +2\epsilon\langle X(t),\sigma(X_{t})dW(t)\rangle,\quad t\geq0
\end{align*}
where $C=C(\epsilon)=4\widetilde{L}|A|+2\epsilon^{2}L^{2}$, and $D=D(\epsilon)=\frac{2|A||g(0)|^{2}}{\widetilde{L}}+2\epsilon^{2}\|\sigma(0)\|_{2}^{2}$.
Then the stochastic variation of constants formula yields that
\begin{align*}
     Z(t)&\leq e^{-2bt}Z(0)+\int_{0}^{t}e^{-2b(t-s)}(C\|X_{s}\|^{2}+D)ds\\
     &+2\epsilon
     \int_{0}^{t}e^{-2b(t-s)}\langle X(s),\sigma(X_{s})dW(s)\rangle\\
        &\leq e^{-2bt}Z(0)+ \frac{C}{2b}\displaystyle\sup_{0\leq s\leq t}\|X_{s}\|^{2}+\frac{D}{2b}\\
        &+2\epsilon\int_{0}^{t}e^{-2b(t-s)}\langle X(s),\sigma(X_{s})dW(s)\rangle, \ \forall t\geq 0.
\end{align*}
Hence, for $0\leq t\leq\tau$ we obtain that
\begin{align*}
     \displaystyle\sup_{0\leq t\leq\tau}e^{t}Z(t)\leq& Z(0)
     +\frac{C}{2b}e^{\tau}\displaystyle\sup_{0\leq s\leq\tau}\|X_{s}\|^{2}
     +\frac{D}{2b}e^{\tau}\\
     &+2\epsilon e^{\tau}\displaystyle\sup_{0\leq t\leq\tau}
     |\int_{0}^{t}e^{-2b(t-s)}\langle X(s),\sigma(X_{s})dW(s)\rangle|,
\end{align*}
where we have used the fact that $b>\frac{1}{2}$. It is easy to see that
for any $\kappa\in(1,e^{3\tau})$, there exists $\gamma=\gamma(\kappa)>1$ such that
\begin{equation}\label{fourineq}
  (x_{1}+x_{2}+x_{3}+x_{4})^{3}\leq \kappa x^{3}_{1}+\gamma(x^{3}_{2}+x^{3}_{3}+x^{3}_{4})\:\: \textrm{for\:all}\: x_{1},x_{2},x_{3},x_{4}\geq 0.
  \footnote{For arbitrary $\kappa\in(1,e^{3\tau})$, choosing $\gamma=\gamma(\kappa)=9[\frac{2\sqrt{\kappa}-1}{(\sqrt{\kappa}-1)^{2}}+1]$ such that
  (\ref{fourineq}) holds.}
\end{equation}
Combining the above inequality (\ref{fourineq}) and taking
expectations, we obtain that
\begin{equation}\label{crest}
\begin{split}
 &\mathbb{E}\big[\displaystyle\sup_{0\leq t\leq\tau}e^{3t}|Z(t)|^{3}\big]\\
 \leq& \kappa\mathbb{E}\big[|Z(0)|^{3}\big]
 +\gamma\Big(\frac{C^{3}}{8b^{3}}e^{3\tau}\mathbb{E}\big[\displaystyle\sup_{0\leq s\leq\tau}\|X_{s}\|^{6}\big]
 +\frac{D^{3}}{8b^{3}}e^{3\tau}\\
 &+8\epsilon^{3} e^{3\tau}\mathbb{E}\big[\displaystyle\sup_{0\leq t\leq\tau}
 |\int_{0}^{t}e^{-2b(t-s)}\langle X(s),\sigma(X_{s})dW(s)\rangle|^{3}\big]\Big).
\end{split}
\end{equation}

From L\'{e}vy's celebrated martingale characterization of Brownian motion (see \cite[Theorem 3.16, p.157]{KaraShre}), we know that there exists a one-dimensional Brownian motion $B$ with respect to the same filtration such that
 $$\langle X(s),\sigma(X_{s})dW(s)\rangle=\beta(s,\omega)dB(s),$$
where
$$\beta(s,\omega)= \Big(\displaystyle\sum_{j=1}^{k}\big(\displaystyle\sum_{i=1}^{m}X_i(s)\sigma_{ij}(X_{s})\big)^2\Big)^\frac{1}{2}.$$

By the technical result \cite[Lemma 2.2]{Scheu} and $(A_{2})$, we get that
\begin{align*}
&\mathbb{E}\big[\displaystyle\sup_{0\leq t\leq\tau}
      |\int_{0}^{t}e^{-2b(t-s)}\langle X(s),\sigma(X_{s})dW(s)\rangle|^{3}\big]\\
    \leq& 2\tau a_{3,2b}\big[\big(2L^{3}+\|\sigma(0)\|_{2}^{3}\big)\mathbb{E}\|X_{\tau}\|^{6}
     +L^{3}\mathbb{E}\|X_{0}\|^{6}+\|\sigma(0)\|_{2}^{3}\big],
\end{align*}
where
\begin{equation}\label{factorfor}
\begin{split}
  a_{3,2b}&=C_{3}\Big(\frac{3-1}{6b}\Big)^{3\alpha-1}\Gamma\Big(\frac{3\alpha-1}{3-1}\Big)^{3-1}
    \Big[\Big(\frac{1}{4b}\Big)^{1-2\alpha}\Gamma\Big(1-2\alpha\Big)\Big]^{\frac{3}{2}}\\
    &:=\Lambda\Big(\alpha\Big)\Big(\frac{1}{b}\Big)^{\frac{1}{2}}.
\end{split}
\end{equation}
where $C_{3}=\Big(\frac{81}{8}\Big)^{\frac{3}{2}}$ is  the universal positive constant in the Burkholder-Davis-Gundy inequality
(see \cite[Theorem 7.3, p.40]{Mao}) and $\Gamma(s)=\int_{0}^{+\infty}t^{s-1}e^{-t}$ is a Gamma function and $\alpha\in(\frac{1}{3},\frac{1}{2})$.

Continuing on from line (\ref{crest}) and using above inequality, we have
\begin{equation}\label{finduc}
{\small \begin{split}
    &\mathbb{E}\big[\displaystyle\sup_{0\leq t\leq\tau}e^{3t}|Z(t)|^{3}\big]\\
    \leq&\kappa\mathbb{E}|Z(0)|^{3}
     +\gamma\frac{C^{3}}{8b^{3}}e^{3\tau}\mathbb{E}\big[\|X_{0}\|^{6}+\|X_{\tau}\|^{6}\big]
     +\gamma\frac{D^{3}}{8b^{3}}e^{3\tau}\\
     &+\gamma8\epsilon^{3} e^{3\tau}2\tau a_{3,2b}\big[(2L^{3}+\|\sigma(0)\|_{2}^{3})\mathbb{E}\|X_{\tau}\|^{6}
     +L^{3}\mathbb{E}\|X_{0}\|^{6}+\|\sigma(0)\|_{2}^{3}\big]\\
    \leq&\gamma\frac{D^{3}}{8b^{3}}e^{3\tau}+16\epsilon^{3}\tau\gamma a_{3,2b}e^{3\tau}\|\sigma(0)\|_{2}^{3}\\
    & +\kappa\mathbb{E}|Z(0)|^{3}+\big(\gamma\frac{C^{3}}{8b^{3}}e^{3\tau}+16\epsilon^{3}\tau\gamma a_{3,2b}e^{3\tau}L^{3}\big)\mathbb{E}\|X_{0}\|^{6}\\
     &+\big[\gamma\frac{C^{3}}{8b^{3}}e^{3\tau}+16\epsilon^{3}\tau\gamma a_{3,2b}e^{3\tau}\big(2L^{3}+\|\sigma(0)\|_{2}^{3}\big)\big] \mathbb{E}\|X_{\tau}\|^{6}.
\end{split}}
\end{equation}
Define a Lyapunov function $V:C([-\tau,0],\mathbb{R})\rightarrow\mathbb{R}_{+}$ by
\[ V(\zeta)=\displaystyle\sup_{-\tau\leq s\leq 0}e^{3s}|\zeta(s)|^{3}.  \]
Let $\psi(s)=|\phi(s)|^{2}$, $s\in[-\tau,0]$. Therefore, (\ref{finduc}) along with the fact that
$\mathbb{E}|Z(0)|^{3}\leq \mathbb{E}V(\psi)$,
$\mathbb{E}\|X_{0}\|^{6}\leq e^{3\tau}\mathbb{E}V(\psi)$,
$\mathbb{E}\big[\displaystyle\sup_{0\leq s\leq\tau}e^{3s}|Z(s)|^{3}\big]=e^{3\tau}\mathbb{E}V(Z_{\tau})$ and
$\mathbb{E}\|X_{\tau}\|^{6}\leq e^{3\tau}\mathbb{E}V(Z_{\tau})$
imply that
\begin{equation}\label{iter1}
   \begin{split}
    &\big\{1-\gamma e^{3\tau}[\frac{C^{3}}{8b^{3}}+16\epsilon^{3}\tau a_{3,2b}(2L^{3}+\|\sigma(0)\|_{2}^{3})]\big\}
    \mathbb{E}V(Z_{\tau})\\
    \leq& \big[\kappa e^{-3\tau}
    +\gamma e^{3\tau}(\frac{C^{3}}{8b^{3}}+16\epsilon^{3}\tau a_{3,2b}L^{3})\big] \mathbb{E}V(\psi)\\
    &+\gamma (\frac{D^{3}}{8b^{3}}+16\epsilon^{3}\tau a_{3,2b}\|\sigma(0)\|_{2}^{3}).
    \end{split}
\end{equation}
We assume  that
\begin{equation*}
  \left\{
 \begin{aligned}
 &  1-\gamma e^{3\tau}[\frac{C^{3}}{8b^{3}}+16\epsilon^{3}\tau a_{3,2b}(2L^{3}+\|\sigma(0)\|_{2}^{3})]>0 \\
 &  \delta=\delta(\epsilon):=\frac{\kappa e^{-3\tau}
    +\gamma e^{3\tau}(\frac{C^{3}}{8b^{3}}+16\epsilon^{3}\tau a_{3,2b}L^{3})}{1-\gamma e^{3\tau}[\frac{C^{3}}{8b^{3}}+16\epsilon^{3}\tau a_{3,2b}(2L^{3}+\|\sigma(0)\|_{2}^{3})]}<1,
 \end{aligned}
\right.
\end{equation*}
which is equivalence to
\begin{equation}\label{bineq1}
  \left\{
 \begin{aligned}
 & \frac{C^{3}}{8b^{3}}+16\epsilon^{3}\tau(2L^{3}+\|\sigma(0)\|_{2}^{3})a_{3,2b}<\frac{1}{\gamma e^{3\tau}} \\
 &  \frac{C^{3}}{4b^{3}}+16\epsilon^{3}\tau (3L^{3}+\|\sigma(0)\|_{2}^{3})a_{3,2b}
 <\frac{1-\kappa e^{-3\tau}}{\gamma e^{3\tau}}.
 \end{aligned}
\right.
\end{equation}
By (\ref{factorfor}) and the fact that $b>1$, it suffices to show that
\begin{equation}\label{estiofb}
  \left\{
\begin{aligned}
 &b> \gamma^{2} e^{6\tau}[\frac{C^{3}}{8}+16\epsilon^{3}\tau(2L^{3}+\|\sigma(0)\|_{2}^{3})
 \Lambda]^{2}\\
  &b>\frac{\gamma^{2} e^{6\tau}[\frac{C^{3}}{4}+16\epsilon^{3}\tau (3L^{3}+\|\sigma(0)\|_{2}^{3})\Lambda]^{2}}{(1-\kappa e^{-3\tau})^{2}}.
\end{aligned}
\right.
\end{equation}
This shows that we can find $\epsilon_{0}>0$ such that for each $\epsilon\leq \epsilon_{0}$, (\ref{estiofb}) holds as long as (\ref{conb}) is satisfied. Let $\rho:=\frac{\gamma (\frac{D^{3}}{8b^{3}}+16\epsilon^{3}\tau a_{3,2b}\|\sigma(0)\|_{2}^{3})}{1-\gamma e^{3\tau}[\frac{C^{3}}{8b^{3}}+16\epsilon^{3}\tau a_{3,2b}(2L^{3}+\|\sigma(0)\|_{2}^{3})]}$.
Then for every $\epsilon\leq \epsilon_{0}$
$$\frac{\rho}{1-\delta}=
\frac{\gamma (\frac{D^{3}}{8b^{3}}+16\epsilon^{3}\tau a_{3,2b}\|\sigma(0)\|_{2}^{3})}
{1-\kappa e^{-3\tau}-\gamma e^{3\tau}[\frac{C^{3}}{4b^{3}}
+16\epsilon^{3}\tau a_{3,2b}(3L^{3}+\|\sigma(0)\|_{2}^{3})]}\leq\frac{\rho(\epsilon_{0})}{1-\delta(\epsilon_{0})}.$$
If (\ref{estiofb}) holds, then from (\ref{bineq1}) we have
\begin{equation}\label{iterfirs}
  \mathbb{E}V(Z_{\tau})\leq \delta \mathbb{E}V(\psi)+\rho.
\end{equation}
Iterating (\ref{iterfirs}), we get that
\begin{equation}\label{iterk}
  \mathbb{E}V(Z_{k\tau})\leq \delta^{k} \mathbb{E}V(\psi)+\rho(\frac{1}{1-\delta})
  \leq \mathbb{E}V(\psi)+\frac{\rho}{1-\delta}  \  \  \textrm{for all } \ k\in\mathbb{N}^{*}.
\end{equation}
This implies that
$$  \displaystyle\sup_{k\in\mathbb{N}^{*}}\mathbb{E}\|Z_{k\tau}\|^{3}\leq e^{3\tau}(\mathbb{E}V(\psi)+\frac{\rho}{1-\delta}). $$
Note that for $t\in[k\tau,(k+1)\tau]$,  $\|Z_{t}\|^3\leq \|Z_{k\tau}\|^3+\|Z_{(k+1)\tau}\|^3$, $\forall k\in\mathbb{N}$. In term of
the original process $X$, we conclude that for all $0<\epsilon\leq\epsilon_{0}$,
$$ \displaystyle\sup_{t\geq 0}\mathbb{E}\|X_{t}^{\epsilon,\phi}\|^{6} \leq 2e^{3\tau}
\left(\mathbb{E}V(\psi)+\frac{\rho}{1-\delta}\right)\leq 2e^{3\tau} \left(\|\phi\|^{6}+\frac{\rho(\epsilon_{0})}{1-\delta(\epsilon_{0})}\right). $$

By adopting the Arzel\`{a}-Ascoli tightness characterization, we can
show that the law $\{P^{\epsilon}_{t}(\phi,\cdot)\}_{t\geq 0}$ of segment process $\{X^{\epsilon,\phi}_{t}\}_{t\geq0}$
is tight in $(\mathcal{C},\mathcal{B}(\mathcal{C}))$ (see, \cite[Theorem 2.3]{Scheu}), which implies that $\{Q_t(\cdot):= \frac{1}{t}\int_0^t P^{\epsilon}_{s}(\phi,\cdot)ds\}_{t\geq 0}$ is tight in $(\mathcal{C},\mathcal{B}(\mathcal{C}))$. Then
applying Krylov-Bogoliubov theorem,  we can conclude that $\{Q_t(\cdot)\}_{t\geq 0}$ has at least a weak convergence limit
$\mu^{\epsilon,\phi}$ which is stationary for the segment process $\{X^{\epsilon,\phi}_{t}\}_{t\geq0}$.
We omit the details and refer the readers to the proof of \cite[Theorem 2.3, 3.2]{Scheu}
and \cite[Theorem 3.1.1 p.21]{DaZaErg}.
\end{proof}

For each $0<\epsilon\leq\epsilon_{0}$, the following assumption is a sufficient
condition to guarantee the uniqueness of a stationary measure.

${\rm(A_{3})}$ The diffusion matrix $\sigma\sigma^{T}$ is \emph{uniformly elliptic} in $\mathcal{C}$, i.e.,
there is a constant $\lambda>0$ such that $x^{T}\sigma(\phi)(\sigma(\phi))^{T}x\geq \lambda|x|^{2}$
for all $\phi \in\mathcal{C}$ and $x\in\mathbb{R}^{m}$.

The next result is implicitly proved in \cite[Theorem 3.1]{Hairer}.

\begin{lemma}\label{amos}
Under the assumptions ${\rm(A_{1})}$, ${\rm(A_{2})}$ and ${\rm(A_{3})}$, there exists a unique stationary
measure for {\rm(\ref{SFDEsnp})}, that is, $\mu^{\epsilon,\phi}\equiv \mu^{\epsilon}$ is independent of $\phi$ for each $0<\epsilon\leq\epsilon_{0}$.
\end{lemma}

\begin{remark}If $\sigma$ satisfies ${\rm(A_{2})}$ and ${\rm(A_{3})}$, then $\sigma$
admits a continuous bounded right inverse, i.e., there exists a continuous
function $\widetilde{\sigma}:\mathcal{C}\rightarrow \mathbb{R}^{k\times m}$
such that for all $\phi \in\mathcal{C}$, $\sigma(\phi)\widetilde{\sigma}(\phi)=I_{m}$,
and $\displaystyle\sup_{\phi\in\mathcal{C}}|\widetilde{\sigma}(\phi)|<\infty$.
Actually, it is easy to see that $\widetilde{\sigma}=\sigma^{T}(\sigma\sigma^{T})^{-1}$.
\end{remark}

Let $\bar{B}_R(0):=\{\phi \in \mathcal{C}:\|\phi\|\leq R\}$ for a given $R>0$.
The following result gives the tightness for the family of stationary measures
$\{\mu^{\epsilon,\phi}\}_{0<\epsilon\leq\epsilon_{0},\phi\in \bar{B}_R(0)}$.

\begin{theorem}\label{unitig} Suppose the assumptions of Theorem {\rm \ref{unbud}} hold. Then the set of stationary measures $\{\mu^{\epsilon,\phi}\}_{0<\epsilon\leq\epsilon_{0},\phi\in \bar{B}_R(0)}$ is tight. If we additionally assume that ${\rm(A_{3})}$ holds,
then the set of stationary measures $\{\mu^{\epsilon}\}_{0<\epsilon\leq\epsilon_{0}}$ is tight.
\end{theorem}

\begin{proof} Fix $R>0$. For given $0<\epsilon\leq\epsilon_{0}$  and $\phi\in \bar{B}_R(0)$, from the proof of Theorem \ref{unbud},
we know that there exists a sequence $\{T_{n}\}\rightarrow+\infty$,
depending on $\epsilon$ and $\phi$,  such that
\begin{equation}\label{KBtw}
  \frac{1}{T_{n}}\int_{0}^{T_{n}}\mathbb{P}(X^{\epsilon,\phi}_{s}\in \cdot)ds\xlongrightarrow{w}\mu^{\epsilon,\phi}(\cdot),
  \  \text{as} \ n\rightarrow\infty.
\end{equation}
Then
  \begin{align*}
&\mu^{\epsilon,\phi}\{\varphi \in \mathcal{C}:|\varphi(0)|>\lambda\}\\
   \leq&  \displaystyle\liminf_{n\rightarrow\infty}
   \frac{1}{T_{n}}\int_{0}^{T_{n}}\mathbb{P}(|X^{\epsilon,\phi}_{s}(0)|>\lambda)ds\\
   =&\displaystyle\liminf_{n\rightarrow\infty}
   \frac{1}{T_{n}}\int_{0}^{T_{n}}\mathbb{P}(|X^{\epsilon,\phi}(s)|>\lambda)ds\\
   \leq & \frac{\displaystyle\sup_{0<\epsilon\leq\epsilon_{0}}\sup_{t\geq 0}\mathbb{E}|X^{\epsilon,\phi}(t)|^{6}}{\lambda^{6}}\\
   \leq & \frac{2e^{3\tau} \big(R^{6}+\widetilde{M}\big)}{\lambda^{6}}
   \longrightarrow 0 \ \textrm{uniformly in}\ 0<\epsilon\leq\epsilon_{0}\   \textrm{and}\ \phi\in \bar{B}_R(0) \ \textrm{as} \ \lambda\rightarrow\infty.
   \end{align*}
Here we have used the fact that $\{\varphi\in \mathcal{C}:|\varphi(0)|>\lambda\}$ is an open set, (\ref{KBtw}) and
Portmanteau Theorem (\cite[Theorem 2.1, p.16 ]{Bill})
to obtain the first inequality,  the Chebyshev's inequality to the second inequality, and (\ref{esti}) to the last inequality.
This means that
\begin{equation}\label{ubound}
 \displaystyle\lim_{\lambda\rightarrow\infty}\sup_{0<\epsilon\leq\epsilon_{0},\phi\in \bar{B}_R(0)}\mu^{\epsilon,\phi}\{\varphi\in \mathcal{C}:|\varphi(0)|>\lambda\}=0.
\end{equation}

For every $\gamma>0$, by a similar argument as used in the above first inequality, we have
\begin{align*}
 &\mu^{\epsilon,\phi}
        \{\varphi\in \mathcal{C}:\displaystyle\sup_{-\tau\leq u\leq v\leq 0 \atop v-u\leq\delta}|\varphi(v)-\varphi(u)|>\gamma\}\\
   \leq & \displaystyle\liminf_{n\rightarrow\infty}
   \frac{1}{T_{n}}\int_{0}^{T_{n}}\mathbb{P}(
   \displaystyle\sup_{-\tau\leq u\leq v\leq 0\atop v-u\leq\delta}|X^{\epsilon,\phi}_{t}(v)-X^{\epsilon,\phi}_{t}(u)|>\gamma)dt\\
   =&\displaystyle\liminf_{n\rightarrow\infty}
   \frac{1}{T_{n}}\int_{0}^{T_{n}}\mathbb{P}(
   \displaystyle\sup_{-\tau\leq u\leq v\leq 0\atop v-u\leq\delta}|X^{\epsilon,\phi}(t+v)-X^{\epsilon,\phi}(t+u)|>\gamma)dt\\
   \leq & \displaystyle\sup_{t\geq\tau}\mathbb{P}(\sup_{-\tau\leq u\leq v\leq 0\atop v-u\leq\delta}|X^{\epsilon,\phi}(t+v)-X^{\epsilon,\phi}(t+u)|>\gamma)\\
   = & \displaystyle\sup_{t\geq0}\mathbb{P}(\sup_{t\leq u\leq v\leq t+\tau \atop v-u\leq\delta}|X^{\epsilon,\phi}(v)-X^{\epsilon,\phi}(u)|>\gamma)\\
   \leq & \displaystyle\sup_{t\geq0}\mathbb{P}(\sup_{t\leq u\leq v\leq t+\tau\atop v-u\leq\delta}
   \int_{u}^{v}|b(X^{\epsilon,\phi}_{s})|ds>\frac{\gamma}{2})\\
   &+\displaystyle\sup_{t\geq0}\mathbb{P}( \epsilon\sup_{t\leq u\leq v\leq t+\tau\atop v-u\leq\delta}
   |\int_{u}^{v}\sigma(X^{\epsilon,\phi}_{s})dW(s)|>\frac{\gamma}{2}).
\end{align*}
Here $b(\phi):=-B\phi(0)+Ag(\phi)$, which  maps bounded sets
in $\mathcal{C}$ into bounded sets in $\mathbb{R}^{m}$. From this fact and (\ref{esti}), it is easy to yield that
$$ \displaystyle\lim_{\delta\rightarrow0}\sup_{0<\epsilon\leq\epsilon_{0},\phi\in \bar{B}_R(0)}\sup_{t\geq0}
\mathbb{P}(\sup_{t\leq u\leq v\leq t+\tau\atop v-u\leq\delta}\int_{u}^{v}|b(X^{\epsilon,\phi}_{s})|ds>\frac{\gamma}{2})=0.$$
Let $J^{\epsilon}(v):=\epsilon\int_{0}^{v}\sigma(X^{\epsilon,\phi}_{s})dW(s)$, $v\geq 0$. The
continuity of $\sigma$ implies that $\{J^{\epsilon}(v),v\geq 0\}$ is a continuous $m$-dimensional
local martingale. Then by Burkholder-Davis-Gundy inequality, H\"{o}lder inequality, ${\rm(A_{2})}$, $C_{r}$-inequality
and Fubini's theorem, we have for any $t>s\geq 0$
\begin{align*}
\mathbb{E}|J^{\epsilon}(t)-J^{\epsilon}(s)|^{6}
   =& \mathbb{E}|\epsilon\int_{s}^{t}\sigma(X^{\epsilon,\phi}_{r})dW(r)|^{6}\\
   \leq & \widetilde{C}_{6}\Big(2e^{3\tau} \big(R^{6}+\widetilde{M}\big)+1\Big)|t-s|^{3},
\end{align*}
where $\widetilde{C}_{6}=2^{5}\epsilon^{6}_{0}C_{6}(L^{6}+\|\sigma(0)\|_{2}^{6})$ is a constant
independent of $\epsilon$. This means that
there exists some positive constant $c$ such that
$$ \displaystyle\sup_{0<\epsilon\leq\epsilon_{0},\phi\in \bar{B}_R(0)}\sup_{t\geq0}
\mathbb{E}|J^{\epsilon}_{t}(v)-J^{\epsilon}_{t}(u)|^{6}\leq c|v-u|^{3} \ \text{for all} \ u,v\in[0,\tau].  $$
From the Kolmogorov's tightness argument (see,  Karatzas and Shreve \cite[Problem 4.11, p.64]{KaraShre}), we can
deduce that
$$ \displaystyle\lim_{\delta\rightarrow0}\sup_{0<\epsilon\leq\epsilon_{0},\phi\in \bar{B}_R(0)}\sup_{t\geq0}
\mathbb{P}\Big(\displaystyle\sup_{t\leq u\leq v\leq t+\tau\atop v-u\leq\delta}|J^{\epsilon}(v)-J^{\epsilon}(u)|>\frac{\gamma}{2}\Big)=0.$$
In other words,
$$ \displaystyle\lim_{\delta\rightarrow0}\sup_{0<\epsilon\leq\epsilon_{0},\phi\in \bar{B}_R(0)}\sup_{t\geq0}
\mathbb{P}\Big(\epsilon\displaystyle\sup_{t\leq u\leq v\leq t+\tau\atop v-u\leq\delta}
|\int_{u}^{v}\sigma(X^{\epsilon,\phi}_{s})dW(s)|>\frac{\gamma}{2}\Big)=0.$$
Therefore we obtain that
\begin{equation}\label{eqcon}
\displaystyle\lim_{\delta\rightarrow0}\sup_{0<\epsilon\leq\epsilon_{0},\phi\in \bar{B}_R(0)}\mu^{\epsilon,\phi}
        \{\varphi\in \mathcal{C}:\displaystyle\sup_{-\tau\leq u\leq v\leq 0 \atop v-u\leq\delta}|\varphi(v)-\varphi(u)|>\gamma\}=0.
\end{equation}
Consequently, the conclusion follows immediately from (\ref{ubound}) and (\ref{eqcon}).
\end{proof}

\begin{example}[Hopfield Neural Network Models with Noise]
\end{example}
We consider the stochastic delayed {\rm Hopfield} equations
\begin{equation}\label{HopSFDEs}
  dX^{\epsilon}(t)=[-BX^{\epsilon}(t)+Ag(X^{\epsilon}(t-\tau))]dt+\epsilon\sigma(X_{t}^{\epsilon})dW(t)
\end{equation}
where $B={\rm diag}(b_{1},\cdots,b_{m})$, $A=(a_{ij})_{m\times m}$,
$g(x)=(g_{1}(x_{1}),\cdots,g_{m}(x_{m}))^{T}$, and $\sigma(\phi)=(\sigma_{ij}(\phi))_{m\times m}$
is an $m\times m$ matrix defined on $\mathcal{C}$.

Hopfield-type neutral networks have many applications to parallel computation and signal processing involving the solution of optimization problems. It is often required that the network should have a unique stationary solution that is globally attractive. For this purpose, we present the following.

\begin{theorem} Assume ${\rm(A_{1})}-{\rm(A_{3})}$, and
\[{\rm(A_{4})} \ \text{There exists some constant} \ \widehat{M}>0 \ \text{such that} \ |g(x)|\leq \widehat{M} \  \text{ for all} \  \ x\in\mathbb{R}^{m}.\]
If $b$ satisfies condition {\rm(\ref{conb})}, where $b=\displaystyle\min_{1\leq i\leq m}b_{i}$,
then for each $\epsilon\in(0,\epsilon_{0}]$, the system ${\rm(\ref{HopSFDEs})}$  has a unique
invariant measure $\mu^{\epsilon}$ for the segment process \{$X^{\epsilon}_{t}\}_{t\geq 0}$.
Furthermore, $\mu^{\epsilon}$ weakly converges to $\delta_{p}$ as $\epsilon\rightarrow 0$, where $p$ is a globally asymptotically stable equilibrium for differential equations {\rm(\ref{HopSFDEs})} with $\epsilon=0$.
\end{theorem}

\begin{proof} Let $\widetilde{b}(\phi)=-B\phi(0)+Ag(\phi(-\tau))$, $\phi\in\mathcal{C}$. It is easy to see
that $\widetilde{b}$ is globally Lipschitz continuous in $\mathcal{C}$. The existence and uniqueness of invariant measure,
the tightness of the set $\{\mu^{\epsilon},0<\epsilon\leq\epsilon_{0}\}$ follow from Theorem \ref{amos} and
Theorem \ref{unitig}, respectively. The probability convergence condition holds by Corollary \ref{corosfd}.
Combining with assumption ${\rm(A_{4})}$, we get that  the unperturbed system has a unique equilibrium $p$
which is globally asymptotically stable (see \cite[Theorem 2.4]{DriZou}),
where we have used the fact that the operator norm is less than the trace norm. The
final assertion follows from Theorem \ref{mthm}.
\end{proof}

\section{ Appendix: Poincar\'{e} recurrence theorem for continuous dynamical system/semiflow}
In this section we give a full proof of the Poincar\'{e} recurrence theorem for a semiflow (or flow)
on the separable metric space $(M,\rho)$. The original idea is borrowed from Ma\~{n}\'{e} \cite{Ma}
and Hirsch \cite{Hirs}.

Throughout this section we assume that $\Phi:\mathbb{R}_{+}\times M\longrightarrow M$ is a mapping
with the following properties

(i) $\Phi_{\cdot}(x)$ is continuous, for all $x\in M$,

(ii) $\Phi_{t}(\cdot)$ is Borel measurable, for all $t\in\mathbb{R}_{+}$,

(iii) $\Phi_{0}={\rm id}$, $\Phi_{t}\circ\Phi_{s}(x)=\Phi_{t+s}(x)$,
for all $t,s\in \mathbb{R}_{+}$. Here $\circ$ denotes composition of mappings.

Let $x\in M$. Then the \emph{$\omega$-limit set of $x$} is defined by

$\omega(x)=
\{y\in M: \textrm{for~every~neighborhood}~U~\textrm{of}~y,~\textrm{and~for~every}~k\in\mathbb{N},
~\textrm{there~exists}~s\geq k~\textrm{such~that}~\Phi_{s}(x)\in U\}$.

We note that $\Phi_{\cdot}(x)$ is continuous, thus

$\omega(x)=
\{y\in M: \textrm{for~every~neighborhood}~U~\textrm{of}~y, \textrm{and~for~every}~k\in\mathbb{N},
~\textrm{there~exists}~~s\geq k~\textrm{and}~s\in\mathbb{Q}~\textrm{such~that}~\Phi_{s}(x)\in U\}$.

Using the semigroup properties of $\Phi$, it is easy to see that $\omega(x)=\omega(\Phi_{t}(x))$ for
every $t\in\mathbb{R}_{+}$.

Let $(M,\mathcal{B}(M),\mu)$ be a probability space and $\mu$ be a \emph{$\Phi$-invariant
probability measure}, i.e., $\mu\circ\Phi_{t}^{-1}=\mu$ for all $t\in\mathbb{R}_{+}$,
where $\mathcal{B}(M)$ is the $\sigma$-algebra of Borel sets in $M$.

The proof of the following Poincar\'e recurrence theorem follows the line of argument for
the discrete time measurable mapping case (see, e.g., Ma\~{n}\'{e} \cite[p.28-29]{Ma}).

\begin{theorem}
  Let $M$ be a separable metric space and $\mu$ be $\Phi$-invariant.
Then $\mu(B(\Phi))=1$, where $B(\Phi)=\overline{\{x\in M:x\in\omega(x)\}}$
denotes the Birkhoff center of $\Phi$.
\end{theorem}
\begin{proof}For given $t\in\mathbb{R}_{+}$. Let $A$ be an open set in $M$ and
$$A_{0}=\{x\in A: \forall k\in\mathbb{N},\:\exists s\geq k \:
\textrm{and}\:s\in\mathbb{Q} \:\textrm{such\:that} \:\Phi_{s}\circ\Phi_{t}(x)\in A     \}.$$
We claim that $A_{0}\in\mathcal{B}(M)$ and $\mu(A)=\mu(A_{0})$. In fact, for every
$k\in\mathbb{N}$, let $C_{k}=\{x\in A:\Phi_{s}\circ\Phi_{t}(x)\not\in A,\:\forall s\geq k\:\textrm{and}\:s\in\mathbb{Q}\}$.
It is easy to see that $A_{0}=A\setminus\displaystyle\bigcup_{k=1}^{\infty}C_{k}$.
Let $x\in A_{0}$ if and only if $x\in A$ and $\forall k\in\mathbb{N}$, $\exists s\geq k$
and $s\in\mathbb{Q}$ such that $\Phi_{s}\circ\Phi_{t}(x)\in A$ if and only if $x\in A$ and
$x\not\in \displaystyle\bigcup_{k=1}^{\infty}C_{k}$. Note that
$C_{k}=A\setminus\displaystyle\bigcup_{s\geq k,s\in\mathbb{Q}}(\Phi_{s}\circ\Phi_{t})^{-1}(A)$
which shows that $C_{k}\in\mathcal{B}(M)$ and implies that
{\small \begin{equation}\label{countg}
C_{k}\subset \displaystyle\bigcup_{s\geq 0}\Phi_{s}^{-1}(A)
\setminus\displaystyle\bigcup_{s\geq k,s\in\mathbb{Q}}\Phi_{s+t}^{-1}(A)
=\displaystyle\bigcup_{s\geq 0, s\in\mathbb{Q}}\Phi_{s}^{-1}(A)
\setminus\displaystyle\bigcup_{s\geq k,s\in\mathbb{Q}}\Phi_{s+t}^{-1}(A),
\end{equation}}
where we have used the fact that $\Phi_{\cdot}(x)$ is continuous and $A$ is an open set.
From (\ref{countg}) we have that
\begin{displaymath}
    \begin{array}{rl}
      \mu(C_{k})&\leq \mu(\displaystyle\bigcup_{s\geq 0, s\in\mathbb{Q}}\Phi_{s}^{-1}(A))-
     \mu(\displaystyle\bigcup_{s\geq k,s\in\mathbb{Q}}\Phi_{s+t}^{-1}(A))\\
     &\leq \mu(\displaystyle\bigcup_{s\geq 0, s\in\mathbb{Q}}\Phi_{s}^{-1}(A))-
     \mu\circ\Phi_{t}^{-1}(\displaystyle\bigcup_{s\geq k,s\in\mathbb{Q}}\Phi_{s}^{-1}(A))\\
     &=\mu(\displaystyle\bigcup_{s\geq 0, s\in\mathbb{Q}}\Phi_{s}^{-1}(A))-
     \mu\circ\Phi_{k}^{-1}(\displaystyle\bigcup_{s\geq 0,s\in\mathbb{Q}}\Phi_{s}^{-1}(A))\\
     &=0.
    \end{array}
\end{displaymath}
Therefore we get that $A_{0}\in\mathcal{B}(M)$ and $\mu(A_{0})\geq\mu(A)-\displaystyle\sum_{k=1}^{\infty}\mu(C_{k})=\mu(A)$.
This completes the proof of the claim.

Since, furthermore, $M$ is a separable metric space, we can find the countable basis $\{U_{n}\}_{n\in\mathbb{N}}$ of $M$ such that $\displaystyle\lim_{n\rightarrow\infty}{\rm diam}(U_{n})=0$
and $\displaystyle\bigcup_{n=k}^{\infty}U_{n}=M$ for every $k\in\mathbb{N}$. Let
$\widehat{U}_{n}=\{x\in U_{n}:\forall k\in\mathbb{N},\exists s\geq k\: \textrm{and}\:s\in\mathbb{Q}
\:\textrm{such\:that} \:\Phi_{s}\circ\Phi_{t}(x)\in U_{n}\}$ for every $n\in\mathbb{N}$.
From the above claim, we have $\widehat{U}_{n}\in\mathcal{B}(M)$ and $\mu(U_{n}\setminus\widehat{U}_{n})=0$.
Let $\widehat{M}=\displaystyle\limsup_{n\rightarrow\infty}\widehat{U}_{n}
=\displaystyle\bigcap_{k=0}^{\infty}\bigcup_{n=k}^{\infty}\widehat{U}_{n}$,
then we have
\begin{displaymath}
    \begin{array}{rl}
    \mu(M\setminus\widehat{M})&=\mu(\displaystyle\bigcup_{k=0}^{\infty}
    (M\setminus\bigcup_{n=k}^{\infty}\widehat{U}_{n} ))\\
    &=\mu(\displaystyle\bigcup_{k=0}^{\infty}
    (\bigcup_{n=k}^{\infty}U_{n}\setminus\bigcup_{n=k}^{\infty}\widehat{U}_{n} ))\\
    &\leq\mu(\displaystyle\bigcup_{k=0}^{\infty}
    \bigcup_{n=k}^{\infty}(U_{n}\setminus\widehat{U}_{n} ))\\
    &=0.
    \end{array}
\end{displaymath}
This means that $\mu(\widehat{M})=1$. Due to this fact it is sufficient
to prove that $\widehat{M}\subset\{x\in M:x\in\omega(x)\}$ which we now prove.
Let $x\in\widehat{M}$. For any $r>0$, since $\displaystyle\lim_{n\rightarrow\infty}{\rm diam}(U_{n})=0$,
$\exists N\in\mathbb{N}$ such that $\forall n\geq N$ ${\rm diam}(U_{n})<\frac{r}{3}$.
Note that $x\in\displaystyle\bigcup_{n=N}^{\infty}\widehat{U}_{n}$.
Hence there exists $n\geq N$ such that $x\in\widehat{U}_{n}\subset U_{n}$, it follows
that $U_{n}\subset B(x,r)$. This implies that $\forall k\in\mathbb{N}$,
$\exists s\geq k$ and $s\in\mathbb{Q}$ such that $\Phi_{s}\circ\Phi_{t}(x)\in U_{n}\subset B(x,r)$,
i.e., $x\in\omega(\Phi_{t}(x))=\omega(x)$.
\end{proof}

This theorem immediately implies the following assertion.

\begin{remark}
 Let ${\rm supp}(\mu)$ denote the support of $\mu$, where $\mu$ is $\Phi$-invariant.
Then ${\rm supp}(\mu)\subset B(\Phi)$.
\end{remark}

If we additionally assume that $\Phi_{t}(\cdot)$ is continuous for every $t\in\mathbb{R}_{+}$,
this is, $\Phi$ is a \emph{semiflow}. (If we can replace $\mathbb{R}_{+}$ by $\mathbb{R}$, then
$\Phi$ defines a \emph{flow}.)
Then we can prove the following assertion.

\begin{proposition}
${\rm supp}(\mu)$ is forward invariant.
If, in addition, ${\rm supp}(\mu)$ is a compact set and $\Phi_{t}(\cdot):M\rightarrow M$ is
an injective mapping (or homeomorphism), then ${\rm supp}(\mu)$ is invariant.
\end{proposition}

\begin{proof} Let $H:={\rm supp}(\mu)$. The continuity of $\Phi_{t}(\cdot)$ implies
that $\Phi_{t}^{-1}(H)$
is a closed set. By the invariance of $\mu$,
$\mu(\Phi_{t}^{-1}(H))=\mu(H)=1$. This implies that
$H\subset\Phi_{t}^{-1}(H)$. Therefore
$\Phi_{t}(H)\subset H$.

The injectivity of $\Phi_{t}$ implies that $\Phi_{t}^{-1}(\Phi_{t}(H))
=H$. Thus
$$ 1=\mu(H)= \mu(\Phi_{t}^{-1}(\Phi_{t}(H)))=\mu(\Phi_{t}(H)).$$
Note that $\Phi_{t}(H)$ is a closed set (more precisely, a compact set)
from the fact that $H$ is a compact set and
$\Phi_{t}(\cdot)$ is continuous.
Therefore we have $H\subset\Phi_{t}(H)$.
\end{proof}

Following we assume that $H:={\rm supp}(\mu)$ is a closed invariant set.
Let $\Phi|_{H}$ denote the restricted semiflow.
Then by the Poincar\'{e} recurrence theorem, we obviously have

\begin{corollary} $H= B(\Phi|_{H})$.
This means that every point of $H$ is recurrent for $\Phi|_{H}$.
\end{corollary}

\begin{proof} The fact that $H\subset B(\Phi|_{H})\subset H$ follows directly
from the Poincar\'{e} recurrence theorem and the definition of
the support of invariant measure $\mu$.
\end{proof}

If, in addition, $H$ is a compact set.
We refer to Hirsch \cite{Hirs} for additional properties of the support of $\mu$ in
the case of $\mu$ is ergodic.
Recall that
$\Phi$-invariant probability measure $\mu$ is said to be \emph{ergodic} if
for any $A\in\mathcal{B}(M)$ with the property $\Phi_{t}(A)=A$ for all $t\in\mathbb{R}_{+}$,
we have either $\mu(A)=0$ or $\mu(A)=1$.

A subset $A\subset M$ is an \emph{attractor} if $A$ is compact and invariant $(\Phi_{t}(A)=A)$
and contained in an open set $N\subset M$ such that
$$  \displaystyle\lim_{t\rightarrow\infty}{\rm dist}(\Phi_{t}(x),A)=0 \:\:\textrm{uniformly\:in}\: x\in N.   $$
Furthermore, if there is an attractor that contains all $\omega$-limit points, then we call
$\Phi$ is \emph{dissipative}.

A nonempty compact invariant set $A\subset M$ is called \emph{attractor-free}
if the restricted flow $\Phi|_{A}$ has no attractor other than $A$ itself.
By a result of Conley \cite{Conl} $A$ is attractor-free this
is the equivalent to $A$ is connected and every point of $A$ is chain recurrent for $\Phi|_{A}$.
The definition of chain recurrent set we refer the reader to \cite{Bena} since this notion will not
be used here. Meanwhile, a detailed discussion about this relation we refer to
Bena\"{i}m \cite[p.23]{Bena1}.

Furthermore, for this compact invariant set $H:={\rm supp}(\mu)$, the
following result is proved in Bena\"{i}m and Hirsch \cite{BenaHirs},
Hirsch \cite{Hirs}, respectively.

\begin{proposition}
Each component of $H$ is attractor-free.
In addition, if $\mu$ is ergodic, then $H$ itself is attractor-free.
\end{proposition}

If $\Phi$ is a strongly monotone semiflow in an ordered Banach space $M$,
let $A$ denote attractor-free, then by more detailed structural analysis of $A$,
Hirsch \cite{Hirs} points out that either $A$ is unordered, or else $A$ is contained in totally ordered,
compact arc of equilibria.

\end{document}